\documentclass[12pt]{amsart}
\usepackage{amsmath,amsthm,amsfonts,amscd,amssymb,eucal,latexsym,mathrsfs, appendix, yhmath, setspace}
\usepackage[numbers,sort&compress]{natbib}
\usepackage[all,cmtip]{xy}
\usepackage{enumerate}

\newtheorem{theorem}{Theorem}[section]
\newtheorem{corollary}[theorem]{Corollary}
\newtheorem{lemma}[theorem]{Lemma}
\newtheorem{proposition}[theorem]{Proposition}

\theoremstyle{definition}
\newtheorem{definition}[theorem]{Definition}
\newtheorem{remark}[theorem]{Remark}
\newtheorem{notation}[theorem]{Notation}
\newtheorem{example}[theorem]{Example}

\newcommand{\im}{{\rm im}}

\newcommand{\id}{{\rm id}}

\newcommand{\mdim}{{\rm mdim}}
\newcommand{\ord}{{\rm ord}}

\newcommand{\cC}{{\mathcal C}}
\newcommand{\cD}{{\mathcal D}}

\newcommand{\cF}{{\mathcal F}}
\newcommand{\cH}{{\mathcal H}}

\newcommand{\cJ}{{\mathcal J}}

\newcommand{\cL}{{\mathcal L}}
\newcommand{\cM}{{\mathcal M}}

\newcommand{\cR}{{\mathcal R}}

\newcommand{\cU}{{\mathcal U}}
\newcommand{\cV}{{\mathcal V}}
\newcommand{\cW}{{\mathcal W}}

\newcommand{\Cb}{{\mathbb C}}

\newcommand{\Fb}{{\mathbb F}}
\newcommand{\Mb}{{\mathbb M}}
\newcommand{\Nb}{{\mathbb N}}
\newcommand{\Pb}{{\mathbb P}}
\newcommand{\Qb}{{\mathbb Q}}
\newcommand{\Rb}{{\mathbb R}}
\newcommand{\Zb}{{\mathbb Z}}

\newcommand{\sA}{{\mathscr A}}
\newcommand{\sB}{{\mathscr B}}

\newcommand{\sD}{{\mathscr D}}

\newcommand{\sF}{{\mathscr F}}

\newcommand{\sL}{{\mathscr L}}
\newcommand{\sM}{{\mathscr M}}
\newcommand{\sN}{{\mathscr N}}

\newcommand{\fA}{{\mathfrak A}}

\newcommand{\fC}{{\mathfrak C}}
\newcommand{\fF}{{\mathfrak F}}
\newcommand{\fM}{{\mathfrak M}}
\newcommand{\fN}{{\mathfrak N}}

\newcommand{\rL}{{\rm L}}

\newcommand{\diam}{{\rm diam}}

\newcommand{\tr}{{\rm tr}}

\newcommand{\op}{{\rm op}}

\newcommand{\rM}{{\rm M}}

\newcommand{\End}{{\rm End}}

\newcommand{\Sym}{{\rm Sym}}
\newcommand{\Map}{{\rm Map}}
\newcommand{\mesh}{{\rm mesh}}

\newcommand{\rspan}{{\rm span}}
\newcommand{\Hom}{{\rm Hom}}

\newcommand{\mL}{{\rm mL}}
\newcommand{\rk}{{\rm rk}}
\newcommand{\mrk}{{\rm mrk}}
\newcommand{\vrk}{{\rm vrk}}

\allowdisplaybreaks

\begin{document}

\title{Sofic Mean Length}

\author{Hanfeng Li}
\author{Bingbing Liang}

\address{\hskip-\parindent
H.L., Center of Mathematics, Chongqing University,
Chongqing 401331, China. \\
Department of Mathematics, SUNY at Buffalo,
Buffalo, NY 14260-2900, U.S.A.}
\email{hfli@math.buffalo.edu}

\address{\hskip-\parindent
B.L., Institute of Mathematics, Polish Academy of Sciences, ul. \'{S}niadeckich 8, 00-656 Warszawa, Poland}
\email{bliang@impan.pl}


\date{January 12, 2019}

\begin{abstract}
Given a length function $\rL$ on the $R$-modules of a unital ring $R$, for each sofic group $\Gamma$ we define a mean length for every locally $\rL$-finite $R\Gamma$-module relative to a bigger $R\Gamma$-module. We establish an addition formula for the mean length.

We give two applications. The first one shows that for any unital left Noetherian ring $R$, $R\Gamma$ is stably direct finite. The second one shows that for any $\Zb\Gamma$-module $\cM$, the mean topological dimension of the induced $\Gamma$-action on the Pontryagin dual of $\cM$ coincides with the von Neumann-L\"{u}ck rank of $\cM$.
\end{abstract}

\maketitle

\tableofcontents

\section{Introduction} \label{S-introduction}

The notion of length function was introduced by Northcott and Reufel in \cite{NR}. Given a unital ring $R$, a {\it length function} $\rL$ on left $R$-modules assigns a numerical isomorphism invariant $\rL(\sM)\in \Rb_{\ge 0}\cup \{+\infty\}$ to every left $R$-module $\sM$ satisfying suitable conditions (see Definition~\ref{D-length} below for details). The major requirement on $\rL$ is that for any left $R$-modules $\sM_1\subseteq \sM_2$, one has the addition formula
\begin{align*}
\rL(\sM_2)=\rL(\sM_1)+\rL(\sM_2/\sM_1).
\end{align*}
In particular, it implies that $\rL$ is increasing in the sense that submodules have smaller $\rL$-length. Length functions have been studied in \cite{Vamos68, Zanardo}, especially in detail in the thesis of V\'{a}mos \cite{VamosT}.

Given a discrete group $\Gamma$, one can form the group ring $R\Gamma$ of $\Gamma$ with coefficients in $R$ (see Section~\ref{SS-group ring} below) \cite{Passman77}. The left $R\Gamma$-modules are exactly left $R$-modules equipped with a $\Gamma$-action by automorphisms. When $\rL$ is a length function on left $R$-modules with $\rL(R)=1$ and $\Gamma$ is infinite, the left $R\Gamma$-modules, e.g. $R\Gamma$ itself, are typically large as left $R$-modules and could easily
have infinite $\rL$-length. Thus it is desirable to define a length function $\mL$ on left $R\Gamma$-modules which takes into account the $\Gamma$-action. This can be thought of as developing an equivariant version of the $\rL$-invariant. More precisely, the question is whether there is a length function $\mL$ on the left $R\Gamma$-modules satisfying
$\mL(R\Gamma)=1$, or more generally, $\mL(R\Gamma\otimes_R\sM)=\rL(\sM)$ for every left $R$-module $\sM$.

This question has been studied and answered affirmatively for amenable groups $\Gamma$ (see Section~\ref{S-amenable mean length} for details). It was studied by Salce-Zanardo and Salce-V\'{a}mos-Virili in \cite{SZ, SVV} for $\Gamma=\Zb$ in the case $\rL$ is discrete in the sense that the set of finite values of $\rL$ is order isomorphic to $\Nb$, and by Salce-Virili in \cite{SV} for $\Gamma=\Zb^d$, with motivation from the entropy theory for algebraic actions. Elek considered finitely generated left $R\Gamma$-modules for any field $R$ and any amenable group $\Gamma$ \cite{Elek03}. We studied the case $\rL(R)<+\infty$ for amenable groups $\Gamma$ in \cite{LL}, and used $\mL$ in the special case $R=\Zb$ and $\rL(\sM)=\dim_\Qb(\Qb\otimes_\Zb \sM)$ to establish a correspondence between the mean dimension for algebraic actions of amenable groups and  the von Neumann-L\"{u}ck rank of left $\Zb\Gamma$-modules in the theory of $L^2$-invariants. Virili considered discrete $\rL$ for amenable groups in \cite{ViriliA}, and used $\mL$ to show that $R\Gamma$ for left Noetherian $R$ and amenable $\Gamma$ is stably direct finite. It turns out that when $\rL(R)=+\infty$, the suitable domain for $\mL$ is the class of left $R\Gamma$-modules $\cM$ which are {\it locally $\rL$-finite} in the sense that every finitely generated $R$-submodule of $\cM$ has finite $\rL$-length.

For nonamenable groups, the above question is hopeless, since the answer is already negative for the free group $\Fb_2$ with $2$ generators. Indeed, it is well known \cite{Passman77} that,  as a left $R\Fb_2$-module, $R\Fb_2$ contains $R\Fb_2\oplus R\Fb_2$ as a submodule, violating the monotonicity  of $\mL$ (see Example~\ref{E-free group} below for more on this example). Interestingly enough, if we take $R=\Zb/2\Zb$, then the embedding $R\Fb_2\oplus R\Fb_2\hookrightarrow R\Fb_2$ of left $R\Fb_2$-modules yields a factor map from the algebraic action of $\Fb_2$ associated to $R\Fb_2$ to the algebraic action of $\Fb_2$ associated to $R\Fb_2\oplus R\Fb_2$, which is
exactly the Ornstein-Weiss example \cite[page 138]{OW} of a factor map from
the Bernoulli shift over $\Fb_2$ with $2$ symbols to the Bernoulli shift over $\Fb_2$ with $4$ symbols (this example had  been an obstruction to developing an entropy theory for nonamenable group actions for many years until the work of Bowen in \cite{Bowen10}).

The class of sofic groups was introduced by Gromov \cite{GromovS} and Weiss \cite{Weiss}. It contains all discrete amenable groups and residually finite groups, especially the free groups. So far it is unknown whether non-sofic groups exist. In the last several years, the entropy and mean dimension theory for amenable group actions has been extended to actions of sofic groups \cite{Bowen10, Bowen14, KL11, KL13a, Kerr, Li}, pioneered by the work of Bowen. Thus one might hope to define $\mL$ for sofic groups. However, in all these works, of crucial use is the a priori given metric structure that allows  one to talk about approximations and then use the size of the space of approximations to define the invariant. In the measure-preserving action situation, the metric structure is that on the measurable subsets coming from taking the measure of the symmetric difference, while in the topological action situation, the metric structure is that of the underlying compact space. Thus it is not clear how one could proceed in the purely algebraic setting when handling length functions.

In this paper, we find a satisfactory way to define mean length $\mL$ for sofic groups. There are two main ingredients in our approach. The first is that we find a way to define invariants for sofic group actions in a purely algebraic setting. The second is that, we define relative invariants for a pair of objects. That is, for any left $R\Gamma$-modules $\cM_1\subseteq \cM_2$, under the mild condition that $\cM_1$ is locally $\rL$-finite, we define the {\it mean length of $\cM_1$ relative to $\cM_2$}, denoted by $\mL_{\Sigma, \omega}(\cM_1|\cM_2)$ (Definition~\ref{D-mean length}). Here $\Sigma$ is a fixed sofic approximation net for $\Gamma$ and $\omega$ is a fixed free ultrafilter. For any left $R\Gamma$-module $\cM$ which is locally $\rL$-finite, we then define its mean length $\mL_{\Sigma, \omega}(\cM)$ as $\mL_{\Sigma, \omega}(\cM|\cM)$. The relative invariants are completely a sofic phenomenon, since $\mL_{\Sigma, \omega}(\cM_1|\cM_2)$ does not depend on $\cM_2$ when $\Gamma$ is amenable and $\cM_2$ is locally $\rL$-finite (Theorem~\ref{T-amenable mean length}). One advantage of introducing invariants for two variables ($\cM_1$ and $\cM_2$) is that the monotonicity is preserved for each variable: $\mL_{\Sigma, \omega}(\cM_1|\cM_2)$ increases with $\cM_1$ and decreases with $\cM_2$ (Proposition~\ref{P-continuity}). Most importantly, using these relative invariants we are able to establish a modified addition formula.

\begin{theorem} \label{T-addition for mean length}
Let $\cM_1\subseteq \cM_2\subseteq \cM_3$ be locally $\rL$-finite left $R\Gamma$-modules. Then
$$ \mL_{\Sigma, \omega}(\cM_2|\cM_3)=\mL_{\Sigma, \omega}(\cM_1|\cM_3)+\mL_{\Sigma, \omega}(\cM_2/\cM_1|\cM_3/\cM_1).$$
In particular, for any locally $\rL$-finite left $R\Gamma$-modules $\cM_1\subseteq \cM_2$, taking $\cM_3=\cM_2$ one has
$$ \mL_{\Sigma, \omega}(\cM_2)=\mL_{\Sigma, \omega}(\cM_1|\cM_2)+\mL_{\Sigma, \omega}(\cM_2/\cM_1).$$
\end{theorem}

We give two applications of the sofic mean length.

The first application concerns the stably direct finiteness of $R\Gamma$. A unital ring $\tilde{R}$ is called {\it directly finite} or {\it von Neumann finite} if for any $a, b\in \tilde{R}$ with $ab=1$, one has $ba=1$. It is called {\it stably direct finite} if $M_n(\tilde{R})$ is directly finite for every $n\in \Nb$.
Kaplansky's {\it direct finiteness conjecture} asserts that for any field $R$ and any  group $\Gamma$, $R\Gamma$ is directly finite \cite[page 123]{Kaplansky}. Kaplansky proved that $R\Gamma$ is stably direct finite when $R$ is a field with  characteristic $0$ and $\Gamma$ is any group \cite[page 122]{Kaplansky} (see also \cite{Montgomery, Passman71} and \cite[Corollary 2.1.9]{Passman77}). Using the sofic mean length, we are able to prove the following result.

\begin{theorem} \label{T-Noetherian}
Let $R$ be a unital left Noetherian ring and  $\Gamma$ be a sofic group. Then $R\Gamma$ is stably direct finite.
\end{theorem}

Theorem~\ref{T-Noetherian} proves a conjecture of Virili in \cite{ViriliS}. It was proved in the cases when $R$ is a skew-field and $\Gamma$ is residually amenable by Ara et al. \cite[Remark 3.5.(iii)]{AOP}, when $R$ is a skew-field and $\Gamma$ is sofic by Elek and Szab\'{o} \cite[Corollary 4.7]{ES04}, when $R$ is an Artinian ring and $\Gamma$ is sofic by Ceccherini-Silberstein and Coornaert \cite[Corollary 1.5]{CC07},
and when $R$ is a unital left Noetherian ring and $\Gamma$ is amenable by Virili \cite[Theorem A]{ViriliA}.

Just as commutative Noetherian rings play a vital role in commutative algebra, left Noetherian rings are at the heart of the theory of noncommutative rings \cite{GW, MR}. When $R$ is a unital subring of the direct product $\prod_{j\in J}R_j$ for a family of rings $\{R_j\}_{j\in J}$, $R\Gamma$ is a unital subring of $\prod_{j\in J} R_j\Gamma$, and hence the direct finiteness of $R\Gamma$ follows from that of $R_j\Gamma$. For unital commutative rings $R$, it is then natural to try to obtain the direct finiteness of $R\Gamma$ by embedding $R$ into the direct product of fields. However, this is not always possible. Indeed, it is easy to show that a unital commutative ring $R$ embeds into the direct product of fields exactly when the nilradical of $R$, consisting of all nilpotent elements, is trivial. For the noncommutative case, Goldie's theorem implies that a semiprime left Noetherian ring has a semisimple (and hence Artinian) left quotient ring  \cite{Goldie} \cite[Theorem 2.3.6]{MR} \cite[Corollary 6.16]{GW}. Beyond that, it is not clear when one can embed a left Noetherian ring into the direct product of Artinian rings. For example, for any finite-dimensional Lie algebra $\mathfrak{g}$ over a field, since the enveloping algebra $U(\mathfrak{g})$ is left Noetherian \cite[Corollary 1.7.4]{MR}, for any two-sided ideal $I$ of $U(\mathfrak{g})$ which is not the intersection of prime ideals, we do not know whether the left Noetherian ring $U(\mathfrak{g})/I$ embeds into the direct product of Artinian rings. Here a two-sided ideal $I'$ of $R$ is called prime if for any $a, b\in R\setminus I'$ one has $aRb\nsubseteq I'$. Thus Theorem~\ref{T-Noetherian} enables us to establish the stable direct finiteness of $R\Gamma$ for many cases in which one cannot reduce to the results of Elek-Szab\'{o} and Ceccherini-Silberstein-Coornaert.

Our second application concerns the correspondence between the mean dimension in dynamical systems and the von Neumann-L\"{u}ck rank in the theory of $L^2$-invariants. Mean topological dimension was introduced by Gromov \cite{GromovM}, and developed systematically by Lindenstrauss and Weiss \cite{LW} for continuous actions of countable amenable groups on compact metrizable spaces, as an equivariant version of the covering dimension for compact metrizable spaces. Lindenstrauss and Weiss also introduced the metric mean dimension in the same setting, as an equivariant version of the infimum of the Minkowski dimensions over the compatible metrics on a compact metrizable space. Later these mean dimensions were extended to continuous actions of countable sofic groups on compact metrizable spaces \cite{Li}.
The mean topological dimension has applications to the problem of embedding one dynamical system into another \cite{LW} and the problem of finding factors with small entropy \cite{Lindenstrauss}. The mean dimensions have attracted much attention in the last several years \cite{Coornaert, CK, Gutman11, Gutman15, Gutman16, GLT, GT, Krieger06, Krieger09, LTs, MT11, MT15, Tsukamoto08, Tsukamoto09D, Tsukamoto09G, Tsukamoto15}.

The von Neumann-L\"{u}ck dimension was originally defined for finitely generated projective left modules over the group von Neumann algebra $\cL\Gamma$ for a discrete group $\Gamma$, and later extended to arbitrary left $\cL\Gamma$-modules by L\"{u}ck \cite{Luck98} in order to extend Atiyah's $L^2$-Betti numbers \cite{Atiyah} to arbitrary continuous $\Gamma$-actions. In our terminology, it is a length function on the left $\cL\Gamma$-modules. It has profound applications to the theory of $L^2$-invariants \cite{Luck02}. Via taking a tensor product with $\cL\Gamma$, one can define the von Neumann-L\"{u}ck rank for any left $\Zb\Gamma$-module $\cM$, which is in fact the $0$-th $L^2$-Betti number of $\cM$.

Despite the fact that the mean dimensions are dynamical invariants while the von Neumann-L\"{u}ck rank is an analytic invariant,  using the mean length we shall show that they correspond to each other in the setting of algebraic actions of sofic groups. Since the relative invariants are crucial for the sofic mean length, in order to establish such a correspondence, we must introduce relative invariants for the mean dimensions and the von Neumann-L\"{u}ck rank. For any countable sofic group $\Gamma$ acting on compact metrizable spaces $X$ and $Y$ and any factor map $X\rightarrow Y$, we define the mean topological dimension $\mdim_{\Sigma, \omega}(Y|X)$ of $\Gamma\curvearrowright Y$ relative to the extension $\Gamma\curvearrowright X$ (Definition~\ref{D-relative mdim}) and  the metric mean  dimension $\mdim_{\Sigma, \omega, \rM}(Y, \rho|X)$ of $\Gamma\curvearrowright Y$ relative to $\Gamma\curvearrowright X$ with respect to a compatible metric $\rho$ on $Y$ (Definition~\ref{D-relative metric mdim}). The
metric mean  dimension $\mdim_{\Sigma, \omega, \rM}(Y|X)$ of $\Gamma\curvearrowright Y$ relative to $\Gamma\curvearrowright X$ is defined as the infimum of $\mdim_{\Sigma, \omega, \rM}(Y, \rho|X)$ for $\rho$ ranging over all compatible metrics on $Y$. For any discrete group $\Gamma$ and any left $\Zb\Gamma$-modules $\cM_1\subseteq \cM_2$, we define the von Neumann-L\"{u}ck rank $\vrk(\cM_1|\cM_2)$ of $\cM_1$ relative to $\cM_2$ (Definition~\ref{D-vrk}).
In the absolute case $X=Y$ or $\cM=\cM_1=\cM_2$, we write the invariants as $\mdim_{\Sigma, \omega}(X)$, $\mdim_{\Sigma, \omega, \rM}(X)$  and $\vrk(\cM)$ respectively.

Note that for any countable $\Zb\Gamma$-module $\cM$, its Pontryagin dual $\widehat{\cM}$, consisting of all group homomorphisms from $\cM$ to the circle group $\Rb/\Zb$, is a compact metrizable abelian group. The $\Zb\Gamma$-module structure of $\cM$ induces a natural $\Gamma$-action on $\widehat{\cM}$ by continuous automorphisms.
The information about the $\Zb\Gamma$-module $\cM$ is equivalent to the information about the
$\Gamma$-action on the compact metrizable abelian group $\widehat{\cM}$. Thus, if we forget the group structure of $\widehat{\cM}$,
a priori there is no reason that the $\Gamma$-action on the compact metrizable space $\widehat{\cM}$ should remember any information about the $\Zb\Gamma$-module $\cM$.
Our second application in the following theorem says that indeed the mean dimensions of the $\Gamma$-action on the compact metrizable space $\widehat{\cM}$ are exactly  the von Neumann-L\"{u}ck rank of the $\Zb\Gamma$-module $\cM$.

\begin{theorem} \label{T-mdim vs vrk}
Let $\Gamma$ be a countable sofic group and $\cM_1\subseteq \cM_2$ be countable $\Zb\Gamma$-modules. Then
$$\mdim_{\Sigma, \omega, \rM}(\widehat{\cM_1}|\widehat{\cM_2})=\mdim_{\Sigma, \omega}(\widehat{\cM_1}|\widehat{\cM_2})=\vrk(\cM_1|\cM_2).$$
Furthermore, there exists a translation-invariant compatible  metric $\rho$ on $\widehat{\cM_1}$ with
\begin{align*}
\mdim_{\Sigma, \omega, \rM}(\widehat{\cM_1}, \rho|\widehat{\cM_2})=\mdim_{\Sigma, \omega}(\widehat{\cM_1}|\widehat{\cM_2}).
\end{align*}
\end{theorem}

The first and the last equalities  of Theorem~\ref{T-mdim vs vrk} are the algebraic action case of the dynamical analogue of the Pontryagin-Schnirelmann theorem \cite{PS} which says that for any compact metrizable space $X$, its covering dimension is the minimal value of the Minkowski dimension of $(X, \rho)$ for $\rho$ ranging over all compatible metrics on $X$.
The amenable group case of Theorem~\ref{T-mdim vs vrk} was proved in \cite{LL}.
Hayes proved \cite{Hayes} that $\mdim_{\Sigma, \omega, \rM}(\widehat{\cM})=\vrk(\cM)$
for countable sofic groups $\Gamma$ and finitely generated $\Zb\Gamma$-modules $\cM$, and that $\mdim_{\Sigma, \omega}(\widehat{\cM})=\vrk(\cM)$
for residually finite groups $\Gamma$ and finitely presented $\Zb\Gamma$-modules $\cM$ when the sofic approximation sequence $\Sigma$ of $\Gamma$ comes from finite quotient groups of $\Gamma$.

The correspondence between mean dimensions and von Neumann-L\"{u}ck rank in Theorem~\ref{T-mdim vs vrk} is parallel to the correspondence between entropy and $L^2$-torsion in \cite[Theorem 1.1]{LT}, though the latter is known only for amenable groups so far.

The ideas in this paper can be applied to many other problems such as dimension for isometric actions on Banach spaces and addition formulas for sofic entropy, which we plan to investigate in subsequent papers.

This paper is organized as follows. We recall some basic definitions and results in Section~\ref{S-preliminary}. In Section~\ref{S-mean length} we introduce the sofic mean length and establish some basic properties including Theorem~\ref{T-addition for mean length}. Theorem~\ref{T-Noetherian} is proved in Section~\ref{S-direct finite}. In Section~\ref{S-amenable mean length} we show that for amenable groups the sofic mean length coincides with the mean length introduced in \cite{LL, ViriliA}. In Section~\ref{S-fg submodule} we give a formula for the mean length of a finitely  generated $R\Gamma$-module relative to a free $R\Gamma$-module. We prove the equality of the mean length and the von Neumann-L\"{u}ck rank when $R$ is a unital subring of $\bar{\Qb}$ in Section~\ref{S-mrk vs vrk}. The relative mean topological dimension and the relative metric mean dimension are introduced in Sections~\ref{S-relative mdim} and \ref{S-relative metric mdim} respectively. Theorem~\ref{T-mdim vs vrk} is proved in Section~\ref{S-mdim vs mrk}. In Section~\ref{S-application} we give three applications to mean dimension.

Throughout this paper, $R$ will be a unital ring, and $\Gamma$ will be a discrete group. All modules are left modules unless specified. For any $d\in \Nb$, we write $[d]$ for the set $\{1, \dots, d\}$ and $\Sym(d)$ for the permutation group of $[d]$. We denote by $\cF(\Gamma)$ the set of all finite subsets of $\Gamma$.

\noindent{\it Acknowledgements.}
H. Li was partially supported by NSF and NSFC grants. He is grateful to Alain Connes and Henri Moscovici for comments at the noncommutative geometry workshop in June 2015 at MFO.
We thank Simone Virili and an anonymous referee for helpful comments.

\section{Preliminaries} \label{S-preliminary}

\subsection{Length functions} \label{SS-length}

In this section we give some examples of length functions, following the thesis of V\'{a}mos \cite{VamosT}.

\begin{definition} \label{D-length}
By {\it a length function} $\rL$ on  $R$-modules we mean associating a value $\rL(\sM)\in \Rb_{\ge 0}\cup \{+\infty\}$ to each $R$-module $\sM$ such that the following conditions are satisfied:
\begin{enumerate}
\item $\rL(0)=0$;
\item (additivity) for any short exact sequence $0\rightarrow \sM_1\rightarrow \sM_2\rightarrow \sM_3\rightarrow 0$ of $R$-modules, one has $\rL(\sM_2)=\rL(\sM_1)+\rL(\sM_3)$;
\item (upper continuity) for any $R$-module $\sM$, one has $\rL(\sM)=\sup_\sN \rL(\sN)$ for $\sN$ ranging over all finitely generated submodules of $\sM$.
\end{enumerate}
\end{definition}

By {\it a chain} $\tau$ for an $R$-module $\sM$, we mean a finite sequence of submodules of $\sM$ of the form $0=\sM_0\subseteq \sM_1\subseteq \cdots \subseteq \sM_n=\sM$.
The modules $C_j=\sM_j/\sM_{j-1}$ for $1\le j\le n$ are called the {\it chain factors} of $\tau$.
One chain $\tau$ {\it refines} another chain $\tau': 0=\sM'_0\subseteq \sM'_1\subseteq \cdots \subseteq \sM'_m=\sM$ if $\tau$ is obtained from $\tau'$ by inserting more submodules. The two chains $\tau$ and $\tau'$ are called {\it equivalent} if $m=n$ and they have the same chain factors up to permutation, i.e. there is some $\sigma \in \Sym(n)$ such that $C_j\cong C'_{\sigma(j)}$ for all $1\le j\le n$, where $C'_j$ for $1\le j\le m$ are the chain factors of $\tau'$. The Jordan-H\"{o}lder theorem \cite[VIII.1.10]{Hungerford} states that any two chains for $\sM$ have equivalent refinements.

If $\sM_1\subseteq \sM_2$ are submodules of an $R$-module $\sM$, then $\sM_2/\sM_1$ is called a {\it segment} of $\sM$.

\begin{example} \label{E-canonical length}
For each $R$-module $\sM$ and each chain $\tau: 0=\sM_0\subseteq \sM_1\subseteq \cdots \subseteq \sM_n=\sM$ for $\sM$ with chain factors $\{C_j\}_{1\le j\le n}$, denote by $l(\tau)$ the number of $1\le j\le n$ for which $C_j$ is a simple module,
and by $l'(\tau)$ the number of $1\le j\le n$ for which $C_j\neq 0$.
Denote by $\rL(\sM)$ the supremum of $l(\tau)$ for $\tau$ ranging over all chains for $\sM$.
 Since every nonzero $R$-module has a segment which is a simple module, it is easily checked that $\rL(\sM)$ is equal to the supremum of $l'(\tau)$ for $\tau$ ranging over all chains for $\sM$.
Using the Jordan-H\"{o}lder theorem it is easy to see that $\rL$ satisfies the conditions (1) and (2) in Definition~\ref{D-length}. Let $\tau$ be a chain for $\sM$ as above, and denote by $J$ the set of  $1\le j\le n$ such that $C_j$ is a simple module. For each $j\in J$, take a finitely generated submodule $\sN_j$ of $\sM_j$ such that $\sN_j$ maps onto $C_j$ under the quotient homomorphism $\sM_j\rightarrow C_j$. Then $\sN:=\sum_{j\in J}\sN_j$ is a finitely generated submodule of $\sM$. For each $j\in J$, set $\sN'_j=\sum_{i\le j, i\in J}\sN_i$. Then $\{\sN'_j\}_{j\in J}$ together with $\sN'_0=0$ is a chain $\tau'$ for $\sN$, and $\sM_j/\sM_{j-1}$ is a quotient module of
$\sN'_j/\sN'_k$ for each $j\in J$, where $k$ is the largest element in $J\cup \{0\}$ strictly less than $j$. Thus $\sN$ has a chain $\tau''$ refining $\tau'$ such that the modules $C_j$ for $j\in J$ appear in the chain factors of $\tau''$. Therefore $l(\tau'')\ge l(\tau)$. Thus $\rL$ satisfies the condition (3) of Definition~\ref{D-length}, and hence
is a length function on $R$-modules.
Furthermore, $\rL$ is the unique length function $\rL'$ on $R$-modules satisfying $\rL'(\sM)=1$ for every simple $R$-module $\sM$ \cite[Proposition 2.12]{VamosT}.
\end{example}

\begin{example} \label{E-Ore}
Let $R$ be a {\it left Ore domain} \cite[10.19]{Lam}, i.e. $ab\neq 0$ and $Ra\cap Rb\neq \{0\}$ for all nonzero $a, b\in R$.
Then $R$ has a {\it left ring of fractions} $Q$ \cite[Theorem 10.6]{Lam}, uniquely determined up to isomorphism by the following properties:
\begin{enumerate}
\item $R$ is a subring of $Q$,

\item every nonzero element of $R$ is invertible in $Q$,

\item every element of $Q$ is of the form $a^{-1}b$ for some $b\in R$ and nonzero $a\in R$.
\end{enumerate}
Note that $Q$ is a skew-field, thus every left $Q$-module $\sM'$ has its dimension $\dim_Q(\sM')$, which is also the same as the length of $\sM'$ defined in Example~\ref{E-canonical length}. Furthermore, $Q$ is flat as a right $R$-module \cite[Exercise 10.17]{Lam}. It follows that the function $\rk$ on left $R$-modules defined
by
$$\rk(\sM):=\dim_Q(Q\otimes_R\sM)$$
is a length function \cite[2.4.II]{VamosT}. It is the unique length function $\rL$ on $R$-modules satisfying $\rL(R)=1$ \cite[Proposition 2.13]{VamosT}.
\end{example}

For the proof of Theorem~\ref{T-hopfian} below, we need a length function $\rL_\alpha$ on $R$-modules with suitable properties for each ordinal $\alpha$, generalizing the length function in Example~\ref{E-canonical length}. This uses the Krull dimension which Gabriel introduced for abelian categories \cite{Gabriel}. Here we follow the approach of V\'{a}mos in \cite{VamosT}. We refer the reader to \cite{Calugareanu, GR, NO} for more information on Krull dimension.

Denote by ${}_R\fM$ the category of $R$-modules. By a  subcategory of ${}_R\fM$ we mean a nonempty full subcategory.
A subcategory $\fA$ of ${}_R\fM$ is called a {\it Serre-category} if for any short exact sequence
$$0\rightarrow \sM_1\rightarrow \sM_2\rightarrow \sM_3\rightarrow 0$$
in ${}_R\fM$, $\sM_2$ is in $\fA$ if and only if both $\sM_1$ and $\sM_3$ are in $\fA$. Note that the category ${}_R\fN$ of Noetherian $R$-modules is a Serre category.

Note that $\sM\in {}_R\fM$ is simple if and only if $\sM\neq 0$ and for any submodule $\sM'$ of $\sM$, either $\sM'=0$ or $\sM/\sM'=0$. This motivates the following definition.
Let $\fA$ be a Serre-subcategory of ${}_R\fN$. We say that $\sM\in {}_R\fN$ is {\it $\fA$-simple} \cite[page 32]{VamosT} if $\sM\not\in \fA$ and for any submodule $\sM'$ of $\sM$, either $\sM'\in \fA$ or $\sM/\sM'\in \fA$. Denote by $\fA'$ the Serre-category generated by $\fA$ and all $\fA$-simple modules in ${}_R\fN$. Then $\fA'$ consists of all $\sM\in {}_R\fM$ admitting a chain $\tau$ such that each chain factor of $\tau$  is either in $\fA$ or $\fA$-simple \cite[Proposition 3.5]{VamosT}.

For each ordinal $\alpha$ (starting with $-1$), we define a Serre-subcategory $\fA_\alpha$ of ${}_R\fN$ as follows. Set $\fA_{-1}$ to be the category consisting of the zero module.
Assume that $\fA_\beta$ has been defined for every ordinal $\beta<\alpha$. We set
$$ \fA_\alpha:=(\fA_\beta)'$$
if $\alpha=\beta+1$, and
$$ \fA_\alpha:=\bigcup_{\beta<\alpha} \fA_\beta$$
if $\alpha$ is a limit ordinal.
Every $\sM\in {}_R\fN$ is in $\fA_\alpha$ for some ordinal $\alpha$ \cite[Theorem 3.7]{VamosT}. The smallest such ordinal is called the {\it Krull dimension} of $\sM$.

\begin{example} \label{E-ordinal length}
Let $\alpha$ be an ordinal. For any $\sM\in {}_R\fM$ and any chain $\tau: 0=\sM_0\subseteq \sM_1\subseteq \cdots \subseteq \sM_n=\sM$ for $\sM$ with chain factors $\{C_j\}_{1\le j\le n}$, denote by $l_\alpha(\tau)$ the number of $1\le j\le n$ for which $C_j$ is Noetherian and $\fA_\alpha$-simple. Denote by $\rL_\alpha(\sM)$ the supremum of $l_\alpha(\tau)$ for $\tau$ ranging over all chains for $\sM$. Note that the length function in Example~\ref{E-canonical length} is exactly $\rL_{-1}$. Using the Jordan-H\"{o}lder theorem and the fact that if $\sM\in {}_R\fN$ is $\fA_\alpha$-simple and $\tau$ as above is a chain for $\sM$, then exactly one of the chain factors $C_j$ is $\fA_\alpha$-simple \cite[Lemmas 3.1, 3.2]{VamosT}, one checks easily that $\rL_\alpha$ is a length function on $R$-modules, as we did in Example~\ref{E-canonical length} for $\rL_{-1}$.
Furthermore, it is clear that $\rL_\alpha$ vanishes on $\fA_\alpha$, and that for any $\sM\in \fA_{\alpha+1}\setminus \fA_\alpha$, one has $0<\rL_\alpha(\sM)<+\infty$.
\end{example}

\subsection{Group rings} \label{SS-group ring}

The {\it group ring of $\Gamma$ with coefficients in $R$}, denoted by $R\Gamma$, consists of all finitely supported functions $f: \Gamma\rightarrow R$. We shall write $f$ as $\sum_{s\in \Gamma}f_ss$, where $f_s\in R$ for all $s\in \Gamma$ and $f_s=0$ for all except finitely many $s\in \Gamma$. The algebraic operations on $R\Gamma$ are defined by
$$ \sum_{s\in \Gamma}f_ss+\sum_{s\in \Gamma}g_ss=\sum_{s\in \Gamma}(f_s+g_s)s, \mbox{ and } \big(\sum_{s\in \Gamma}f_s s\big)\big(\sum_{t\in \Gamma}g_tt\big)=\sum_{s, t\in \Gamma}f_sg_t(st).$$
We refer the reader to \cite{Passman77} for more information on group rings.

\subsection{von Neumann-L\"{u}ck dimension} \label{SS-vdim}

We recall the definitions and basic properties of the von Neumann-L\"{u}ck dimension. For details, we refer the reader to \cite{Luck98} and \cite[Section 1.1, Chapter 6]{Luck02}.

Denote by $\ell^2(\Gamma)$ the Hilbert space of square summable functions $f: \Gamma\rightarrow \Cb$, i.e. $\sum_{s\in \Gamma}|f_s|^2<+\infty$. Then $\Gamma$ has two canonical commuting unitary representations on $\ell^2(\Gamma)$, namely the {\it left regular representation} $l$ and the {\it right regular representation} $r$ defined by
$$ (l_sx)_t=x_{s^{-1}t}, \mbox{ and } (r_sx)_t=x_{ts}$$
for all $x\in \ell^2(\Gamma)$ and $s, t\in \Gamma$. The {\it group von Neumann algebra} of $\Gamma$, denoted by $\cL\Gamma$, consists of all bounded linear operators
$\ell^2(\Gamma)\rightarrow \ell^2(\Gamma)$ commuting with $r_s$ for all $s\in \Gamma$. Since $l$ commutes with $r$, the map $\sum_{s\in \Gamma}f_s s\mapsto \sum_{s\in \Gamma}f_sl_s$ is an embedding from $\Cb\Gamma$ into $\cL\Gamma$.

Denote by $\delta_{e_\Gamma}$ the unit vector of $\ell^2(\Gamma)$, which is  $1$ at the identity element $e_\Gamma$ of $\Gamma$, and $0$ everywhere else. The canonical {\it trace} on
$\cL\Gamma$ is
the linear functional $\tr_{\cL\Gamma}: \cL\Gamma\rightarrow \Cb$ given by $\tr_{\cL\Gamma}T=\left<T\delta_{e_\Gamma}, \delta_{e_\Gamma}\right>$.
 For each $n\in \Nb$, the extension of $\tr_{\cL\Gamma}$ to $M_n(\cL\Gamma)$ sending $(T_{j, k})_{1\le j, k\le n}$ to $\sum_{j=1}^n\tr_{\cL\Gamma}(T_{j, j})$
 will also be denoted by $\tr_{\cL\Gamma}$. It is {\it faithful} in the sense that for any $T\in M_n(\cL\Gamma)$, $\tr_{\cL\Gamma}(T^*T)=0$ exactly when $T=0$.

For any finitely generated projective $\cL\Gamma$-module $\Pb$, one has $\Pb\cong (\cL\Gamma)^{1\times n}P$ for some $n\in \Nb$ and some $P\in M_n(\cL\Gamma)$ with $P^2=P$. The {\it von Neumann dimension}  of $\Pb$ is defined as
$$ \dim'_{\cL\Gamma}(\Pb):=\tr_{\cL\Gamma}P\in [0, n],$$
and does not depend on the choice of $n$ and $P$. For an arbitrary $\cL\Gamma$-module $\Mb$, its {\it von Neumann-L\"{u}ck dimension} \cite[Definition 6.6]{Luck02} is defined as
$$ \dim_{\cL\Gamma}(\Mb):=\sup_{\Pb}\dim'_{\cL\Gamma}(\Pb),$$
where $\Pb$ ranges over all finitely generated projective $\cL\Gamma$-submodules of $\Mb$.

The following theorem collects the fundamental properties of the von Neumann-L\"{u}ck dimension \cite[Theorem 6.7]{Luck02}.

\begin{theorem} \label{T-vdim}
$\dim_{\cL\Gamma}$ is a length function on the $\cL\Gamma$-modules with $\dim_{\cL\Gamma}(\cL\Gamma)=1$.
\end{theorem}

\subsection{Sofic groups} \label{SS-sofic group}

The group $\Gamma$ is called {\it sofic} if for any $F\in \cF(\Gamma)$ and any $\delta>0$ there exists a map $\sigma: \Gamma\rightarrow \Sym(d)$ for some $d\in \Nb$ such that
$|\{v\in [d]: \sigma_s\sigma_t(v)=\sigma_{st}(v)\}|/d>1-\delta$ for all $s, t\in F$ and $|\{v\in [d]: \sigma_s(v)\neq \sigma_t(v)\}|/d>1-\delta$ for all distinct $s, t\in F$.
Such a map $\sigma$ is called a sofic approximation for $\Gamma$. We shall frequently write $sv$ for $\sigma_s(v)$. Equivalently, $\Gamma$ is sofic if there is a net $\Sigma:=\{\sigma_i: \Gamma\rightarrow \Sym(d_i)\}_{i\in J}$
satisfying the following conditions:
\begin{enumerate}
\item $\lim_{i\to \infty}|\{v\in [d_i]: \sigma_{i,s}\sigma_{i,t}(v)=\sigma_{i, st}(v)\}|/d_i=1$ for all $s, t\in \Gamma$,

\item $\lim_{i\to \infty}|\{v\in [d_i]: \sigma_{i, s}(v)\neq \sigma_{i,t}(v)\}|/d_i=1$ for all distinct $s, t\in \Gamma$,

\item $\lim_{i\to \infty} d_i=+\infty$.
\end{enumerate}
Such a net $\Sigma$ is called {\it a sofic approximation net} for $\Gamma$. When $\Gamma$ is countable and sofic, one can find a sofic approximation sequence for $\Gamma$.

The group $\Gamma$ is called {\it amenable} if for any nonempty $K\in \cF(\Gamma)$ and any $\delta>0$ there is a nonempty $F\in \cF(\Gamma)$ with $|KF\setminus F|<\delta |F|$.
All discrete amenable groups and residually finite groups are sofic.

We refer the reader to \cite{CC10, Pestov} for more information on sofic groups.

\section{Sofic mean length} \label{S-mean length}

In the rest of this paper, we fix a unital ring $R$ and   a length function $\rL$ on $R$-modules.
For any $R$-module $\sM$, we denote by $\sF(\sM)$ the set of all finitely generated $R$-submodules of $\sM$.
We fix a sofic group $\Gamma$ with the identity element $e_\Gamma$ and a sofic approximation net $\Sigma=\{\sigma_i:\Gamma\rightarrow \Sym(d_i)\}_{i\in J}$ for $\Gamma$. We also fix an ultrafilter $\omega$ on $J$ such that $\omega$ is {\it free} in the sense that for any $j\in J$, the set $\{i\in J: i\ge j\}$ is in $\omega$.

In this section we define the sofic mean length and establish some basic properties.
We say that an $R$-module $\sM$ is {\it locally $\rL$-finite} if $\rL(\sN)<+\infty$ for every $\sN\in \sF(\sM)$.

Let $\cM$ be an $R\Gamma$-module. Let $\sA, \sB\in \sF(\cM)$, $F\in \cF(\Gamma)$, and
$\sigma$ be a map $\Gamma\rightarrow \Sym(d)$ for some $d\in \Nb$. For any $x\in \cM$ and $v\in [d]$, denote by $\delta_vx$ the element of $\cM^d$ taking value $x$ at $v$ and $0$ everywhere else. Denote by $\sM(\sB, F, \sigma)$ the $R$-submodule of $\cM^d$ generated by the elements $\delta_vb-\delta_{sv} sb$ for all $v\in [d], b\in \sB$ and $s\in F$, and by
$\sM(\sA, \sB, F, \sigma)$ the image of $\sA^d$ in $\cM^d/\sM(\sB, F, \sigma)$ under the quotient map $\cM^d\rightarrow \cM^d/\sM(\sB, F, \sigma)$.

\begin{definition} \label{D-mean length}
Let $\cM_1\subseteq \cM_2$ be  $R\Gamma$-modules such that $\cM_1$ is locally $\rL$-finite. For $\sA\in \sF(\cM_1)$, $\sB\in \sF(\cM_2)$ and $F\in \cF(\Gamma)$, we define
\begin{align*}
\mL_{\Sigma, \omega}(\sA|\sB, F)&=\lim_{i\to \omega}\frac{\rL(\sM(\sA, \sB, F, \sigma_i))}{d_i}, \\
\mL_{\Sigma, \omega}(\sA|\sB)&=\inf_{F\in \cF(\Gamma)}\mL_{\Sigma, \omega}(\sA|\sB, F),\\
\mL_{\Sigma, \omega}(\sA|\cM_2)&=\inf_{\sB\in \sF(\cM_2)}\mL_{\Sigma, \omega}(\sA|\sB),\\
\mL_{\Sigma, \omega}(\cM_1|\cM_2)&=\sup_{\sA\in \sF(\cM_1)}\mL_{\Sigma, \omega}(\sA|\cM_2).
\end{align*}
The {\it mean length of $\cM_1$ relative to $\cM_2$} is defined as $\mL_{\Sigma, \omega}(\cM_1|\cM_2)$.
The {\it mean length of $\cM_1$} is defined as $\mL_{\Sigma, \omega}(\cM_1):=\mL_{\Sigma, \omega}(\cM_1|\cM_1)$.
\end{definition}

We prove Theorem~\ref{T-addition for mean length}.

\begin{proof}[Proof of Theorem~\ref{T-addition for mean length}]
Denote by $\pi$ the quotient map $\cM_3\rightarrow \cM_3/\cM_1$.

We prove first $\mL_{\Sigma, \omega}(\cM_2|\cM_3)\ge \mL_{\Sigma, \omega}(\cM_1|\cM_3)+\mL_{\Sigma, \omega}(\cM_2/\cM_1|\cM_3/\cM_1)$. Let $\sA_1\in \sF(\cM_1)$ and $\sA_3\in \sF(\cM_2/\cM_1)$. Take $\overline{\sA_3}\in \sF(\cM_2)$ with $\pi(\overline{\sA_3})=\sA_3$. Set $\sA_2=\sA_1+\overline{
\sA_3}\in \sF(\cM_2)$. Then it suffices to show
$$ \mL_{\Sigma, \omega}(\sA_2|\cM_3)\ge \mL_{\Sigma, \omega}(\sA_1|\cM_3)+\mL_{\Sigma, \omega}(\sA_3|\cM_3/\cM_1).$$
Let $\sB_2\in \sF(\cM_3)$ and $F\in \cF(\Gamma)$. Set $\sB_1=\sB_2$ and $\sB_3=\pi(\sB_2)\in \sF(\cM_3/\cM_1)$. Now it suffices to show
\begin{align}  \label{E-addition for mean rank3}
 \mL_{\Sigma, \omega}(\sA_2|\sB_2, F)\ge \mL_{\Sigma, \omega}(\sA_1|\sB_1, F)+\mL_{\Sigma, \omega}(\sA_3|\sB_3, F).
\end{align}
Let $\sigma$ be a map from $\Gamma$ to $\Sym(d)$ for some $d\in \Nb$. Note that $\sM(\sA_1, \sB_1, F, \sigma)$ is an $R$-submodule of $\sM(\sA_2, \sB_2, F, \sigma)$, and
$$\sM(\sA_3, \sB_3, F, \sigma)=\sA_2^d/(\sA_2^d\cap (\cM_1^d+\sM(\sB_2, F, \sigma)))$$
is a quotient $R$-module of
$$\sM(\sA_2, \sB_2, F, \sigma)/\sM(\sA_1, \sB_1, F, \sigma)=\sA_2^d/(\sA_2^d\cap (\sA_1^d+\sM(\sB_2, F, \sigma))).$$
Thus
\begin{align*}
\rL(\sM(\sA_2, \sB_2, F, \sigma))&=\rL(\sM(\sA_1, \sB_1, F, \sigma))+\rL(\sM(\sA_2, \sB_2, F, \sigma)/\sM(\sA_1, \sB_1, F, \sigma))\\
&\ge \rL(\sM(\sA_1, \sB_1, F, \sigma))+\rL(\sM(\sA_3, \sB_3, F, \sigma)).
\end{align*}
It follows that \eqref{E-addition for mean rank3} holds.

Next we prove $\mL_{\Sigma, \omega}(\cM_2|\cM_3)\le \mL_{\Sigma, \omega}(\cM_1|\cM_3)+\mL_{\Sigma, \omega}(\cM_2/\cM_1|\cM_3/\cM_1)$. Let $\sA_2\in \sF(\cM_2)$. Set $\sA_3=\pi(\sA_2)\in \sF(\cM_2/\cM_1)$. Then it suffices to show
$$ \mL_{\Sigma, \omega}(\sA_2|\cM_3)\le \mL_{\Sigma, \omega}(\cM_1|\cM_3)+\mL_{\Sigma, \omega}(\sA_3|\cM_3/\cM_1).$$
Let $\sB_3\in \sF(\cM_3/\cM_1)$ and $F_3\in \cF(\Gamma)$ containing $e_\Gamma$. Let $\varepsilon>0$. Take $\overline{\sB_3}\in \sF(\cM_3)$ with $\pi(\overline{\sB_3})=\sB_3$.
Set $\sD=(\sA_2+\sum_{t\in F_3}t\overline{\sB_3})\cap \cM_1$.
Note $$\rL(\sD)\le \rL(\sA_2+\sum_{t\in F_3}t\overline{\sB_3})<+\infty.$$
By the upper continuity of $\rL$, we can find an $\sA_1\in \sF(\cM_1)$ with  $\sA_1\subseteq \sD$ and
$$\rL(\sD)<\rL(\sA_1)+\varepsilon.$$
(If $R$ is left Noetherian, then $\sD$ is a finitely generated $R$-module, and we can simply take $\sA_1=\sD$ without invoking the upper continuity of $\rL$.)
Then it suffices to show
$$ \mL_{\Sigma, \omega}(\sA_2|\cM_3)\le \mL_{\Sigma, \omega}(\sA_1|\cM_3)+\mL_{\Sigma, \omega}(\sA_3|\sB_3, F_3)+\varepsilon.$$
Let $\sB_1\in \sF(\cM_3)$ and $F_1\in \cF(\Gamma)$. Set $\sB_2=\sB_1+\overline{\sB_3}\in \sF(\cM_3)$ and $F_2=F_1\cup F_3\in \cF(\Gamma)$. Now it suffices to show
\begin{align} \label{E-addition for mean rank4}
\mL_{\Sigma, \omega}(\sA_2|\sB_2, F_2)\le \mL_{\Sigma, \omega}(\sA_1|\sB_1, F_1)+\mL_{\Sigma, \omega}(\sA_3|\sB_3, F_3)+\varepsilon.
\end{align}
Let $\sigma$ be a map from $\Gamma$ to $\Sym(d)$ for some $d\in \Nb$. Denote by $\varphi$ the quotient map $\cM_3^d\rightarrow \cM_3^d/\sM(\overline{\sB_3}, F_3, \sigma)$, and by
$\psi$ the quotient map $\cM_3^d/\sM(\overline{\sB_3}, F_3, \sigma)\rightarrow \cM_3^d/\sM(\sB_2, F_2, \sigma)$.
From the quotient map $\varphi(\sD^d)\rightarrow \varphi(\sD^d)/\varphi(\sA_1^d)$, we see that
$(\varphi(\sA_2^d)\cap \varphi(\sD^d))/(\varphi(\sA_2^d)\cap \varphi(\sA_1^d))$ is isomorphic to an $R$-submodule of
$\varphi(\sD^d)/\varphi(\sA_1^d)$. Therefore
\begin{align*}
\rL((\varphi(\sA_2^d)\cap \varphi(\sD^d))/(\varphi(\sA_2^d)\cap \varphi(\sA_1^d)))
&\le \rL(\varphi(\sD^d)/\varphi(\sA_1^d))\\
&\le \rL(\sD^d/\sA_1^d)\\
&= \rL((\sD/\sA_1)^d)\\
&= d\cdot \rL(\sD/\sA_1)\\
&< d\varepsilon.
\end{align*}
Note that
$$\varphi(\sA_2^d)\cap \varphi(\cM_1^d)= \varphi(\sA_2^d)\cap \varphi(\sD^d).$$
Also note that
$\sM(\sA_3, \sB_3, F_3, \sigma)$ is isomorphic to $\varphi(\sA_2^d)/(\varphi(\sA_2^d)\cap \varphi(\cM_1^d))$ as $R$-modules.
Thus
\begin{align*}
\rL(\sM(\sA_3, \sB_3, F_3, \sigma))&=\rL(\varphi(\sA_2^d)/(\varphi(\sA_2^d)\cap \varphi(\cM_1^d)))\\
&=\rL(\varphi(\sA_2^d)/(\varphi(\sA_2^d)\cap \varphi(\sD^d)))\\
&=\rL(\varphi(\sA_2^d)/(\varphi(\sA_2^d)\cap \varphi(\sA_1^d)))\\
&\quad\quad -\rL((\varphi(\sA_2^d)\cap \varphi(\sD^d))/(\varphi(\sA_2^d)\cap \varphi(\sA_1^d)))\\
&\ge \rL(\varphi(\sA_2^d)/(\varphi(\sA_2^d)\cap \varphi(\sA_1^d)))-d\varepsilon.
\end{align*}
Note that $\psi(\varphi(\sA_2^d)\cap \varphi(\sA_1^d))\subseteq \psi\circ \varphi(\sA_2^d)\cap \psi\circ \varphi(\sA_1^d)$, whence $\psi\circ \varphi(\sA_2^d)/(\psi\circ \varphi(\sA_2^d)\cap \psi\circ \varphi(\sA_1^d))$ is a quotient $R$-module of
$\varphi(\sA_2^d)/(\varphi(\sA_2^d)\cap \varphi(\sA_1^d))$. Therefore
\begin{align*}
\rL(\sM(\sA_3, \sB_3, F_3, \sigma))&\ge \rL(\varphi(\sA_2^d)/(\varphi(\sA_2^d)\cap \varphi(\sA_1^d)))-d\varepsilon\\
&\ge \rL(\psi\circ \varphi(\sA_2^d)/(\psi\circ \varphi(\sA_2^d)\cap \psi\circ \varphi(\sA_1^d)))-d\varepsilon\\
&=\rL(\psi\circ \varphi(\sA_2^d))-\rL(\psi\circ \varphi(\sA_2^d)\cap \psi\circ \varphi(\sA_1^d))-d\varepsilon\\
&\ge \rL(\psi\circ \varphi(\sA_2^d))-\rL(\psi\circ \varphi(\sA_1^d))-d\varepsilon\\
&= \rL(\sM(\sA_2, \sB_2, F_2, \sigma))-\rL(\sM(\sA_1, \sB_2, F_2, \sigma))-d\varepsilon\\
&\ge \rL(\sM(\sA_2, \sB_2, F_2, \sigma))-\rL(\sM(\sA_1, \sB_1, F_1, \sigma))-d\varepsilon.
\end{align*}
It follows that \eqref{E-addition for mean rank4} holds.
\end{proof}

\begin{corollary} \label{C-addition for mean rank2}
For any locally $\rL$-finite $R\Gamma$-modules $\cM_1$ and $\cM_2$, one has
$$\mL_{\Sigma, \omega}(\cM_1|\cM_1\oplus \cM_2)=\mL_{\Sigma, \omega}(\cM_1), \mbox{ and } \mL_{\Sigma, \omega}(\cM_1\oplus \cM_2)=\mL_{\Sigma, \omega}(\cM_1)+\mL_{\Sigma, \omega}(\cM_2).$$
\end{corollary}
\begin{proof} The first identity is obvious. The second one follows from Theorem~\ref{T-addition for mean length}.
\end{proof}

The following lemma says that to calculate the mean length for finitely generated $R\Gamma$-modules, it suffices to consider generators.

\begin{lemma} \label{L-mean length generating}
Let $\cM_1\subseteq \cM_2$ be $R\Gamma$-modules such that $\cM_1$ is locally $\rL$-finite. The following are true.
\begin{enumerate}
\item If $\sA\in \sF(\cM_1)$ generates $\cM_1$ as an $R\Gamma$-module, then
$\mL_{\Sigma, \omega}(\cM_1|\cM_2)=\mL_{\Sigma, \omega}(\sA|\cM_2)\le \rL(\sA)$.
\item If $\cM_2$ is also locally $\rL$-finite and $\sB\in \sF(\cM_2)$ generates $\cM_2$ as an $R\Gamma$-module, then
$\mL_{\Sigma, \omega}(\sA|\cM_2)=\mL_{\Sigma, \omega}(\sA|\sB)$ for all $\sA\in \sF(\cM_1)$.
\end{enumerate}
\end{lemma}
\begin{proof} (1). Let $\sA'\in \sF(\cM_1)$. Take $K\in \cF(\Gamma)$ with $\sA'\subseteq \sum_{s\in K}s\sA$.

Let $\sB\in \sF(\cM_2)$ and $F\in \cF(\Gamma)$.
Set $\sB'=\sB+ \sA\in \sF(\cM_2)$ and $F'=F\cup K\in \cF(\Gamma)$.
Let $\sigma$ be a map from $\Gamma$ to $\Sym(d)$ for some $d\in \Nb$. Denote by $\varphi$ the quotient map
$\cM_2^d\rightarrow \cM_2^d/\sM(\sB', F', \sigma)$. For any $v\in [d], a\in \sA$ and $s\in K$, taking $v'\in [d]$ with $sv'=v$, we have
$$\varphi(\delta_vsa)=\varphi(\delta_{sv'}sa)=\varphi(\delta_{v'}a)\in \sM(\sA, \sB', F', \sigma).$$
Thus $\sM(\sA', \sB', F', \sigma)\subseteq \varphi((\sum_{s\in K}s\sA)^d)\subseteq \sM(\sA, \sB', F', \sigma)$. Therefore
$$\rL(\sM(\sA', \sB', F', \sigma))\le \rL(\sM(\sA, \sB', F', \sigma))\le \rL(\sM(\sA, \sB, F, \sigma)).$$
It follows that
$\mL_{\Sigma, \omega}(\sA'|\sB', F')\le \mL_{\Sigma, \omega}(\sA|\sB, F)$.
Therefore
$\mL_{\Sigma, \omega}(\sA'|\cM_2)\le \mL_{\Sigma, \omega}(\sA|\cM_2)$.
It is obvious that
$\mL_{\Sigma, \omega}(\sA|\cM_2)\le \rL(\sA)$.

(2). Let $\sB'\in \sF(\cM_2)$. Take $K\in \cF(\Gamma)$ with $\sB'\subseteq \sum_{t\in K}t\sB$. Set $\sB''=\sum_{t\in K}t\sB$.

Let $F'\in \cF(\Gamma)$ containing $e_\Gamma$. Set $F=F'K\in \cF(\Gamma)$. Then $F\supseteq K$. Let $0<\varepsilon<1$.
Let $\sigma:\Gamma\rightarrow \Sym(d)$ be a sufficiently good  sofic approximation for $\Gamma$ with $|\cW|\ge (1-\varepsilon/|K|)d$, where
$$\cW:=\{v\in [d]: s(tv)=(st)v \mbox{ for all } s\in F', t\in K\}.$$
Denote by $\sM$ (resp. $\sM^\dag$) the $R$-submodule of $\cM_2^d$ generated by $\delta_vb''-\delta_{sv}sb''$ for $v\in \bigcap_{t\in K}t\cW$ (resp. $v\in [d]\setminus \bigcap_{t\in K}t\cW$), $s\in F'$ and $b''\in \sB''$.
Then $\sM(\sB'', F', \sigma)=\sM+\sM^\dag$.
For any $v\in \bigcap_{t\in K}t\cW, b\in \sB, t\in K$ and $s\in F'$, writing $v$ as $tv'$ for some $v'\in \cW$, we have
$$ \delta_vtb-\delta_{sv}stb=(\delta_{v'}b-\delta_{stv'}stb)-(\delta_{v'}b-\delta_{tv'}tb)\in \sM(\sB, F, \sigma).$$
Thus $\sM\subseteq \sM(\sB, F, \sigma)$.
Denote by $\psi$ and $\varphi$ the quotient maps
$\cM_2^d\rightarrow \cM_2^d/\sM(\sB, F, \sigma)$ and $\cM_2^d\rightarrow \cM_2^d/\sM$ respectively. Then $\psi$ factors through $\varphi$.
Thus
\begin{align*}
\rL(\sM(\sA, \sB', F', \sigma))&\ge \rL(\sM(\sA, \sB'', F', \sigma))\\
&= \rL(\varphi(\sA^d)/(\varphi(\sA^d)\cap \varphi(\sM^\dag)))\\
&=\rL(\varphi(\sA^d)+\varphi(\sM^\dag))-\rL(\varphi(\sM^\dag))\\
&\ge \rL(\varphi(\sA^d))-\rL(\sM^\dag)\\
&\ge \rL(\psi(\sA^d))-\big|[d]\setminus \bigcap_{t\in K}t\cW\big| \cdot |F'|\cdot \rL(\sB'')\\
&\ge \rL(\sM(\sA, \sB, F, \sigma))-d\varepsilon |F'|\cdot |K|\cdot \rL(\sB).
\end{align*}
It follows that
$\mL_{\Sigma, \omega}(\sA|\sB', F')\ge \mL_{\Sigma, \omega}(\sA|\sB, F)-\varepsilon |F'|\cdot |K|\cdot \rL(\sB)$.
Letting $\varepsilon\to 0$, we get $\mL_{\Sigma, \omega}(\sA|\sB', F')\ge \mL_{\Sigma, \omega}(\sA|\sB, F)$.
Therefore
$\mL_{\Sigma, \omega}(\sA|\sB')\ge \mL_{\Sigma, \omega}(\sA|\sB)$.
\end{proof}

Next we discuss some continuity properties for mean length.

\begin{proposition} \label{P-continuity}
Let $\cM_1\subseteq \cM_2$ be $R\Gamma$-modules such that $\cM_1$ is locally $\rL$-finite. The following are true.
\begin{enumerate}
\item If $\{\cM_j'\}_{j\in \cJ}$ is an increasing net of $R\Gamma$-submodules of $\cM_1$ with union $\cM_1$, then  $\mL_{\Sigma, \omega}(\cM_j'|\cM_2)\nearrow \mL_{\Sigma, \omega}(\cM_1|\cM_2)$. If furthermore $\cM_2$ is locally $\rL$-finite and $\mL_{\Sigma, \omega}(\cM_2)<+\infty$, then $\mL_{\Sigma, \omega}(\cM_2/\cM_j')\searrow \mL_{\Sigma, \omega}(\cM_2/\cM_1)$.
\item If $\cM_1$ is a finitely generated $R\Gamma$-module, and $\{\cM_j'\}_{j\in \cJ}$ is an increasing net of $R\Gamma$-submodules of $\cM_2$ containing $\cM_1$ with union $\cM_2$, then $\mL_{\Sigma, \omega}(\cM_1|\cM_j')\searrow \mL_{\Sigma, \omega}(\cM_1|\cM_2)$.
\end{enumerate}
\end{proposition}
\begin{proof} (1). The first part is trivial. The second part follows from the first part and Theorem~\ref{T-addition for mean length}.

(2) follows from Lemma~\ref{L-mean length generating}.
\end{proof}

Now we consider $R\Gamma$-modules of the form $R\Gamma\otimes_R\sN$ for $R$-modules $\sN$.

\begin{proposition} \label{P-mean rank basic}
The following are true.
\begin{enumerate}
\item For any $R$-modules $\sN_1\subseteq \sN_2$ such that $\sN_2$ is locally $\rL$-finite, one has
$\mL_{\Sigma, \omega}(R\Gamma\otimes_R\sN_1|R\Gamma\otimes_R\sN_2)=\rL(\sN_1)$.
\item For any locally $\rL$-finite $R$-module $\sN$ and any $R\Gamma$-submodule $\cM$ of $R\Gamma\otimes_R\sN$, $\mL_{\Sigma, \omega}(\cM|R\Gamma\otimes_R \sN)=0$ if and only if
$\mL_{\Sigma, \omega}(\cM)=0$, if and only if $\rL(\cM)=0$.
\end{enumerate}
\end{proposition}
\begin{proof}
(1). Let $\sA\in \sF(R\Gamma\otimes_R\sN_1)$. Then $\sA\subseteq R\Gamma\otimes_R\sA'$ for some
$\sA'\in \sF(\sN_1)$. Note that $e_\Gamma\otimes \sA'$ is in $\sF(R\Gamma\otimes_R\sA')$ and generates $R\Gamma\otimes_R\sA'$ as an $R\Gamma$-module. By Lemma~\ref{L-mean length generating} one has
$$ \mL_{\Sigma, \omega}(\sA|R\Gamma\otimes_R\sN_2)\le \mL_{\Sigma, \omega}(R\Gamma\otimes_R\sA'|R\Gamma\otimes_R\sN_2)\le \rL(e_\Gamma\otimes\sA')=\rL(\sA')\le \rL(\sN_1).$$
Thus $\mL_{\Sigma, \omega}(R\Gamma\otimes_R\sN_1|R\Gamma\otimes_R\sN_2)\le \rL(\sN_1)$.

Let $\sA\in \sF(\sN_1)$.
Let $\sB\in \sF(R\Gamma\otimes_R\sN_2)$. Then $\sB\subseteq R\Gamma\otimes_R \sB'$ for some
$\sB'\in \sF(\sN_2)$. Let $F\in \cF(\Gamma)$ and $0<\varepsilon<1$. Let $\sigma: \Gamma\rightarrow \Sym(d)$ be a good enough sofic approximation for $\Gamma$ with $|\cW|\ge (1-\varepsilon)d$, where
$$ \cW:=\{v\in [d]: e_\Gamma v=v\}.$$
Note that $(e_\Gamma\otimes \sA)^d\cap \sM(e_\Gamma\otimes \sB', F, \sigma)\subseteq \sum_{v\in [d]\setminus \cW}(\delta_v(e_\Gamma\otimes \sB')+\delta{e_\Gamma v}(e_\Gamma\otimes \sB'))$.
Thus
\begin{align*}
\rL(\sM(e_\Gamma\otimes \sA, e_\Gamma\otimes \sB', F, \sigma))&=\rL((e_\Gamma\otimes \sA)^d)-\rL((e_\Gamma\otimes \sA)^d\cap \sM(e_\Gamma\otimes \sB', F, \sigma))\\
&\ge d\rL(\sA)-\rL(\sum_{v\in [d]\setminus \cW}(\delta_v(e_\Gamma\otimes \sB')+\delta{e_\Gamma v}(e_\Gamma\otimes \sB')))\\
&\ge d\rL(\sA)-2\big|[d]\setminus \cW\big|\rL(\sB')\\
&\ge d\rL(\sA)-2d\varepsilon \rL(\sB').
\end{align*}
Therefore
$$ \mL_{\Sigma, \omega}(e_\Gamma\otimes \sA|e_\Gamma\otimes \sB', F)\ge \rL(\sA)-2\varepsilon \rL(\sB').$$
Letting $\varepsilon\to 0$, we get $\mL_{\Sigma, \omega}(e_\Gamma\otimes \sA|e_\Gamma\otimes \sB', F)\ge \rL(\sA)$. Since $F$ is an arbitrary finite subset of $\Gamma$, we conclude
that $\mL_{\Sigma, \omega}(e_\Gamma\otimes \sA|e_\Gamma\otimes \sB')\ge \rL(\sA)$. By Lemma~\ref{L-mean length generating} we get
$$ \mL_{\Sigma, \omega}(e_\Gamma\otimes \sA|\sB)\ge \mL_{\Sigma, \omega}(e_\Gamma\otimes \sA|R\Gamma\otimes_R\sB')=\mL_{\Sigma, \omega}(e_\Gamma\otimes \sA|e_\Gamma\otimes \sB')\ge \rL(\sA).$$
Since $\sB$ is an arbitrary finitely generated $R$-submodule of $R\Gamma\otimes_R \sN_2$, we get
$$\mL_{\Sigma, \omega}(R\Gamma\otimes_R\sN_1|R\Gamma\otimes_R\sN_2)\ge \mL_{\Sigma, \omega}(e_\Gamma\otimes \sA|R\Gamma\otimes_R\sN_2)\ge \rL(\sA).$$
By the upper continuity of $\rL$, we conclude that $\mL_{\Sigma, \omega}(R\Gamma\otimes_R\sN_1|R\Gamma\otimes_R\sN_2)\ge \rL(\sN_1)$.
Therefore $\mL_{\Sigma, \omega}(R\Gamma\otimes_R\sN_1|R\Gamma\otimes_R\sN_2)=\rL(\sN_1)$.

(2). If $\rL(\cM)=0$, then clearly $\mL_{\Sigma, \omega}(\cM)=0$. Since $\mL_{\Sigma, \omega}(\cM|R\Gamma\otimes_R\sN)\le \mL_{\Sigma, \omega}(\cM)$, it suffices to show that
if $\mL_{\Sigma, \omega}(\cM|R\Gamma\otimes_R\sN)=0$, then $\rL(\cM)=0$.

Let $\sA\in \sF(\cM)$. Then there exists some $K\in \cF(\Gamma)$ such that $e_\Gamma\in K=K^{-1}$ and $\sA\subseteq \sum_{t\in K}t\otimes \sN$.
Let $\sB\in \sF(R\Gamma\otimes_R\sN)$. Then $\sB\subseteq R\Gamma\otimes_R\sB'$ for some
$\sB'\in \sF(\sN)$. Let $F\in \cF(\Gamma)$.

Let $0<\varepsilon<1$. Let $\sigma: \Gamma\rightarrow \Sym(d)$ be a good enough sofic approximation of $\Gamma$ with $|\cW|\ge (1-\varepsilon)d$, where
$$\cW:=\{v\in [d]: t(sv)=(ts)v \mbox{ for all } s,t\in K, sv\neq s'v \mbox{ for all } s\neq s' \in K, \mbox{ and } e_\Gamma v=v\}.$$
Take a maximal subset $\cV$ of $\cW$ subject to the condition that
the sets $Kv$ are pairwise disjoint for $v\in \cV$.
Then $\sigma(K)^{-1}\sigma(K)\cV\supseteq \cW$, whence $|\cV|\ge |\cW|/|K|^2\ge (1-\varepsilon)d/|K|^2$.
Denote by $\sM$ (resp. $\sM^\dag$) the $R$-submodule of $R\Gamma\otimes_R\sN$ generated by $\delta_ve_\Gamma\otimes b'-\delta_{sv}s\otimes b'$ for all $v\in \cW$ (resp. $v\in [d]\setminus \cW$), $b'\in \sB'$ and $s\in F$.
Then $\sM(e_\Gamma\otimes \sB', F, \sigma)=\sM+\sM^\dag$.

Set $\tilde{\sA}=\sum_{v\in \cV}\delta_v\sA$.
We claim that $\tilde{\sA}\cap \sM=\{0\}$. Let $x\in \tilde{\sA}\cap \sM$. Then $x=\sum_{v\in \cW}\sum_{s\in F}(\delta_ve_\Gamma\otimes f(v, s)-\delta_{sv}s\otimes f(v, s))$
 for some map $f: \cW\times F\rightarrow \sB'$. Since $e_\Gamma v=v$ for all $v\in \cW$, we may assume that $f(v, e_\Gamma)=0$ for all $v\in \cW$, in case $e_\Gamma\in F$.
 Because $x\in \tilde{\sA}$, we must have $f(v, s)=0$ for all  $(v, s)\in \cW\times (F\setminus K)$ and all $(v, s)\in \cW\times F$ with $sv\not\in \cV$. By our choice of $\cV$, we then conclude that $f=0$. Thus $x=0$. This proves our claim.

Denote by $\varphi$ the quotient map $(R\Gamma\otimes_R\sN)^d\rightarrow (R\Gamma\otimes_R\sN)^d/\sM$. Then
\begin{align*}
\rL(\sM(\sA, e_\Gamma\otimes \sB', F, \sigma))&=\rL(\varphi(\sA^d)/(\varphi(\sA^d)\cap \varphi(\sM^\dag)))\\
&=\rL(\varphi(\sA^d)+\varphi(\sM^\dag))-\rL(\varphi(\sM^\dag))\\
&\ge \rL(\varphi(\sA^d))-\big| [d]\setminus \cW\big|\cdot |F|\cdot \rL(\sB')\\
&\ge \rL(\varphi(\tilde{\sA}))-d\varepsilon \cdot |F|\cdot \rL(\sB')\\
&= \rL(\tilde{\sA})- d\varepsilon \cdot |F|\cdot \rL(\sB')\\
&= |\cV|\cdot \rL(\sA)-d\varepsilon \cdot |F|\cdot \rL(\sB')\\
&\ge \rL(\sA)(1-\varepsilon)d/|K|^2-d\varepsilon \cdot |F|\cdot \rL(\sB').
\end{align*}
Thus
$$ \mL_{\Sigma, \omega}(\sA|e_\Gamma\otimes \sB', F)\ge \rL(\sA)(1-\varepsilon)/|K|^2-\varepsilon \cdot |F|\cdot \rL(\sB').$$
Letting $\varepsilon\to 0$, we get
$ \mL_{\Sigma, \omega}(\sA|e_\Gamma\otimes \sB', F)\ge \rL(\sA)/|K|^2$, and hence $\mL_{\Sigma, \omega}(\sA|e_\Gamma\otimes \sB')\ge \rL(\sA)/|K|^2$.
By Lemma~\ref{L-mean length generating}, we have
$$\mL_{\Sigma, \omega}(\sA|\sB)\ge \mL_{\Sigma, \omega}(\sA|R\Gamma\otimes_R \sB')=\mL_{\Sigma, \omega}(\sA|e_\Gamma\otimes \sB')\ge \rL(\sA)/|K|^2.$$
Therefore
$$ \mL_{\Sigma, \omega}(\sA|R\Gamma\otimes_R\sN)\ge \rL(\sA)/|K|^2.$$
If $\mL_{\Sigma, \omega}(\cM|R\Gamma\otimes_R\sN)=0$, then $\mL_{\Sigma, \omega}(\sA|R\Gamma\otimes_R\sN)=0$ and hence $\rL(\sA)=0$ for all $\sA\in \sF(\cM)$, which implies $\rL(\cM)=0$ by the upper continuity of $\rL$.
\end{proof}

To round this section, we compute the mean length in a fairly simple case.

\begin{proposition} \label{P-finite}
Suppose that $\Gamma$ is infinite and $\cM$ is an $R\Gamma$-module with $\rL(\cM)<+\infty$. Then $\mL_{\Sigma, \omega}(\cM)=0$.
\end{proposition}
\begin{proof} Let $\sA\in \sF(\cM)$ and $\varepsilon>0$. By the upper continuity of $\rL$, we can find a $\sB\in \sF(\cM)$ such that
$\rL(\cM)<\rL(\sB)+\varepsilon$.

Let $K\in \cF(\Gamma)$ be nonempty. Set $F=K^{-1}K\in \cF(\Gamma)$. Let $\sigma: \Gamma\rightarrow \Sym(d)$ be a good enough sofic approximation of $\Gamma$ with $|\cW|\ge (1-\varepsilon)d$, where
$$ \cW:=\{v\in [d]: \sigma_sv\neq \sigma_tv \mbox{ for all distinct } s, t\in K, \mbox{ and } (\sigma_s)^{-1}\sigma_tv=\sigma_{s^{-1}t}v \mbox{ for all } s,t \in K\}.$$
Take a maximal subset $\cV$ of $\cW$ subject to the condition that the sets $Kv$ are pairwise disjoint for $v\in \cV$.
Then $\sigma(F)\cV=\sigma(K^{-1}K)\cV=\sigma(K)^{-1}\sigma(K)\cV\supseteq \cW$ and
$|\cV|\le d/|K|$.

Denote by $\sM$ the $R$-submodule of $\cM^d$ generated by $\delta_v\cM$ for $v\in \cV$.
Denote by $\psi$ the quotient map $\cM^d\rightarrow \cM^d/\sM(\sB, F, \sigma)$. For each $s\in F$, set $\sB_s:=s\sB\cap \sA$.
Then $\rL(\sA/\sB_s)\le \rL(\cM/s\sB)\le \varepsilon$. For any $(v, s)\in \cV\times F$ and any $a\in \sB_s$, say $a=sb$ with $b\in \sB$, we have
$$\psi(\delta_{sv}a)=\psi(\delta_{sv}sb)=\psi(\delta_v b)\in \psi(\sM). $$
Thus
$ \psi(\sum_{v\in \cV}\sum_{s\in F}\delta_{sv}\sB_s)\subseteq \psi(\sM)$, whence
$$\rL(\psi(\sum_{v\in \cV}\sum_{s\in F}\delta_{sv}\sB_s))\le \rL(\psi(\sM))\le \rL(\sM)=|\cV|\rL(\cM).$$
Set $\cW':=\{sv: s\in F, v\in \cV\}$. Then $\cW'\supseteq \cW$. Therefore
\begin{align*}
\rL(\sM(\sA, \sB, F, \sigma))&=\rL(\psi(\sum_{v\in \cV}\sum_{s\in F}\delta_{sv}\sB_s))+\rL(\psi(\sA^d)/\psi(\sum_{v\in \cV}\sum_{s\in F}\delta_{sv}\sB_s))\\
&\le |\cV|\rL(\cM)+\rL(\sA^d/\sum_{v\in \cV}\sum_{s\in F}\delta_{sv}\sB_s)\\
&= |\cV|\rL(\cM)+\big|[d]\setminus \cW'\big|\rL(\sA)+\sum_{v'\in \cW'}\rL(\sA/\sum_{v\in \cV, s\in F, sv=v'}\sB_s)\\
&\le d(1/|K|+\varepsilon)\rL(\cM)+d\varepsilon.
\end{align*}
It follows that $\mL_{\Sigma, \omega}(\sA|\cM)\le (1/|K|+\varepsilon)\rL(\cM)+\varepsilon$. Since $K$ is an arbitrary nonempty finite subset of $\Gamma$, we get $\mL_{\Sigma, \omega}(\sA|\cM)\le (\rL(\cM)+1)\varepsilon$. As $\varepsilon$ is an arbitrary positive number, we get $\mL_{\Sigma, \omega}(\sA|\cM)=0$. Thus $\mL_{\Sigma, \omega}(\cM)=0$.
\end{proof}

\section{Stably direct finiteness} \label{S-direct finite}

In this section we prove Theorem~\ref{T-Noetherian}.

For any unital ring $R$, one has its opposite ring $R^{\op}:=\{a^{\op}: a\in R\}$ with operations defined by $a^{\op}-b^{\op}=(a-b)^{\op}$ and $a^{\op}b^{\op}=(ba)^{\op}$. Using the transpose map it is easy to see that
$(M_n(R))^{\op}\cong M_n(R^{\op})$ for every $n\in \Nb$. Thus
$R$ is stably direct finite if and only if $R^{\op}$ is so. Then Theorem~\ref{T-Noetherian} follows from Lemma~\ref{L-opposite} and Theorem~\ref{T-directly finite} below.

\begin{lemma} \label{L-opposite}
For any unital ring $R$ and any (not necessarily sofic) group $\Gamma$, $R^{\op}\Gamma$ is isomorphic to $(R\Gamma)^{\op}$.
\end{lemma}
\begin{proof} For each $x\in R^{\op}\Gamma$, writing $x$ as $\sum_{s\in \Gamma}a^{\op}_ss$, where $a_s\in R$ for each $s\in \Gamma$ and $a_s=0$ for all but finitely many $s\in \Gamma$, we
define an element $\varphi(x)$ of $(R\Gamma)^{\op}$ by
$\varphi(x)=(\sum_{s\in \Gamma}a_ss^{-1})^{\op}$. It is easily checked that the map $\varphi: R^{\op}\Gamma\rightarrow (R\Gamma)^{\op}$ sending $x$ to $\varphi(x)$ is an isomorphism.
\end{proof}

For any $R$-module $\sM$, we denote by $\End_R(\sM)$ the endomorphism ring of $\sM$.

\begin{theorem} \label{T-directly finite}
Let $R$ be a unital ring, $\Gamma$ be a sofic group, and $\sM$ be a nonzero Noetherian $R$-module. Set $\tilde{R}:=\End_R(\sM)$. Then $\tilde{R}\Gamma$ is stably direct finite.
\end{theorem}

Theorem~\ref{T-directly finite} was proved by Ceccherini-Silberstein and Coornaert
under the further assumption that $\sM$ is also Artinian \cite[Corollary 1.4]{CC07}.
Theorem~\ref{T-directly finite} follows from  Lemmas~\ref{L-end}, \ref{L-direct summand} and Theorem~\ref{T-hopfian} below, and the fact that $\End_R(\sM^{\oplus n})\cong M_n(\End_R(\sM))$ for every $n\in \Nb$, where $\sM^{\oplus n}$ denotes the direct sum of $n$ copies of $\sM$.

\begin{lemma} \label{L-end}
Let $R$ be a unital ring, $\sM$ be a nonzero finitely generated $R$-module, and $\Gamma$ be a (not necessarily sofic) group. Set $\tilde{R}:=\End_R(\sM)$. Then $\End_{R\Gamma}(R\Gamma\otimes_R\sM)$ is isomorphic to $\tilde{R}\Gamma$.
\end{lemma}
\begin{proof} Each element $a$ of $R\Gamma\otimes_R\sM$ can be written as $\sum_{t\in \Gamma} t\otimes a_t$ with $a_t\in\sM$ for each $t\in \Gamma$ and $a_t=0$ for all but finitely many $t\in \Gamma$.
For each $f\in \tilde{R}\Gamma$, writing $f$ as $\sum_{s\in \Gamma} f_s s$ with $f_s\in \tilde{R}$ for each $s\in \Gamma$ and $f_s=0$ for all but finitely many $s\in \Gamma$, we define a $\varphi(f)\in \End_{R\Gamma}(R\Gamma\otimes_R\sM)$ by
$$\varphi(f)(\sum_{t\in \Gamma}t\otimes a_t)=\sum_{s, t\in \Gamma}ts^{-1}\otimes f_s(a_t).$$
It is easily checked that the map $\varphi: \tilde{R}\Gamma\rightarrow \End_{R\Gamma}(R\Gamma\otimes_R\sM)$  sending $f$ to $\varphi(f)$ is an isomorphism.
\end{proof}

The following lemma is well known.

\begin{lemma} \label{L-direct summand}
Let $R$ be a unital ring and $\sM$ be a nonzero $R$-module. Then $\End_R(\sM)$ is directly finite if and only if $\sM$ has no nontrivial direct summand isomorphic to itself.
\end{lemma}
\begin{proof} We prove the ``if'' part first. Suppose that $\End_R(\sM)$ is not directly finite. Then we can find $f, g\in \End_R(\sM)$ such that $fg$ is the identity map $\id$ on $\sM$, while $gf$ is not. Note that
 $(gf)^2=gf$. Set  $\sM_1=gf(\sM)$ and $\sM_2=(\id-gf)(\sM)$. Then
$\sM=\sM_1\oplus \sM_2$ and $\sM_2\neq 0$. Note that $\sM_1\subseteq g(\sM)$. Since $gx=gf(gx)\in \sM_1$ for every $x\in \sM$, we have $g(\sM)\subseteq \sM_1$, and hence $\sM_1=g(\sM)$.
The surjective homomorphism $\sM\rightarrow g(\sM)$ sending $x$ to $gx$ is injective since $f(gx)=x$. Therefore $\sM_1$ is isomorphic to $\sM$.

Next we prove the ``only if'' part. Suppose that $\sM$ has a nontrivial direct summand  isomorphic to itself. Say, $\sM=\sM_1\oplus \sM_2$, $\sM_2\neq 0$, and $\sM\cong \sM_1$. Let $\varphi:\sM\rightarrow \sM_1$ be an isomorphism. Denoting by $\iota$ the embedding $\sM_1\rightarrow \sM$, we set $g=\iota\circ\varphi\in \End_R(\sM)$. Denoting by $\pi$ the projection $\sM\rightarrow \sM_1$, we set $f=\varphi^{-1}\circ\pi\in \End_R(\sM)$. Then $fg=\id$ while $gf\neq \id$.
\end{proof}

An $R\Gamma$-module $\cM$ is called {\it hopfian} if every surjective module homomorphism $\cM\rightarrow \cM$ is injective.

\begin{theorem} \label{T-hopfian}
Let $R$ be a unital ring, $\Gamma$ be a sofic group, and $\sM$ be a Noetherian $R$-module. Then the  $R\Gamma$-module $R\Gamma\otimes_R\sM$ is hopfian.
\end{theorem}

Using mean length for amenable groups, Virili showed that if $R$ is a unital left Noetherian ring, $\Gamma$ is a finitely generated amenable group,
and $\sM$  is a finitely generated $R$-module, then the $R\Gamma$-module $R\Gamma\otimes_R\sM$ is {\it hereditarily hopfian} in the sense that every $R\Gamma$-submodule of $R\Gamma\otimes_R\sM$ is hopfian \cite[Theorem A]{ViriliA}. Our proof of Theorem~\ref{T-hopfian} follows Virili's method, but uses the sofic mean length defined in Section~\ref{S-mean length}.

We shall use the notations in Section~\ref{SS-length}.
Let $\alpha$ be an ordinal and $\sM\in {}_R\fM$. Denote by $T_\alpha(\sM)$ the set of all $x\in \sM$ such that $Rx\in \fA_\alpha$. Since $\fA_\alpha$ is a Serre-category,
$T_\alpha(\sM)$ is an $R$-submodule of $\sM$. Note that $T_\alpha(\sM)\subseteq T_\gamma(\sM)$ for all ordinals $\alpha<\gamma$, and
$T_\alpha(\sM)=\bigcup_{\beta<\alpha}T_\beta(\sM)$ for every limit ordinal $\alpha$.

\begin{lemma} \label{L-ordinal for kernel}
Let $R, \Gamma$ and $\sM$ be as in Theorem~\ref{T-hopfian}. Let $\sN$ be a nonzero $R$-submodule of $R\Gamma\otimes_R\sM$. Then there is an ordinal $\alpha$ such that $T_\alpha(\sN)\neq \sN$ and $T_{\alpha+1}(\sN)=\sN$.
\end{lemma}
\begin{proof} For any ordinal $\alpha$, one has $T_\alpha(R\Gamma\otimes_R\sM)=R\Gamma\otimes_RT_\alpha(\sM)$ and
$T_\alpha(\sN)=T_\alpha(R\Gamma\otimes_R\sM)\cap \sN$. Since $\sM$ is Noetherian, there are only finitely many ordinals $\alpha$ satisfying $T_\alpha(\sM)\neq T_{\alpha+1}(\sM)$. Thus there are only finitely many ordinals $\alpha$ satisfying $T_\alpha(\sN)\neq T_{\alpha+1}(\sN)$.  Denote by $\beta$ the Krull dimension of $\sM$. Then $T_\beta(\sM)=\sM$, whence $T_\beta(\sN)=\sN$. Therefore, for the largest ordinal $\alpha$ satisfying $T_\alpha(\sN)\neq T_{\alpha+1}(\sN)$, we must have $T_\alpha(\sN)\neq \sN$ and $T_{\alpha+1}(\sN)=\sN$.
\end{proof}

\begin{lemma} \label{L-ordinal surjective}
Let $\varphi: \sM\rightarrow \sN$ be a surjective homomorphism of $R$-modules. Suppose that $Rx$ is Noetherian for every $x\in \sM$ and that $\beta$ is an ordinal with $T_\beta(\ker(\varphi))=\ker(\varphi)$. Then $\varphi(T_\beta(\sM))=T_\beta(\sN)$.
\end{lemma}
\begin{proof} Clearly $\varphi(T_\beta(\sM))\subseteq T_\beta(\sN)$. Let $y\in T_\beta(\sN)$. Take $x\in \sM$ with $\varphi(x)=y$. Then we have a short exact sequence
$$0\rightarrow Rx\cap \ker(\varphi)\rightarrow Rx\rightarrow Ry\rightarrow 0$$
of $R$-modules. Since $Rx$ is Noetherian, $Rx\cap \ker(\varphi)$ is a finitely generated $R$-module. Because $T_\beta(\ker(\varphi))=\ker(\varphi)$, for every $z\in Rx\cap \ker(\varphi)$ we have $Rz\in \fA_\beta$. It follows that $Rx\cap \ker(\varphi)\in \fA_\beta$. Because $y\in T_\beta(\sN)$, we also have $Ry\in \fA_\beta$. Thus $Rx\in \fA_\beta$.
That is, $x\in T_\beta(\sM)$.
\end{proof}

We are ready to prove Theorem~\ref{T-hopfian}.

\begin{proof}[Proof of Theorem~\ref{T-hopfian}]
Let $\varphi$ be a surjective $R\Gamma$-module homomorphism from $R\Gamma\otimes_R\sM$ onto itself. Suppose that $\ker(\varphi)\neq 0$. By Lemma~\ref{L-ordinal for kernel} we can find an ordinal $\alpha$ such that $T_\alpha(\ker(\varphi))\neq \ker(\varphi)$ and $T_{\alpha+1}(\ker(\varphi))=\ker(\varphi)$. By Lemma~\ref{L-ordinal surjective} we have
$\varphi(T_{\alpha+1}(R\Gamma\otimes_R\sM))=T_{\alpha+1}(R\Gamma\otimes_R\sM)$. But $T_{\alpha+1}(R\Gamma\otimes_R\sM)=R\Gamma\otimes_R\sN$ for $\sN:=T_{\alpha+1}(\sM)$.
Thus $(R\Gamma\otimes_R\sN)/\ker(\varphi)\cong R\Gamma\otimes_R\sN$ as $R\Gamma$-modules.

Let $\rL:=\rL_\alpha$ be the length function on $R$-modules defined in Example~\ref{E-ordinal length}. Then we have the sofic mean length $\mL_{\Sigma, \omega}$ defined in Definition~\ref{D-mean length}. Since $\rL_\alpha$ takes finite values on modules in $\fA_{\alpha+1}$, $\sN$ is locally $\rL_\alpha$-finite. By Theorem~\ref{T-addition for mean length} we have
\begin{align*}
 \mL_{\Sigma, \omega}(R\Gamma\otimes_R\sN)&=\mL_{\Sigma, \omega}(\ker(\varphi)|R\Gamma\otimes_R\sN)+\mL_{\Sigma, \omega}((R\Gamma\otimes_R\sN)/\ker(\varphi))\\
 &=\mL_{\Sigma, \omega}(\ker(\varphi)|R\Gamma\otimes_R\sN)+\mL_{\Sigma, \omega}(R\Gamma\otimes_R\sN).
\end{align*}
Since $\sM$ is Noetherian, $\sN$ is finitely generated. Thus $\sN\in \fA_{\alpha+1}$, and hence $\rL_\alpha(\sN)<+\infty$.
By Proposition~\ref{P-mean rank basic} we have
$\mL_{\Sigma, \omega}(R\Gamma\otimes_R\sN)=\rL_\alpha(\sN)<+\infty$. Therefore
$\mL_{\Sigma, \omega}(\ker(\varphi)|R\Gamma\otimes_R\sN)=0$. By Proposition~\ref{P-mean rank basic} we  get $\rL_\alpha(\ker(\varphi))=0$.

Since $T_\alpha(\ker(\varphi))\neq \ker(\varphi)$, we can find some $x\in \ker(\varphi)$ with $Rx\not\in \fA_\alpha$. Then $\rL_\alpha(\ker(\varphi))\ge \rL_\alpha(Rx)>0$, which is a contradiction. Therefore $\varphi$ must be injective.
\end{proof}

\section{Amenable group case for mean length} \label{S-amenable mean length}

In this section we consider the amenable group case for the sofic mean length and prove Theorem~\ref{T-amenable mean length}.
Throughout this section, we let $\Gamma$ be a discrete amenable group.

We recall first the definition of mean length for amenable groups.
For any locally $\rL$-finite $R\Gamma$-module $\cM$, any $\sA\in \sF(\cM)$, and any $F\subseteq \Gamma$, we set
$\sA^F=\sum_{s\in F}s^{-1}\sA$. For any $t\in \Gamma$ and $F\in \cF(\Gamma)$ we have $\rL(\sA^{Ft})=\rL(\sA^F)$.
For any $F_1, F_2\in \cF(\Gamma)$, noting that $\sA^{F_1\cap F_2}$ is isomorphic to an $R$-submodule of the kernel of the natural surjective $R$-module homomorphism $\sA^{F_1}\oplus \sA^{F_2}\rightarrow \sA^{F_1\cup F_2}$ we have
\begin{align*}
\rL(\sA^{F_1\cup F_2})+\rL(\sA^{F_1\cap F_2})\le \rL(\sA^{F_1}\oplus \sA^{F_2})=\rL(\sA^{F_1})+\rL(\sA^{F_2}).
\end{align*}
Thus the limit $\lim_F\frac{\rL(\sA^F)}{|F|}$ as $F\in \cF(\Gamma)$ becomes more and more left invariant exists and is equal to $\inf_{F\in \cF(\Gamma)\setminus \{\emptyset\}}\frac{\rL(\sA^F)}{|F|}$ \cite[Lemma 3.3]{LT}, which we denote by $\mL(\sA)$. That is, for any $\varepsilon>0$, there exist a nonempty $K\in \cF(\Gamma)$ and a $\delta>0$ such that for any nonempty $F\in \cF(\Gamma)$ with $|KF\setminus F|<\delta |F|$, one has $\big|\frac{\rL(\sA^F)}{|F|}-\mL(\sA)\big|<\varepsilon$.
Then the {\it mean length} of $\cM$ is defined as
$$\mL(\cM):=\sup_{\sA\in \sF(\cM)}\mL(\sA).$$
This mean length $\mL(\cM)$ was studied first in \cite{SZ, SVV} for the case $\Gamma=\Zb$ and $\rL$ is discrete in the sense that the set of finite values of $\rL$ is order isomorphic to $\Nb$, then in \cite{LL} for general amenable groups assuming $\rL(R)<+\infty$,
and also in \cite{ViriliA} for general amenable groups assuming $\rL$ is discrete.

\begin{theorem} \label{T-amenable mean length}
For any locally $\rL$-finite $R\Gamma$-module $\cM$ and any $\sA\in \sF(\cM)$, we have
$$\mL_{\Sigma, \omega}(\sA|\cM)=\mL(\sA).$$
If particular, for any $R\Gamma$-modules $\cM_1\subseteq \cM_2$ such that $\cM_2$ is locally $\rL$-finite, we have
$$\mL_{\Sigma, \omega}(\cM_1|\cM_2)=\mL(\cM_1).$$
\end{theorem}

The following lemma is \cite[Lemma 4.6]{KL13a}. We need it a few times.

\begin{lemma} \label{L-Rokhlin}
For any $K\in \cF(\Gamma)$ and $0<\varepsilon<1/2$, there are $\ell\in \Nb$ and nonempty $F_1, \dots, F_\ell\in \cF(\Gamma)$ with $|KF_k\setminus F_k|<\varepsilon |F_k|$ for all $1\le k\le \ell$ such that,  for any good enough sofic approximation $\sigma: \Gamma\rightarrow \Sym(d)$ for $\Gamma$ and any $\cW\subseteq [d]$ with $|\cW|\ge (1-\varepsilon)d$, there exist $\cC_1, \dots, \cC_\ell\subseteq \cW$ such that
\begin{enumerate}
\item for every $k=1, \dots, \ell$, the map $(s, c)\mapsto sc$ from $F_k\times \cC_k$ to $[d]$ is injective,
\item the family $\{F_1\cC_1, \dots, F_\ell \cC_\ell\}$ is disjoint and $\big|\bigcup_{1\le k\le \ell}F_k\cC_k\big|\ge (1-2\varepsilon)d$.
\end{enumerate}
\end{lemma}

Theorem~\ref{T-amenable mean length} follows from Lemmas~\ref{L-amenable mean length lower} and \ref{L-amenable mean length upper} below.

\begin{lemma} \label{L-amenable mean length lower}
For any locally $\rL$-finite $R\Gamma$-module $\cM$ and any $\sA\in \sF(\cM)$, we have $\mL_{\Sigma, \omega}(\sA|\cM)\ge \mL(\sA)$.
\end{lemma}
\begin{proof} Let $\sB\in \sF(\cM)$ and $K\in \cF(\Gamma)$.

Take $0<\varepsilon<1/2$. Then we have $\ell$ and $F_1, \dots, F_\ell$ in Lemma~\ref{L-Rokhlin}.

Let $\sigma: \Gamma\rightarrow \Sym(d)$ be a good enough sofic approximation for $\Gamma$ such that $|\cW|>(1-\varepsilon)d$ for
\begin{align*}
\cW:=\{v\in [d]: s(tv)=(st)v \mbox{ for all } s\in K \mbox{ and } t\in \bigcup_{k=1}^\ell F_k\}.
\end{align*}
Then we have $\cC_1, \dots, \cC_\ell$ as in Lemma~\ref{L-Rokhlin}. Denote by $\sL$ the set of all $(k, c)$ such that $k\in \{1, \dots, \ell\}$ and $c\in \cC_k$.

Let $1\le k\le \ell$. Set $F'_k=\{t\in F_k: Kt\subseteq F_k\}$. Then $|F_k\setminus F'_k|\le |K|\cdot |KF_k\setminus F_k|<\varepsilon |F_k|\cdot |K|$.
For each $c\in \cC_k$, denote by $\sM(\sB, k, c)$ the $R$-submodule of $\cM^d$ generated by $\delta_{tc}b-\delta_{stc}sb$ for $t\in F'_k$, $s\in K$ and $b\in \sB$.

Denote by $\sM^\dag$ the $R$-submodule of $\cM^d$ generated by $\delta_vb-\delta_{sv}sb$ for $v\in [d]\setminus \bigcup_{k=1}^\ell F'_k\cC_k$, $s\in K$ and $b\in \sB$.
Set $\sM=\sum_{(k, c)\in \sL}\sM(\sB, k, c)$.
Then $\sM(\sB, K, \sigma)=\sM+\sM^\dag$.

Note that
\begin{align*}
\big| [d]\setminus \bigcup_{1\le k\le \ell}F'_k\cC_k\big|&= \big| [d]\setminus \bigcup_{1\le k\le \ell}F_k\cC_k\big|+\big|\bigcup_{1\le k\le \ell}(F_k\setminus F'_k)\cC_k\big| \\
&\le 2\varepsilon d+ \sum_{1\le k\le \ell}|F_k\setminus F'_k|\cdot |\cC_k|\\
&\le 2\varepsilon d+\sum_{1\le k\le \ell}\varepsilon |F_k|\cdot |K|\cdot |\cC_k|\\
&= 2\varepsilon d+\varepsilon |K|\cdot \big|\bigcup_{1\le k\le \ell}F_k\cC_k\big|\\
&\le 2\varepsilon d+\varepsilon |K| d.
\end{align*}

Denote by $\phi$ the $R$-module homomorphism $\cM^d\rightarrow \bigoplus_{(k, c)\in \sL}\cM$ such that $\phi(\delta_v\cM)=0$ for all $v\in [d]\setminus \bigcup_{k=1}^\ell F_k\cC_k$
and $\phi(\delta_{tc}tx)$ takes value $x$ at $(k, c)$ and $0$ everywhere else for all $(k, c)\in \sL$, $t\in F_k$ and $x\in \cM$. Note that $\phi(\sA^d)=\bigoplus_{(k, c)\in \sL}\sA^{F_k}$ and $\phi(\sM)=0$. Thus
\begin{align*}
\rL(\sM(\sA, \sB, K, \sigma))&= \rL((\sA^d+\sM(\sB, K, \sigma))/\sM(\sB, K, \sigma))\\
&\ge \rL(\phi(\sA^d+\sM(\sB, K, \sigma))/\phi(\sM(\sB, K, \sigma)))\\
&= \rL((\bigoplus_{(k, c)\in \sL}\sA^{F_k}+\phi(\sM^\dag))/\phi(\sM^\dag))\\
&=\rL(\bigoplus_{(k, c)\in \sL}\sA^{F_k}+\phi(\sM^\dag))-\rL(\phi(\sM^\dag))\\
&\ge \rL(\bigoplus_{(k, c)\in \sL}\sA^{F_k})-\rL(\sM^\dag)\\
&\ge \sum_{(k, c)\in \sL}\rL(\sA^{F_k})-\big|[d]\setminus \bigcup_{1\le k\le \ell}F'_k\cC_k\big|\cdot |K|\cdot \rL(\sB)\\
&\ge \sum_{(k, c)\in \sL}|F_k|\mL(\sA)-(2+|K|)\varepsilon d\cdot |K|\cdot \rL(\sB)\\
&=\big| \bigcup_{1\le k\le \ell}F_k\cC_k\big|\mL(\sA)-(2+|K|)\varepsilon d\cdot |K|\cdot \rL(\sB)\\
&\ge d(1-2\varepsilon)\mL(\sA)-(2+|K|)\varepsilon d\cdot |K|\cdot \rL(\sB).
\end{align*}
Therefore
$$\mL_{\Sigma, \omega}(\sA|\sB, K)\ge (1-2\varepsilon)\mL(\sA)-(2+|K|)\varepsilon\cdot|K|\cdot \rL(\sB).$$
Letting $\varepsilon\to 0$, we get
$\mL_{\Sigma, \omega}(\sA|\sB, K)\ge \mL(\sA)$. Since $\sB\in \sF(\cM)$ and $K\in \cF(\Gamma)$ are arbitrary, we conclude that $\mL_{\Sigma, \omega}(\sA|\cM)\ge \mL(\sA)$.
\end{proof}

\begin{lemma} \label{L-amenable mean length upper}
For any $R\Gamma$-module $\cM$ and any $\sA\in \sF(\cM)$ with $\rL(\sA)<+\infty$, we have $\mL_{\Sigma, \omega}(\sA|\cM)\le \mL(\sA)$.
\end{lemma}
\begin{proof} Let $\delta>0$.  There exist nonempty $K\in \cF(\Gamma)$ and $0<\varepsilon<1/2$ such that for any nonempty $F\in \cF(\Gamma)$ with $|KF\setminus F|<\varepsilon |F|$, one has $\rL(\sA^F)\le |F|(\mL(\sA)+\delta)$. Then we have $\ell$ and $F_1, \dots, F_\ell$ in Lemma~\ref{L-Rokhlin}.
In particular,
$$\rL(\sA^{F_k})\le |F_k|(\mL(\sA)+\delta)$$
for all $1\le k\le \ell$.

Set $F=\bigcup_{k=1}^\ell F_k\in \cF(\Gamma)$ and $\sB=\sA^F\in \sF(\cM)$.

Let $\sigma: \Gamma\rightarrow \Sym(d)$ be a good enough sofic approximation for $\Gamma$, and set $\cW=[d]$. Then we have $\cC_1, \dots, \cC_\ell$ as in Lemma~\ref{L-Rokhlin}.
Denote by $\varphi$ the quotient map $\cM^d\rightarrow \cM^d/\sM(\sB, F, \sigma)$. Set $\sA'=\sum_{v\in [d]\setminus \bigcup_{k=1}^\ell F_k\cC_k}\delta_v \sA$. Let $1\le k\le \ell$ and $c\in \cC_k$.
Set $\sA^\sharp(k, c)=\delta_c(\sA^{F_k})$. Then $\sA^\sharp(k, c)$ is isomorphic to $\sA^{F_k}$ as $R$-modules. For  any $s\in F_k$ and $a\in \sA$, we have
$$\varphi(\delta_{sc}a)=\varphi(\delta_c s^{-1}a)\in \varphi(\sA^\sharp(k, c)).$$
Thus
$$\sM(\sA, \sB, F, \sigma)\subseteq \varphi(\sA')+\sum_{1\le k\le \ell}\sum_{c\in \cC_k}\varphi(\sA^\sharp(k, c)).$$
Therefore
\begin{align*}
\rL(\sM(\sA, \sB, F, \sigma))
&\le \rL(\varphi(\sA')+\sum_{1\le k\le \ell}\sum_{c\in \cC_k}\varphi(\sA^\sharp(k, c)))\\
&\le \rL(\sA')+\sum_{1\le k\le \ell}\sum_{c\in \cC_k}\rL(\sA^\sharp(k, c)) \\
&\le \rL(\sA)\cdot \big|[d]\setminus \bigcup_{1\le k\le \ell}F_k\cC_k\big|+\sum_{1\le k\le \ell}|\cC_k|\cdot \rL(\sA^{F_k})\\
&\le 2\rL(\sA)d\varepsilon+\sum_{1\le k\le \ell}|\cC_k|\cdot |F_k|\cdot(\mL(\sA)+\delta)\\
&\le 2\rL(\sA)d\varepsilon+d\cdot (\mL(\sA)+\delta).
\end{align*}
Thus
$$\mL_{\Sigma, \omega}(\sA|\cM)\le \mL_{\Sigma, \omega}(\sA|\sB, F)\le 2\rL(\sA)\varepsilon+\mL(\sA)+\delta.$$
Letting $\varepsilon\to 0$ and then $\delta\to 0$,  we get $\mL_{\Sigma, \omega}(\sA|\cM)\le  \mL(\sA)$.
\end{proof}

\section{Finitely generated submodules of free modules} \label{S-fg submodule}

In this section we give a formula for the mean length of a finitely generated $R\Gamma$-module relative to a free $R\Gamma$-module.

Let $\sigma$ be a map from $\Gamma$ to $\Sym(d)$ for some $d\in \Nb$. For each $f\in R\Gamma$, we define an $R$-module homomorphism
$\bar{\sigma}_f: R^d\rightarrow R^d$ by
\begin{align} \label{E-linear map1}
 (\bar{\sigma}_f(w))_v=\sum_{s\in \Gamma}w_{sv}f_s
\end{align}
for all $v\in [d]$.
For any $m,n \in \Nb$ and $f\in M_{m, n}(R\Gamma)$, we then define an $R$-module homomorphism
$\bar{\sigma}_f: (R^d)^{1\times m}\rightarrow (R^d)^{1\times n}$ by
\begin{align} \label{E-linear map2}
\bar{\sigma}_f=(\bar{\sigma}_{f_{k, j}})_{1\le k\le m, 1\le j\le n}\in M_{m, n}(\End_R(R^d)).
\end{align}

\begin{proposition} \label{P-principal mean rank}
Let $f\in M_{m, n}(R\Gamma)$ for some $m, n\in \Nb$. Set $\cM_1=(R\Gamma)^{1\times m}f$ and $\cM_2=(R\Gamma)^{1\times n}$.
Suppose that $\rL(R)<\infty$.
Then
\begin{align*}
 \mL_{\Sigma, \omega}(\cM_1|\cM_2)=\lim_{i\to \omega}\frac{\rL(\im \bar{\sigma}_{i, f})}{d_i}.
\end{align*}
\end{proposition}
\begin{proof}
Denote by $A$ the set of all rows of $f$, and by $\sA$ the $R$-submodule of $\cM_1$ generated by $A$. Then $\sA$ generates $\cM_1$ as an $R\Gamma$-module. By Lemma~\ref{L-mean length generating} we have
$\mL_{\Sigma, \omega}(\cM_1|\cM_2)=\mL_{\Sigma, \omega}(\sA|\cM_2)$.

Denote by $F'$ the union of the supports of elements in $A$ (as finite subsets of $\Gamma$).
Let $e_1, \dots, e_n$ be the standard basis of $\cM_2$.
Denote by $\tilde{\sA}$ the $R$-submodule of $\cM_2$ generated by $e_1, \dots, e_n$.
Let $\sB\in \sF(\cM_2)$ contain $\tilde{\sA}$, and $F\in \cF(\Gamma)$ contain $F'$.
Let $\sigma$ be a map from $\Gamma$ to $\Sym(d)$ for some $d\in \Nb$. Denote by
$\varphi$ the quotient map $\cM_2^d\rightarrow \cM_2^d/\sM(\sB, F, \sigma)$.
For each $1\le k\le m$, denote by $f_k$ the $k$-th row of $f$, and write it as
$\sum_{j=1}^n\big(\sum_{s\in \Gamma}f_{k,j,s}s\big)e_j$ with $f_{k, j, s}\in R$.
Define a surjective $R$-module homomorphism $\phi: (R^d)^{1\times m}\rightarrow \sA^d$ by
$$\phi((w_{k, v})_{1\le k\le m, v\in [d]})=\sum_{k=1}^m\sum_{v\in [d]}w_{k,v}\delta_v f_k.$$
Also define an $R$-module isomorphism $\psi$ from $(R^d)^{1\times n}$ to $\tilde{\sA}^d$
sending $(w_{j, v})_{1\le j\le n, v\in [d]}$ to $\sum_{j=1}^n\sum_{v\in [d]}w_{j, v}\delta_v e_j$.
Let $w=(w_{k, v})_{1\le k\le m, v\in [d]}\in (R^d)^{1\times m}$.  We have
\begin{align*}
\varphi(\phi(w))
&=\varphi\big(\sum_{k=1}^m\sum_{v\in [d]}w_{k,v}\delta_vf_k\big)\\
&=\sum_{k=1}^m\sum_{v\in [d]}\sum_{j=1}^n\sum_{s\in F'}w_{k,v}f_{k, j, s}\varphi(\delta_v se_j)\\
&=\sum_{k=1}^m\sum_{v\in [d]}\sum_{j=1}^n\sum_{s\in F'}w_{k, v}f_{k, j, s}\varphi(\delta_{\sigma_s^{-1}(v)} e_j)\\
&=\sum_{k=1}^m\sum_{v\in [d]}\sum_{j=1}^n\sum_{s\in \Gamma}w_{k, sv}f_{k, j, s}\varphi(\delta_v e_j)\\
&=\varphi\big(\sum_{k=1}^m\sum_{v\in [d]}\sum_{j=1}^n\sum_{s\in \Gamma}w_{k, sv}f_{k, j, s}\delta_ve_j\big).
\end{align*}
For any $1\le j\le n$ and $v\in [d]$, we have
$$ (\bar{\sigma}_f(w))_{j, v}=\sum_{k=1}^m\big(\bar{\sigma}_{f_{k, j}}(w_k)\big)_v=\sum_{k=1}^m\sum_{s\in \Gamma}w_{k, sv}f_{k, j, s},$$
thus
$$\psi(\bar{\sigma}_f(w))=\sum_{j=1}^n\sum_{v\in [d]}\sum_{k=1}^m\sum_{s\in \Gamma}w_{k, sv}f_{k, j, s}\delta_v e_j.$$
Therefore
$$\varphi(\phi(w))=\varphi(\psi(\bar{\sigma}_f(w))).$$
Then
\begin{align} \label{E-principal mean rank3}
\sM(\sA, \sB, F, \sigma)=\varphi(\sA^d)=\varphi(\phi((R^d)^{1\times m}))=\varphi(\psi(\bar{\sigma}_f((R^d)^{1\times m})))=\varphi(\psi(\im \bar{\sigma}_f)).
\end{align}
From \eqref{E-principal mean rank3} we have
$$\rL(\sM(\sA, \sB, F, \sigma))=\rL(\varphi(\psi(\im\bar{\sigma})))\le \rL(\im \bar{\sigma}_f),$$
whence
\begin{align} \label{E-principal mean rank4}
\mL_{\Sigma, \omega}(\sA|\sB, F)\le \lim_{i\to \omega}\frac{\rL(\im \bar{\sigma}_{i, f})}{d_i}.
\end{align}
From \eqref{E-principal mean rank3} we also have
\begin{align*}
\rL(\sM(\sA, \sB, F, \sigma))&=\rL(\psi(\im \bar{\sigma}_f))-\rL(\psi(\im \bar{\sigma}_f)\cap \sM(\sB, F, \sigma))\\
&\ge \rL(\im \bar{\sigma}_f)-\rL(\tilde{\sA}^d\cap \sM(\sB, F, \sigma))\\
&=\rL(\im \bar{\sigma}_f)-\rL(\tilde{\sA}^d)+\rL(\sM(\tilde{\sA}, \sB, F, \sigma))\\
&=\rL(\im \bar{\sigma}_f)-nd\rL(R)+\rL(\sM(\tilde{\sA}, \sB, F, \sigma)),
\end{align*}
and hence
$$ \mL_{\Sigma, \omega}(\sA|\sB, F)\ge \lim_{i\to \omega}\frac{\rL(\im \bar{\sigma}_{i, f})}{d_i}-n\rL(R)+\mL_{\Sigma, \omega}(\tilde{\sA}|\sB, F).$$
Note that $\tilde{\sA}$ generates $\cM_2=R\Gamma\otimes_RR^n$ as an $R\Gamma$-module. Thus by Lemma~\ref{L-mean length generating} and Proposition~\ref{P-mean rank basic} we have $\mL_{\Sigma, \omega}(\tilde{\sA}|\cM_2)=\mL_{\Sigma, \omega}(\cM_2)=n\rL(R)$. Therefore $\mL(\tilde{\sA}|\sB, F)\ge n\rL(R)$, whence
\begin{align} \label{E-principal mean rank5}
\mL_{\Sigma, \omega}(\sA|\sB, F)\ge \lim_{i\to \omega}\frac{\rL(\im \bar{\sigma}_{i, f})}{d_i}.
\end{align}
Combining \eqref{E-principal mean rank4} and \eqref{E-principal mean rank5} we get
$\mL_{\Sigma, \omega}(\sA|\sB, F)=\lim_{i\to \omega}\frac{\rL(\im \bar{\sigma}_{i, f})}{d_i}$. By Lemma~\ref{L-mean length generating} we get
$\mL_{\Sigma, \omega}(\cM_1|\cM_2)=\mL_{\Sigma, \omega}(\sA|\cM_2)=\lim_{i\to \omega}\frac{\rL(\im \bar{\sigma}_{i, f})}{d_i}$.
\end{proof}

From Propositions~\ref{P-principal mean rank} and \ref{P-mean rank basic}, and Theorem~\ref{T-addition for mean length} we have the following consequence.

\begin{corollary} \label{C-subgroup}
Let $\Gamma'$ be a subgroup of $\Gamma$ and denote by $\Sigma'$ the restriction of $\Sigma$ to $\Gamma'$.
Let $f\in M_{m, n}(R\Gamma')$ for some $m, n\in \Nb$.
Suppose that $\rL(R)<\infty$.
Then
\begin{align*}
 \mL_{\Sigma, \omega}((R\Gamma)^{1\times m}f|(R\Gamma)^{1\times n})=\mL_{\Sigma', \omega}((R\Gamma')^{1\times m}f|(R\Gamma')^{1\times n}),
\end{align*}
and
\begin{align*}
\mL_{\Sigma, \omega}((R\Gamma)^{1\times n}/(R\Gamma)^{1\times m}f)=\mL_{\Sigma', \omega}((R\Gamma')^{1\times n}/(R\Gamma')^{1\times m}f).
\end{align*}
\end{corollary}

\begin{example} \label{E-free group}
Suppose that $\rL(R)=1$ and $s\in \Gamma$ has infinite order. Then the subgroup $\Gamma'$ of $\Gamma$ generated by $s$ is isomorphic to $\Zb$. Clearly $s-1$ is not a right zero-divisor in $R\Gamma'$, whence $R\Gamma' (s-1)\cong R\Gamma'$ as $R\Gamma'$-modules. By Corollary~\ref{C-subgroup} and Theorem~\ref{T-amenable mean length} we have
$$ \mL_{\Sigma, \omega}(R\Gamma (s-1)|R\Gamma)=\mL_{\Sigma', \omega}(R\Gamma' (s-1)|R\Gamma')=\mL(R\Gamma'(s-1))=\mL(R\Gamma')=1,$$
and
$$\mL_{\Sigma, \omega}(R\Gamma/R\Gamma (s-1))=\mL_{\Sigma', \omega}(R\Gamma'/R\Gamma'(s-1))=\mL(R\Gamma')-\mL(R\Gamma'(s-1))=0.$$
In particular, we can take $\Gamma$ to be the free group $\Fb_2$ with canonical generators $s$ and $t$. It is well known that
$R\Fb_2$ has a free submodule with generators $s-1$ and $t-1$ \cite[Corollary 10.3.7.(iv)]{Passman77}. Thus $R\Fb_2/R\Fb_2 (s-1)$
contains an $R\Fb_2$-submodule isomorphic to $R\Fb_2$, while $ \mL_{\Sigma, \omega}(R\Fb_2/R\Fb_2 (s-1))=0$.
\end{example}

\section{Mean rank and von Neumann-L\"{u}ck rank} \label{S-mrk vs vrk}

In this section we introduce the relative von Neumann-L\"{u}ck rank and study its relation with mean length.
Throughout this section, we fix $R$ to be a unital subring of $\Cb$.
As in Example~\ref{E-Ore}, we consider the length function on $R$-modules given by $\rk(\sM)=\dim_Q(Q\otimes_R\sM)$, where $Q$ denotes the fraction field for $R$. Then we have the relative mean length
$\mrk_{\Sigma, \omega}(\cM_1|\cM_2)$ defined in Definition~\ref{D-mean length} for all $R\Gamma$-modules $\cM_1\subseteq \cM_2$.

We shall use the notations in Section~\ref{SS-vdim}. For any $R\Gamma$-module $\cM$, its {\it von Neumann-L\"{u}ck rank} is defined as
$$\vrk(\cM):=\dim_{\cL\Gamma}(\cL\Gamma\otimes_{R\Gamma}\cM).$$

We shall show that $\mrk_{\Sigma, \omega}(\cM)=\vrk(\cM)$ for any $R\Gamma$-module $\cM$ when $R$ is contained in the algebraic closure $\bar{\Qb}$ of $\Qb$. The proof uses the relative mean length $\mrk_{\Sigma, \omega}(\cM_1|\cM_2)$.
Correspondingly, we must introduce a relative version of the von Neumann-L\"{u}ck rank.

\begin{definition} \label{D-vrk}
For any $R\Gamma$-modules $\cM_1\subseteq \cM_2$, denoting by $\phi$ the natural map $\cL\Gamma\otimes_{R\Gamma}\cM_1\rightarrow \cL\Gamma\otimes_{R\Gamma}\cM_2$, we define
the {\it von Neumann-L\"{u}ck rank of $\cM_1$ relative to $\cM_2$} as
$$\vrk(\cM_1|\cM_2):=\dim_{\cL\Gamma}\phi(\cL\Gamma\otimes_{R\Gamma}\cM_1).$$
\end{definition}

Note that when $\cM_1=\cM_2$, we have $\vrk(\cM_1|\cM_2)=\vrk(\cM_1)$. Now we can state our main result of this section.

\begin{theorem}  \label{T-mrk vs vrk}
Suppose that $R\subseteq \bar{\Qb}$.
For any $R\Gamma$-modules $\cM_1\subseteq \cM_2$, we have
$$\mrk_{\Sigma, \omega}(\cM_1|\cM_2)=\vrk(\cM_1|\cM_2).$$
In particular, for any $R\Gamma$-module $\cM$ we have
$\mrk_{\Sigma, \omega}(\cM)=\vrk(\cM)$.
\end{theorem}

From Theorem~\ref{T-mrk vs vrk} and Proposition~\ref{P-finite} we have the following application to the von Neumann-L\"{u}ck rank.

\begin{corollary} \label{C-finite vrk}
Suppose that $\Gamma$ is  infinite and  $R\subseteq \bar{\Qb}$.  Let $\cM$ be an $R\Gamma$-module with $\rk(\cM)<+\infty$. Then $\vrk(\cM)=0$.
\end{corollary}

For any $m,n \in \Nb$ and $f\in M_{m, n}(R\Gamma)$, denote by $\ker f$ the kernel of the bounded linear operator $M_f: (\ell^2(\Gamma))^{n\times 1}\rightarrow (\ell^2(\Gamma))^{m\times 1}$ sending $z$ to $fz$, and by $P_f$ the orthogonal projection from $(\ell^2(\Gamma))^{n\times 1}$ onto $\ker f$.
Taking direct sum of the right regular representation $r$ of $\Gamma$ on $\ell^2(\Gamma)$, we get the unitary representations $r^{\oplus n}$ and $r^{\oplus m}$ of $\Gamma$
on $(\ell^2(\Gamma))^{n\times 1}$ and $(\ell^2(\Gamma))^{m\times 1}$ respectively. Clearly $M_f\circ r^{\oplus n}_s=r^{\oplus m}_s\circ M_f$ for every $s\in \Gamma$. Thus $r^{\oplus n}_s(\ker f)=\ker f$, whence $P_f$ commutes with $r^{\oplus n}_s$ for every $s\in \Gamma$. It follows that $P_f\in M_n(\cL\Gamma)$.

\begin{lemma} \label{L-limit}
Suppose that $R\subseteq \bar{\Qb}$. Let $f\in M_{m, n}(R\Gamma)$ for some $m,n\in \Nb$. Then
$$ \lim_{i\to \omega}\frac{\rk(\im \bar{\sigma}_{i, f})}{d_i}=n-\tr_{\cL\Gamma}P_f,$$
where $\bar{\sigma}_{i, f}$ is defined in \eqref{E-linear map2}.
\end{lemma}
\begin{proof} Let $\sigma$ be a map from $\Gamma$ to $\Sym(d)$ for some $d\in \Nb$. The formulas \eqref{E-linear map1} and \eqref{E-linear map2} in fact define $\Cb$-linear maps $\Cb^d\rightarrow \Cb^d$ and $(\Cb^d)^{1\times m}\rightarrow (\Cb^d)^{1\times n}$ respectively. We still denote the latter by $\bar{\sigma}_f$. Then
\begin{align*}
\rk(\bar{\sigma}_f((R^d)^{1\times m}))&=\dim_\Cb(\bar{\sigma}_f((\Cb^d)^{1\times m})).
\end{align*}

For each $s\in \Gamma$, the permutation $\sigma_s$ induces a linear isomorphism $\Cb^d\rightarrow \Cb^d$, which we still denote by $\sigma_s$. Explicitly,
$(\sigma_sw)_{sv}=w_v$ for all $w\in \Cb^d$ and $v\in [d]$. We extend this map $\Gamma\rightarrow \End_\Cb(\Cb^d)$ linearly to
$R\Gamma\rightarrow \End_\Cb(\Cb^d)$, and denote the image of $h\in R\Gamma$ by $\sigma_h$. Explicitly,
$$ (\sigma_hw)_v=\sum_{s\in \Gamma, v'\in [d],sv'=v}h_sw_{v'}$$
for all $h\in R\Gamma$, $w\in \Cb^d$ and $v\in [d]$.
Now we extend the map $R\Gamma\rightarrow \End_\Cb(\Cb^d)$ to $M_{m, n}(R\Gamma)\rightarrow M_{m, n}(\End_\Cb(\Cb^d))=\Hom_\Cb((\Cb^d)^{n\times 1}, (\Cb^d)^{m\times 1})$, and still denote the image of $h\in M_{m,n}(R\Gamma)$ by $\sigma_h$.
Explicitly, for any $h\in M_{m,n}(R\Gamma)$ and $w\in (\Cb^d)^{n\times 1}$, one has
$$ (\sigma_hw)_k=\sum_{j=1}^n\sigma_{h_{k, j}}w_j$$
for all $1\le k\le m$,
and thus
$$(\sigma_hw)_{k, v}=\big(\sum_{j=1}^n\sigma_{h_{k, j}}w_j\big)_v=\sum_{j=1}^n(\sigma_{h_{k, j}}w_j)_v=\sum_{j=1}^n\sum_{s\in \Gamma, v'\in [d], sv'=v}h_{k, j, s}w_{j, v'}$$
for all $1\le k\le m$ and $v\in [d]$.

We have the canonical bilinear pairing between $(\Cb^d)^{1\times n}$ and $(\Cb^d)^{n\times 1}$ given by
$$\left<w', w\right>=\sum_{j=1}^n\left<w'_j, w_j\right>=\sum_{j=1}^n\sum_{v\in [d]}w'_{j, v}w_{j, v}$$
for all $w'\in (\Cb^d)^{1\times n}$ and $w\in (\Cb^d)^{n\times 1}$, and similarly the pairing between $(\Cb^d)^{1\times m}$ and $(\Cb^d)^{m\times 1}$. For any $u\in (\Cb^d)^{1\times m}$ and $w\in (\Cb^d)^{n\times 1}$, we have
\begin{align*}
(\bar{\sigma}_fu)_{j, v}=\big(\sum_{k=1}^m\bar{\sigma}_{f_{k, j}}(u_k)\big)_v=\sum_{k=1}^m\sum_{s\in \Gamma}u_{k, sv}f_{k, j, s}
\end{align*}
for all $1\le j\le n$ and $v\in [d]$,
and hence
\begin{align*}
\left<\bar{\sigma}_fu, w\right>&=\sum_{j=1}^n\sum_{v\in [d]}(\bar{\sigma}_f(u))_{j, v}w_{j, v}\\
&=\sum_{j=1}^n\sum_{v\in [d]}\sum_{k=1}^m\sum_{s\in \Gamma}u_{k, sv}f_{k, j, s}w_{j, v}\\
&=\sum_{j=1}^n\sum_{k=1}^m\sum_{v\in [d]}\sum_{s\in \Gamma, v'\in [d], sv'=v}u_{k, v}f_{k, j, s}w_{j, v'}\\
&=\sum_{k=1}^m\sum_{v\in [d]}u_{k, v}(\sigma_fw)_{k, v}\\
&=\left<u, \sigma_fw\right>.
\end{align*}
Thus $\left<\bar{\sigma}_f((\Cb^d)^{1\times m}), w\right>=0$ iff $\sigma_fw=0$. Therefore $\ker \sigma_f$ is naturally the dual vector space of $(\Cb^d)^{1\times n}/\bar{\sigma}_f((\Cb^d)^{1\times m})$. It follows that
\begin{align*}
 \dim_\Cb(\ker \sigma_f)&=\dim_\Cb((\Cb^d)^{1\times n}/\bar{\sigma}_f((\Cb^d)^{1\times m}))\\
 &=dn-\dim_\Cb(\bar{\sigma}_f((\Cb^d)^{1\times m})) \\
 &=dn-\rk(\bar{\sigma}_f((R^d)^{1\times m})).
\end{align*}
Since $R\subseteq \bar{\Qb}$, by \cite[Theorem 4.3.(iv)]{Thom} we have
$$\tr_{\cL\Gamma}P_f=\lim_{i\to \omega}\frac{\dim_\Cb(\ker \sigma_{i, f})}{d_i}.$$
It follows that
$$ \tr_{\cL\Gamma}P_f=n-\lim_{i\to \omega}\frac{\rk(\bar{\sigma}_{i, f}((R^{d_i})^{1\times m}))}{d_i}.$$
\end{proof}

The following lemma is \cite[Lemma 5.4]{LL}, which was stated only for the case $R=\Zb$, but whose proof works for any unital subring $R$ of $\Cb$.

\begin{lemma} \label{L-vrk for fp}
For any (not necessarily sofic) group $\Gamma$, any $m,n \in \Nb$ and any $f\in M_{m,n}(R\Gamma)$, setting $\cM=(R\Gamma)^{1\times n}/(R\Gamma)^{1\times m}f$, one has
$$ \tr_{\cL\Gamma}P_f=\vrk(\cM).$$
\end{lemma}

\begin{lemma} \label{L-mrk vs vrk fp}
Suppose that $R\subseteq \bar{\Qb}$. For any finitely presented $R\Gamma$-module $\cM$, one has
$$ \mrk_{\Sigma, \omega}(\cM)=\vrk(\cM).$$
\end{lemma}
\begin{proof} We have $\cM=(R\Gamma)^{1\times n}/(R\Gamma)^{1\times m}f$ for some $m,n \in \Nb$ and $f\in M_{m,n}(R\Gamma)$.
Set $\cM_1=(R\Gamma)^{1\times m}f$ and $\cM_2=(R\Gamma)^{1\times n}$.
From Proposition~\ref{P-principal mean rank} and Lemmas~\ref{L-limit} and \ref{L-vrk for fp} we get
$$ \mrk_{\Sigma, \omega}(\cM_1|\cM_2)=n-\tr_{\cL\Gamma}P_f=n-\vrk(\cM).$$
By Proposition~\ref{P-mean rank basic} and
Theorem~\ref{T-addition for mean length} we have
\begin{align*}
n&=\mrk_{\Sigma, \omega}(\cM_2)\\
&= \mrk_{\Sigma, \omega}(\cM_1|\cM_2)+\mrk_{\Sigma, \omega}(\cM)\\
&=n-\vrk(\cM) +\mrk_{\Sigma, \omega}(\cM).
\end{align*}
Thus $\mrk_{\Sigma, \omega}(\cM)=\vrk(\cM)$.
\end{proof}

\begin{lemma} \label{L-unique}
Let $\cR$ be a unital ring. Suppose that for any $\cR$-modules $\cM_1\subseteq \cM_2$ we assign a value $\rL(\cM_1|\cM_2)\in \Rb_{\ge 0}\cup \{+\infty\}$ with the following properties:
\begin{enumerate}[\upshape (i)]
\item $\rL(\cM_1|\cM_2)$ depends only on the isomorphism class of the pair $(\cM_1, \cM_2)$, i.e. if $\cM_1\subseteq \cM_2$ and $\cM'_1\subseteq \cM'_2$ are $\cR$-modules and there is an isomorphism $\varphi: \cM_2\rightarrow \cM'_2$ with $\varphi(\cM_1)=\cM'_1$, then $\rL(\cM_1|\cM_2)=\rL(\cM'_1|\cM'_2)$,

\item one has $\rL(\cM_2)=\rL(\cM_1|\cM_2)+\rL(\cM_2/\cM_1)$, where we set $\rL(\cM)=\rL(\cM|\cM)$ for all $\cR$-modules $\cM$,

\item $\rL(\cR^n)<+\infty$ for every $n\in \Nb$,

\item if $\{\cM_j'\}_{j\in \cJ}$ is an increasing net of $\cR$-submodules of $\cM_1$ with union $\cM_1$, then  $\rL(\cM_j'|\cM_2)\to \rL(\cM_1|\cM_2)$ as $j\to \infty$,

\item if $\cM_1$ is a finitely generated $\cR$-module, and $\{\cM_j'\}_{j\in \cJ}$ is an increasing net of $\cR$-submodules of $\cM_2$ containing $\cM_1$ with union $\cM_2$, then $\rL(\cM_1|\cM_j')\to \rL(\cM_1|\cM_2)$ as $j\to \infty$.
\end{enumerate}
Then $\rL$ is determined by the values $\rL(\cM)$ for all finitely presented $\cR$-modules $\cM$.
\end{lemma}
\begin{proof} We shall give an algorithm to compute $\rL(\cM_1|\cM_2)$ for all $\cR$-modules $\cM_1\subseteq \cM_2$, given the values of $\rL(\cM)$ for all finitely presented $\cR$-modules $\cM$.

Consider first the case $\cM_2$ is a finitely generated free $\cR$-module, and $\cM_1$ is a finitely generated submodule of $\cM_2$. Then both $\cM_2$ and $\cM_2/\cM_1$ are finitely presented $\cR$-modules. By the conditions (ii) and (iii), we have
$\rL(\cM_1|\cM_2)+\rL(\cM_2/\cM_1)=\rL(\cM_2)<+\infty$, whence $\rL(\cM_1|\cM_2)=\rL(\cM_2)-\rL(\cM_2/\cM_1)$ is  determined.

Next we consider the case $\cM$ is a finitely generated $\cR$-module. We can write $\cM$ as $\cM_2/\cM_1$ for some finitely generated free $\cR$-module $\cM_2$ and some submodule $\cM_1$ of $\cM_2$. Take an increasing net $\{\cM_j'\}_{j\in \cJ}$ of finitely generated submodules of $\cM_1$ with union $\cM_1$. We have determined $\rL(\cM_j'|\cM_2)$ already.
By the condition (iv), $\rL(\cM_1|\cM_2)=\lim_{j\to \infty} \rL(\cM_j'|\cM_2)$ is determined. By the conditions (ii) and (iii), we have
$\rL(\cM_1|\cM_2)+\rL(\cM_2/\cM_1)=\rL(\cM_2)<+\infty$, whence $\rL(\cM)=\rL(\cM_2/\cM_1)=\rL(\cM_2)-\rL(\cM_1|\cM_2)<+\infty$ is  determined.

Now we consider the case $\cM_1$ is a finitely generated submodule of some $\cR$-module $\cM_2$. Take an increasing net $\{\cM_j'\}_{j\in \cJ}$ of finitely generated submodules of $\cM_2$ containing $\cM_1$ with union $\cM_2$. Then both $\cM_j'$ and $\cM_j'/\cM_1$ are finitely generated, thus we have determined $\rL(\cM_j')$ and $\rL(\cM_j'/\cM_1)$ already and we know that $\rL(\cM_j'/\cM_1)<+\infty$. From the condition (ii) we get that
$\rL(\cM_1|\cM_j')=\rL(\cM_j')-\rL(\cM_j'/\cM_1)$ is determined. By the condition (v),
$\rL(\cM_1|\cM_2)=\lim_{j\to \infty}\rL(\cM_1|\cM_j')$ is determined.

Finally we consider arbitrary $\cR$-modules $\cM_1\subseteq \cM_2$.  Take an increasing net $\{\cM_j'\}_{j\in \cJ}$ of finitely generated submodules of $\cM_1$ with union $\cM_1$. We have determined $\rL(\cM_j'|\cM_2)$ already.
By the condition (iv), $\rL(\cM_1|\cM_2)=\lim_{j\to \infty} \rL(\cM_j'|\cM_2)$ is determined.
\end{proof}

\begin{lemma} \label{L-vrk basic}
The von Neumann-L\"{u}ck rank $\vrk(\cM_1|\cM_2)$ for $R\Gamma$-modules $\cM_1\subseteq \cM_2$ satisfies the hypotheses of Lemma~\ref{L-unique}.
\end{lemma}
\begin{proof} (i). This is obvious.

(ii). Denote by $\varphi$ the natural map $\cL\Gamma\otimes_{R\Gamma}\cM_1\rightarrow \cL\Gamma\otimes_{R\Gamma}\cM_2$.
Since the tensor functor is right exact \cite[Proposition 19.13]{AF}, the sequence
$$ \cL\Gamma\otimes_{R\Gamma}\cM_1\rightarrow \cL\Gamma\otimes_{R\Gamma}\cM_2\rightarrow \cL\Gamma\otimes_{R\Gamma}(\cM_2/\cM_1)\rightarrow 0$$
of $\cL\Gamma$-modules is exact.
In other words, we have the short exact sequence
$$0\rightarrow \varphi(\cL\Gamma\otimes_{R\Gamma}\cM_1)\rightarrow \cL\Gamma\otimes_{R\Gamma}\cM_2\rightarrow \cL\Gamma\otimes_{R\Gamma}(\cM_2/\cM_1)\rightarrow 0$$
of $\cL\Gamma$-modules. Since $\dim_{\cL\Gamma}$ is a length function on $\cL\Gamma$-modules, we get
\begin{align*}
\vrk(\cM_2)&=\dim_{\cL\Gamma}(\cL\Gamma\otimes_{R\Gamma}\cM_2)\\
&=\dim_{\cL\Gamma}(\varphi(\cL\Gamma\otimes_{R\Gamma}\cM_1))+\dim_{\cL\Gamma}(\cL\Gamma\otimes_{R\Gamma}(\cM_2/\cM_1))\\
&=\vrk(\cM_1|\cM_2)+\vrk(\cM_2/\cM_1).
\end{align*}

(iii). By Theorem~\ref{T-vdim} we have $\dim_{\cL\Gamma}(\cL\Gamma)=1$, thus
$ \vrk((R\Gamma)^n)=\dim_{\cL\Gamma}((\cL\Gamma)^n)=n\dim_{\cL\Gamma}(\cL\Gamma)=n$.

(iv).
Denote by  $\varphi$ the natural map $\cL\Gamma\otimes_{R\Gamma}\cM_1\rightarrow \cL\Gamma\otimes_{R\Gamma}\cM_2$ and by $\varphi_j$ the natural map $\cL\Gamma\otimes_{R\Gamma}\cM_j'\rightarrow \cL\Gamma\otimes_{R\Gamma}\cM_2$ for each $j\in \cJ$. Then $\{\varphi_j(\cL\Gamma\otimes_{R\Gamma}\cM_j')\}_{j\in \cJ}$ is an increasing net of $\cL\Gamma$-submodules of
$\varphi(\cL\Gamma\otimes_{R\Gamma}\cM_1)$ with union $\varphi(\cL\Gamma\otimes_{R\Gamma}\cM_1)$. Since $\dim_{\cL\Gamma}$ is a length function on $\cL\Gamma$-modules, we get
\begin{align*}
\vrk(\cM_1|\cM_2)&=\dim_{\cL\Gamma}(\varphi(\cL\Gamma\otimes_{R\Gamma}\cM_1))\\
&=\lim_{j\to \infty}\dim_{\cL\Gamma}(\varphi_j(\cL\Gamma\otimes_{R\Gamma}\cM_j'))=\lim_{j\to \infty}\vrk(\cM_j'|\cM_2).
\end{align*}

(v).
 Note that $\cL\Gamma\otimes_{R\Gamma}\cM_1$ is a finitely generated $\cL\Gamma$-module, and hence $\dim_{\cL\Gamma}(\cL\Gamma\otimes_{R\Gamma}\cM_1)<+\infty$ by the condition (iii) and the fact that $\dim_{\cL\Gamma}$ is a length function on $\cL\Gamma$-modules.
Denote by  $\varphi$ the natural map $\cL\Gamma\otimes_{R\Gamma}\cM_1\rightarrow \cL\Gamma\otimes_{R\Gamma}\cM_2$ and by $\varphi_j$ the natural map $\cL\Gamma\otimes_{R\Gamma}\cM_1\rightarrow \cL\Gamma\otimes_{R\Gamma}\cM_j'$ for each $j\in \cJ$. Since $\cM_2$ is the direct limit $\varinjlim \cM_j'$ and the tensor functor $\cL\Gamma\otimes_{R\Gamma}\cdot$  preserves direct limits \cite[Theorem 6.159]{Rotman}, we have $\cL\Gamma\otimes_{R\Gamma}\cM_2=\varinjlim \cL\Gamma\otimes_{R\Gamma}\cM_j'$. It follows that
$\{\ker \varphi_j\}_{j\in \cJ}$ is an increasing net of $\cL\Gamma$-submodules of $\ker \varphi$ with union $\ker \varphi$.
Because $\dim_{\cL\Gamma}$ is a length function on $\cL\Gamma$-modules, we get
\begin{align*}
\vrk(\cM_1|\cM_2)&=\dim_{\cL\Gamma}(\im \varphi)\\
&=\dim_{\cL\Gamma}(\cL\Gamma\otimes_{R\Gamma}\cM_1)-\dim_{\cL\Gamma}(\ker \varphi)\\
&=\lim_{j\to \infty} (\dim_{\cL\Gamma}(\cL\Gamma\otimes_{R\Gamma}\cM_1)-\dim_{\cL\Gamma}(\ker \varphi_j))\\
&=\lim_{j\to \infty} \dim_{\cL\Gamma}(\im \varphi_j)\\
&=\lim_{j\to \infty}\vrk(\cM_1|\cM_j').
\end{align*}
\end{proof}

By Theorem~\ref{T-addition for mean length}, Lemma~\ref{L-mean length generating} and Proposition~\ref{P-continuity} we know that the mean rank $\mrk_{\Sigma, \omega}(\cM_1|\cM_2)$ for $R\Gamma$-modules $\cM_1\subseteq \cM_2$ satisfies the hypotheses of Lemma~\ref{L-unique}. Then Theorem~\ref{T-mrk vs vrk} follows from Lemmas~\ref{L-mrk vs vrk fp}, \ref{L-unique} and \ref{L-vrk basic}.

\section{Relative sofic mean topological dimension} \label{S-relative mdim}

In this section we introduce the relative mean topological dimension and establish some basic properties. Throughout the rest of this paper,
$\Gamma$ will be a countable sofic group, and $\Sigma=\{\sigma_i: \Gamma\rightarrow \Sym(d_i)\}_{i\in \Nb}$ will be a sofic approximation sequence for $\Gamma$.

We recall first the mean topological dimension for actions of sofic groups defined in \cite{Li}.

For a finite open cover $\cU$ of a compact metrizable space $Z$, set
$$ \ord(\cU)=\max_{z\in Z}\sum_{U\in \cU}1_U(z)-1, \mbox{ and } \cD(\cU)=\inf_{\cV}\ord(\cV)$$
for $\cV$ ranging over all finite open covers of $Z$ finer than $\cU$, i.e. every element of $\cV$ is contained in some element of $\cU$. Then the {\it covering dimension} of $Z$ is defined as $\sup_\cU\cD(\cU)$ for $\cU$ ranging over all finite open covers of $Z$.

Let $\Gamma$  act continuously on a compact metrizable space $X$.

\begin{definition} \label{D-map space}
Let $\rho$ be a continuous pseudometric on $X$. Let $F\in \cF(\Gamma)$ and $\delta>0$.
Let $\sigma$ be a map from $\Gamma$ to $\Sym(d)$ for some $d\in \Nb$.
We define continuous pseudometrics $\rho_2$ and $\rho_\infty$ on $X^d$ by
\begin{align*}
\rho_2(\varphi, \psi)&=\big(\frac{1}{d}\sum_{v\in [d]}\rho(\varphi_v, \psi_v)^2\big)^{1/2}, \\
\rho_\infty(\varphi, \psi)&=\max_{v\in [d]}\rho(\varphi_v, \psi_v).
\end{align*}
We define $\Map(\rho, F, \delta, \sigma)$ to be the set of all maps $\varphi: [d]\rightarrow X$ satisfying
$\rho_2(s\varphi, \varphi s)\le \delta$ for all $s\in F$.
\end{definition}

We remark that in Definition~\ref{D-map space} we treat $\sigma$ as an approximate action of $\Gamma$ on $[d]$,  and then $\Map(\rho, F, \delta, \sigma)$ is intuitively the set of all approximately equivariant maps $[d]\rightarrow X$.

A continuous pseudometric $\rho$ on $X$ is called {\it dynamically generating} if for any distinct $x_1, x_2\in X$ one has $\sup_{s\in \Gamma}\rho(sx_1, sx_2)>0$.

The following lemma is \cite[Lemma 2.3]{KL13b}. It essentially says that  the space $\Map(\rho, F, \delta, \sigma)$ does not depend on the choice of the dynamically generating continuous pseudometric $\rho$.

\begin{lemma} \label{L-generating metric}
Let $\rho$ and $\rho'$ be continuous pseudometrics on $X$ such that $\rho'$ is dynamically generating. For any $F\in \cF(\Gamma)$ and any $\delta>0$, there exist $F'\in \cF(\Gamma)$ and $\delta'>0$ such that  for any sufficiently good sofic approximation $\sigma: \Gamma\rightarrow \Sym(d)$ one has $\Map(\rho', F', \delta', \sigma)\subseteq \Map(\rho, F, \delta, \sigma)$.
\end{lemma}

Note that $\Map(\rho, F, \delta, \sigma)$ is a closed subset of $X^{d}$. For a finite open cover $\cU$ of $X$, denote by $\cU^d$ the open cover of $X^d$ consisting of $\prod_{v\in [d]}U_v$, where $U_v\in \cU$ for each $v\in [d]$.
Consider the restriction $\cU^d|_{\Map(\rho, F, \delta, \sigma)}:=\{U\cap \Map(\rho, F, \delta, \sigma): U\in \cU^d\}$ of $\cU^d$ to $\Map(\rho, F, \delta, \sigma)$.
Denote $\cD(\cU^d|_{\Map(\rho, F, \delta, \sigma)})$
by $\cD(\cU, \rho, F, \delta, \sigma)$.

\begin{definition} \label{D-mdim}
Let $\rho$ be a dynamically generating continuous pseudometric on $X$. Let $F\in \cF(\Gamma)$ and $\delta>0$. For a finite open cover $\cU$ of $X$ we define
\begin{align*}
\mdim_{\Sigma, \omega}(\cU, \rho, F, \delta)&=\lim_{i\to \omega}\frac{\cD(\cU, \rho, F, \delta, \sigma_i)}{d_i}, \\
\mdim_{\Sigma, \omega}(\cU, \rho)&=\inf_{F\in \cF(\Gamma)}\inf_{\delta>0}\mdim_{\Sigma, \omega}(\cU, \rho, F, \delta).
\end{align*}
If the set of $i\in \Nb$ with $\Map(\rho, F, \delta, \sigma_i)=\emptyset$ is in $\omega$, we set $\mdim_{\Sigma, \omega}(\cU, \rho, F, \delta)=-\infty$. The {\it sofic mean topological dimension of $\Gamma\curvearrowright X$} is defined as
$$ \mdim_{\Sigma, \omega}(X, \rho):=\sup_\cU \mdim_{\Sigma, \omega}(\cU, \rho)$$
for $\cU$ ranging over all finite open covers  of $X$. By Lemma~\ref{L-generating metric} the quantities $\mdim_{\Sigma, \omega}(\cU, \rho)$ and $\mdim_{\Sigma, \omega}(X, \rho)$ do not depend on the choice of $\rho$, and we shall write them as $\mdim_{\Sigma, \omega}(\cU)$ and $\mdim_{\Sigma, \omega}(X)$ respectively.
\end{definition}

\begin{remark} \label{R-ultrafilter}
The definition of sofic mean topological dimension in \cite[Definition 2.4]{Li} uses  $\varlimsup_{i\to \infty}$ instead of $\lim_{i\to \omega}$, and the resulting quantities will be denoted by $ \mdim_{\Sigma}(\cU, \rho, F, \delta), \mdim_{\Sigma}(\cU)$ and $\mdim_{\Sigma}(X)$ respectively. In this paper we use $\lim_{i\to \omega}$ instead of $\varlimsup_{i\to \infty}$ for three reasons. The first is all the results known for using $\varlimsup_{i\to \infty}$ hold for using $\lim_{i\to \omega}$ with the same proof.
The second is
that $\lim_{i\to \omega}$ is needed in Definition~\ref{D-mean length} in order to obtain the addition formula in Theorem~\ref{T-addition for mean length}, and then correspondingly we must use $\lim_{i\to \omega}$  in Definition~\ref{D-mdim}. The third is that in fact one can recover $\mdim_{\Sigma}(X)$ from  $\mdim_{\Sigma, \omega}(X)$ for all free ultrafilters $\omega$ on $\Nb$, as Proposition~\ref{P-ultrafilter} below shows. In particular, if for some action $\Gamma\curvearrowright X$, $\mdim_{\Sigma, \omega}(X)$ takes the same value for all $\omega$, then so does $\mdim_{\Sigma}(X)$.
\end{remark}

\begin{proposition} \label{P-ultrafilter}
For any continuous action of $\Gamma$ on a compact metrizable space $X$, one has $\mdim_\Sigma(X)=\sup_\omega\mdim_{\Sigma, \omega}(X)$ for $\omega$ ranging over all free ultrafilters on $\Nb$.
\end{proposition}
\begin{proof}
For any free ultrafilter $\omega$ on $\Nb$, clearly one has $\mdim_{\Sigma}(\cU)\ge \mdim_{\Sigma, \omega}(\cU)$ for every finite open cover $\cU$ of $X$, and hence $\mdim_{\Sigma}(X)\ge \mdim_{\Sigma, \omega}(X)$.

To show $\mdim_{\Sigma}(X)\le \sup_\omega \mdim_{\Sigma, \omega}(X)$, it suffices to show that
for any  finite open cover $\cU$ of $X$,  there exists a free ultrafilter $\omega$ on $\Nb$ with $\mdim_{\Sigma}(\cU)=\mdim_{\Sigma, \omega}(\cU)$.
Take an increasing sequence $\{F_n\}_{n\in \Nb}$ in $\cF(\Gamma)$ with union $\Gamma$ and a decreasing sequence $\{\delta_n\}_{n\in \Nb}$ of positive numbers converging to $0$.
We can find a strictly increasing sequence $\{i_n\}_{n\in \Nb}$ of positive integers such that $\frac{\cD(\cU, \rho, F_n, \delta_n, \sigma_{i_n})}{d_{i_n}}\ge \mdim_\Sigma(\cU)-\frac{1}{n}$ for all $n\in \Nb$.

For any infinite subset $W$ of $\Nb$, the set $\fF$ consisting of all subsets $V$ of $\Nb$ satisfying $|W\setminus V|<+\infty$ is a filter over $\Nb$. By Zorn's lemma we can find  a maximal filter $\omega$ over $\Nb$ containing $\fF$. Then $\omega$ is a free ultrafilter on $\Nb$ with $W\in \omega$.

In particular, we can find a free ultrafilter $\omega$ on $\Nb$ with $\{i_n: n\in \Nb\}\in \omega$. Then $\mdim_{\Sigma, \omega}(\cU, \rho, F_n, \delta_n)\ge \mdim_\Sigma(\cU)-\frac{1}{n}$ for all $n\in \Nb$. It follows that $\mdim_{\Sigma, \omega}(\cU)\ge \mdim_\Sigma(\cU)$, and hence $\mdim_{\Sigma, \omega}(\cU)=\mdim_\Sigma(\cU)$.
\end{proof}

Now we define the sofic mean topological dimension of an action relative to an extension. Let $\Gamma$ act on compact metrizable spaces $X$ and $Y$ respectively, and let $\pi: X\rightarrow Y$ be a $\Gamma$-equivariant continuous surjective map. Then $\pi$ is called a {\it factor map}, $\Gamma \curvearrowright Y$ is called a {\it factor} of $\Gamma\curvearrowright X$, and $\Gamma\curvearrowright X$ is called an {\it extension} of $\Gamma\curvearrowright Y$.

Let $\rho$ be a continuous pseudometric on $X$. Let $F\in \cF(\Gamma)$ and $\delta>0$.
Let $\sigma$ be a map from $\Gamma$ to $\Sym(d)$ for some $d\in \Nb$.
Denote by $\Map(\pi, \rho, F, \delta, \sigma)$ the set of all $\pi\circ \varphi$ for $\varphi$ ranging in $\Map(\rho, F, \delta, \sigma)$.

Note that $\Map(\pi, \rho, F, \delta, \sigma)$ is a closed subset of $Y^{d}$. For a finite open cover $\cU$ of $Y$,
consider the restriction $\cU^d|_{\Map(\pi, \rho, F, \delta, \sigma)}:=\{U\cap \Map(\pi, \rho, F, \delta, \sigma): U\in \cU^d\}$ of $\cU^d$ to $\Map(\pi, \rho, F, \delta, \sigma)$.
Denote $\cD(\cU^d|_{\Map(\pi, \rho, F, \delta, \sigma)})$
by $\cD(\pi, \cU, \rho, F, \delta, \sigma)$.

\begin{definition} \label{D-relative mdim}
Let $\rho$ be a dynamically generating continuous pseudometric on $X$. Let $F\in \cF(\Gamma)$ and $\delta>0$. For a finite open cover $\cU$ of $Y$ we define
\begin{align*}
\mdim_{\Sigma, \omega}(\pi, \cU, \rho, F, \delta)&=\lim_{i\to \omega}\frac{\cD(\pi, \cU, \rho, F, \delta, \sigma_i)}{d_i}, \\
\mdim_{\Sigma, \omega}(\pi, \cU, \rho)&=\inf_{F\in \cF(\Gamma)}\inf_{\delta>0}\mdim_{\Sigma, \omega}(\pi, \cU, \rho, F, \delta).
\end{align*}
If the set of $i\in \Nb$ with $\Map(\rho, F, \delta, \sigma_i)=\emptyset$ is in $\omega$, we set $\mdim_{\Sigma, \omega}(\pi, \cU, \rho, F, \delta)=-\infty$. We define the {\it sofic mean topological dimension of $\Gamma\curvearrowright Y$ relative to the extension $\Gamma\curvearrowright X$} as
$$ \mdim_{\Sigma, \omega}(Y|X, \rho):=\sup_\cU \mdim_{\Sigma, \omega}(\pi, \cU, \rho)$$
for $\cU$ ranging over all finite open covers  of $Y$.
By Lemma~\ref{L-generating metric} the quantities $\mdim_{\Sigma, \omega}(\pi, \cU, \rho)$ and $\mdim_{\Sigma, \omega}(Y|X, \rho)$ do not depend on the choice of $\rho$, and we shall write them as $\mdim_{\Sigma, \omega}(\cU|X)$ and $\mdim_{\Sigma, \omega}(Y|X)$ respectively.
We also define the {\it lower sofic mean topological dimension of $\Gamma\curvearrowright Y$ relative to the extension $\Gamma\curvearrowright X$} as
$$ \underline{\mdim}_{\Sigma, \omega}(Y|X):=\sup_\cU \mdim_{\Sigma, \omega}(\pi^{-1}(\cU))$$
for $\cU$ ranging over all finite open covers  of $Y$.
\end{definition}

Clearly when $\pi$ is a homeomorphism, one has
$$\underline{\mdim}_{\Sigma, \omega}(Y|X)=\mdim_{\Sigma, \omega}(Y|X)=\mdim_{\Sigma, \omega}(Y).$$
In general one has the following obvious proposition.

\begin{proposition} \label{P-compare with absolute}
One has $\underline{\mdim}_{\Sigma, \omega}(Y|X)\le \mdim_{\Sigma, \omega}(Y|X)\le \mdim_{\Sigma, \omega}(Y)$ and $\underline{\mdim}_{\Sigma, \omega}(Y|X)\le \mdim_{\Sigma, \omega}(X)$.
\end{proposition}

Now we discuss the behaviour of the relative mean topological dimension under taking inverse limits.
For a sequence of topological spaces $\{Y_j\}_{j\in \Nb}$ with continuous maps $p_j: Y_{j+1}\rightarrow Y_j$ for all $j\in \Nb$, the inverse limit $\varprojlim_{j\to \infty} Y_j$ is defined as the subspace of $\prod_{j\in \Nb} Y_j$ consisting of all elements $(y_j)_{j\in \Nb}$ satisfying $p_j(y_{j+1})=y_j$ for all $j\in \Nb$.

\begin{proposition} \label{P-topology inverse limit}
Let $\{Y_j\}_{j\in \Nb}$ be a sequence of compact metrizable spaces carrying continuous $\Gamma$-actions and factor maps $p_j: Y_{j+1}\rightarrow Y_j$ for all $j\in \Nb$.
Set $Y=\varprojlim_{j\to \infty}Y_j$.
Let $X$ be a compact metrizable space carrying a continuous $\Gamma$-action and $\pi: X\rightarrow Y$ be a factor map. Consider the factor map $X\rightarrow Y\rightarrow Y_j$.
Then
$$\mdim_{\Sigma, \omega}(Y|X)\le \varlimsup_{j\to \infty}\mdim_{\Sigma, \omega}(Y_j|X),$$
and
$$ \underline{\mdim}_{\Sigma, \omega}(Y|X)=\lim_{j\to \infty}\underline{\mdim}_{\Sigma, \omega}(Y_j|X)=\sup_{j\in \Nb}\underline{\mdim}_{\Sigma, \omega}(Y_j|X).$$
\end{proposition}
\begin{proof} Clearly $\underline{\mdim}_{\Sigma, \omega}(Y_j|X)$ increases with $j$ and $\underline{\mdim}_{\Sigma, \omega}(Y|X)\ge \lim_{j\to \infty}\underline{\mdim}_{\Sigma, \omega}(Y_j|X)$.

Denote by $q_j$ the factor map $Y\rightarrow Y_j$, and set $\pi_j:=q_j\circ \pi$.
For any finite open cover $\cU_j$ of $Y_j$, clearly one has
$$ \mdim_{\Sigma, \omega}(\pi_j^{-1}(\cU_j))\le \mdim_{\Sigma, \omega}(q_j^{-1}(\cU_j)|X)\le \mdim_{\Sigma, \omega}(\cU_j|X).$$
For any finite open cover $\cU$ of $Y$ and any $N\in \Nb$, there exists a finite open cover $\cU_j$ of $Y_j$ for some $j\ge N$ such that
$q_j^{-1}(\cU_j)$ refines $\cU$, whence
$$ \mdim_{\Sigma, \omega}(\cU|X)\le \mdim_{\Sigma, \omega}(q_j^{-1}(\cU_j)|X)\le \mdim_{\Sigma, \omega}(\cU_j|X)\le\mdim_{\Sigma, \omega}(Y_j|X),$$
and
$$\mdim_{\Sigma, \omega}(\pi^{-1}(\cU))\le \mdim_{\Sigma, \omega}(\pi_j^{-1}(\cU_j))\le \underline{\mdim}_{\Sigma, \omega}(Y_j|X)\le \lim_{k\to \infty}\underline{\mdim}_{\Sigma, \omega}(Y_k|X).$$
It follows that $\mdim_{\Sigma, \omega}(Y|X)\le \varlimsup_{k\to \infty}\mdim_{\Sigma, \omega}(Y_k|X)$ and  $\underline{\mdim}_{\Sigma, \omega}(Y|X)\le \lim_{k\to \infty}\underline{\mdim}_{\Sigma, \omega}(Y_k|X)$. Therefore $\underline{\mdim}_{\Sigma, \omega}(Y|X)=\lim_{k\to \infty}\underline{\mdim}_{\Sigma, \omega}(Y_k|X)$.
\end{proof}

Next we discuss the relation between the small-boundary property (SBP) and zero mean topological dimension. The notion of the SBP was introduced in \cite{Lindenstrauss, LW}.
We extend the formulation of the SBP in \cite[Definition 8.1]{Li} to relative situation.

\begin{definition} \label{D-SBP}
Denote by $M(X, \Gamma)$ the set of $\Gamma$-invariant Borel probability measures on $X$. We say that a closed subset $Z$ of $Y$ is {\it small relative to $\Gamma\curvearrowright X$} if $\pi_*\mu(Z)=0$ for every $\mu\in M(X, \Gamma)$. We say that $\Gamma\curvearrowright Y$ has the {\it small-boundary property (SBP) relative to $\Gamma\curvearrowright X$} if for every point $y$ in $Y$ and every neighborhood $U$ of $y$, there is an open neighborhood $V\subseteq U$ of $y$ such that the boundary of $V$ is small relative to $\Gamma\curvearrowright X$.
\end{definition}

The following result is the relative version of \cite[Theorem 8.2]{Li}. The proof of \cite[Theorem 8.2]{Li} also works here.

\begin{theorem} \label{T-SBP}
Suppose that $\Gamma\curvearrowright Y$ has the SBP relative to $\Gamma\curvearrowright X$. Then $\mdim_{\Sigma, \omega}(Y|X)\le 0$.
\end{theorem}

To round this section, we discuss what happens when $\Gamma$ is amenable.

We recall first the definition of the mean topological dimension of $\Gamma\curvearrowright X$ for amenable group $\Gamma$ in \cite{LW}. Let $\cU$ be a finite open cover of $X$. For any nonempty finite subset $F$ of $\Gamma$, we set $\cU^F=\bigvee_{s\in F}s^{-1}\cU$. The function $F\mapsto \cD(\cU^F)$ defined on the set of all nonempty finite subsets of $\Gamma$ satisfies the conditions of the Ornstein-Weiss lemma \cite{OW} \cite[Theorem 6.1]{LW}, whence $\frac{\cD(\cU^F)}{|F|}$ converges to some real number, denoted by $\mdim(\cU)$, as $F$ becomes more and more left invariant. That is, for any $\varepsilon>0$, there exist a nonempty finite subset $K$ of $\Gamma$ and $\delta>0$ such that for any nonempty finite subset $F$ of $\Gamma$ with $|KF\setminus F|<\delta |F|$, one has $\big|\frac{\cD(\cU^F)}{|F|}-\mdim(\cU)\big|<\varepsilon$. The {\it mean  topological dimension} of $\Gamma\curvearrowright X$ is defined as
$$\mdim(X):=\sup_\cU\mdim(\cU)$$
for $\cU$ ranging over all finite open covers of $X$ \cite[page 5]{LW}.

The following lemma is the relative version of \cite[Lemma 3.5]{Li}. The proof of \cite[Lemma 3.5]{Li} also works here.

\begin{lemma} \label{L-amenable case}
Suppose that $\Gamma$ is a countably infinite amenable group. Then for any finite open cover $\cU$ of $Y$ we have $\mdim_{\Sigma, \omega}(\cU|X)\ge \mdim(\cU)$. In particular, $\mdim_{\Sigma, \omega}(Y|X)\ge \mdim(Y)$.
\end{lemma}

By Proposition~\ref{P-compare with absolute} and \cite[Lemma 3.7]{Li}, when $\Gamma$ is amenable, we have
$$\mdim_{\Sigma, \omega}(Y|X)\le \mdim_{\Sigma, \omega}(Y)\le \mdim(Y).$$
Combining  these inequalities with Lemma~\ref{L-amenable case}, we get

\begin{theorem} \label{T-amenable case}
Suppose that $\Gamma$ is a countably infinite amenable group. Then
$$\mdim_{\Sigma, \omega}(Y|X)=\mdim_{\Sigma, \omega}(Y)=\mdim(Y).$$
\end{theorem}

\section{Relative sofic metric mean dimension} \label{S-relative metric mdim}

In this section we introduce relative topological entropy and relative metric mean dimension, and establish some basic properties.

We recall first the definition of sofic topological entropy in \cite[Definition 2.5]{KL13b} and sofic metric mean dimension in \cite[Definition 4.1]{Li}.

For any continuous pseudometric $\rho_X$ on a compact space $X$ and any $\varepsilon>0$,  a subset $Z$ of $X$ is called {\it $(\rho_X, \varepsilon)$-separated} if $\rho_X(z_1, z_2)>\varepsilon$ for all distinct $z_1, z_2\in Z$. We set $N_\varepsilon(X, \rho_X):=\max_Z |Z|$ for $Z$ ranging over all $(\rho_X, \varepsilon)$-separated subsets of $X$.

Let $\Gamma$  act continuously on a compact metrizable space $X$.

\begin{definition} \label{D-entropy}
Let $\rho$  be a continuous pseudometric of $X$. Let $F\in \cF(\Gamma)$ and $\delta>0$.  For $\varepsilon>0$ we define
\begin{align*}
h^\varepsilon_{\Sigma, \omega, \infty}(\rho, F, \delta)&=\lim_{i\to \omega}\frac{1}{d_i}\log N_\varepsilon(\Map(\rho, F, \delta, \sigma_i), \rho_\infty), \\
h^\varepsilon_{\Sigma, \omega, \infty}(\rho)&=\inf_{F\in \cF(\Gamma)}\inf_{\delta>0} h^\varepsilon_{\Sigma, \omega, \infty}(\rho, F, \delta),\\
h_{\Sigma, \omega, \infty}(\rho)&=\sup_{\varepsilon>0}h^\varepsilon_{\Sigma, \omega, \infty}(\rho).
\end{align*}
If the set of $i\in \Nb$ with $\Map(\rho, F, \delta, \sigma_i)=\emptyset$ is in $\omega$, we set $h^\varepsilon_{\Sigma, \omega, \infty}(\rho, F, \delta)=-\infty$.
We similarly define $h^\varepsilon_{\Sigma, \omega, 2}(\rho, F, \delta), h^\varepsilon_{\Sigma, \omega, 2}(\rho)$ and $h_{\Sigma, \omega, 2}(\rho)$ using
$N_\varepsilon(\cdot, \rho_2)$ instead of $N_\varepsilon(\cdot, \rho_\infty)$.
By \cite[Proposition 2.4]{KL13b} the quantities $h_{\Sigma, \omega, \infty}(\rho)$ and $h_{\Sigma, \omega, 2}(\rho)$ coincide, and  do not depend on the choice of dynamically generating continuous pseudometric $\rho$. They will be called the {\it sofic topological entropy of $\Gamma\curvearrowright X$} and denoted by $h_{\Sigma, \omega}(X)$. We define the {\it sofic metric mean dimension of $\Gamma\curvearrowright X$ with respect to $\rho$} as
$$ \mdim_{\Sigma, \omega, \rM}(X, \rho):=\varliminf_{\varepsilon \to 0}\frac{1}{|\log \varepsilon|} h^\varepsilon_{\Sigma, \omega, \infty}(\rho),$$
and the {\it sofic metric mean dimension of $\Gamma\curvearrowright X$} as
$$ \mdim_{\Sigma, \omega, \rM}(X):=\inf_\rho \mdim_{\Sigma, \omega, \rM}(X, \rho)$$
for $\rho$ ranging over all compatible metrics on $X$.
\end{definition}

\begin{remark} \label{R-ultrafilter entropy}
Remark~\ref{R-ultrafilter} and Proposition~\ref{P-ultrafilter} also apply to the sofic topological entropy, though it is not clear whether Proposition~\ref{P-ultrafilter} holds for
the sofic metric mean dimension.
\end{remark}

Now we define the sofic topological entropy and the sofic metric mean dimension of an action relative to an extension. Let $\Gamma$ act on compact metrizable spaces $X$ and $Y$ respectively, and let $\pi: X\rightarrow Y$ be a factor map.

\begin{definition} \label{D-relative metric mdim}
Let $\rho_X$ and $\rho_Y$ be continuous pseudometrics of $X$ and $Y$ respectively such that $\rho_X$ is dynamically generating. Let $F\in \cF(\Gamma)$ and $\delta>0$.  For $\varepsilon>0$ we define
\begin{align*}
h^\varepsilon_{\Sigma, \omega, \infty}(\rho_Y, F, \delta|\rho_X)&=\lim_{i\to \omega}\frac{1}{d_i}\log N_\varepsilon(\Map(\pi, \rho_X, F, \delta, \sigma_i), \rho_{Y, \infty}), \\
h^\varepsilon_{\Sigma, \omega, \infty}(\rho_Y|X)&=\inf_{F\in \cF(\Gamma)}\inf_{\delta>0} h^\varepsilon_{\Sigma, \omega, \infty}(\rho_Y, F, \delta|\rho_X),\\
h_{\Sigma, \omega, \infty}(\rho_Y|X)&=\sup_{\varepsilon>0}h^\varepsilon_{\Sigma, \omega, \infty}(\rho_Y|X).
\end{align*}
If the set of $i\in \Nb$ with $\Map(\rho_X, F, \delta, \sigma_i)=\emptyset$ is in $\omega$, we set $h^\varepsilon_{\Sigma, \omega, \infty}(\rho_Y, F, \delta|\rho_X)=-\infty$.
By Lemma~\ref{L-generating metric}, the quantities $h^\varepsilon_{\Sigma, \omega, \infty}(\rho_Y|X)$ and $h_{\Sigma, \omega, \infty}(\rho_Y|X)$  do not depend on the choice of $\rho_X$.
We similarly define $h^\varepsilon_{\Sigma, \omega, 2}(\rho_Y, F, \delta|\rho_X), h^\varepsilon_{\Sigma, \omega, 2}(\rho_Y|X)$ and $h_{\Sigma, \omega, 2}(\rho_Y|X)$ using
$N_\varepsilon(\cdot, \rho_{Y, 2})$ instead of $N_\varepsilon(\cdot, \rho_{Y, \infty})$.
By Lemma~\ref{L-entropy indep of metric} below, the quantities $h_{\Sigma, \omega, \infty}(\rho_Y|X)$ and $h_{\Sigma, \omega, 2}(\rho_Y|X)$ coincide, and do not depend on the choice of the dynamically generating continuous pseudometric $\rho_Y$ on $Y$. They will be called the {\it sofic topological entropy of $\Gamma\curvearrowright Y$ relative to $\Gamma\curvearrowright X$} and denoted by $h_{\Sigma, \omega}(Y|X)$.
We define the {\it sofic metric mean dimension of $\Gamma\curvearrowright Y$ relative to $\Gamma\curvearrowright X$ with respect to $\rho_Y$} as
$$ \mdim_{\Sigma, \omega, \rM}(Y, \rho_Y|X):=\varliminf_{\varepsilon \to 0}\frac{1}{|\log \varepsilon|} h^\varepsilon_{\Sigma, \omega, \infty}(\rho_Y|X),$$
and the {\it sofic metric mean dimension of $\Gamma\curvearrowright Y$ relative to $\Gamma\curvearrowright X$} as
$$ \mdim_{\Sigma, \omega, \rM}(Y|X):=\inf_{\rho_Y} \mdim_{\Sigma, \omega, \rM}(Y, \rho_Y|X)$$
for $\rho_Y$ ranging over all compatible metrics on $Y$.
\end{definition}

Clearly when $\pi$ is a homeomorphism, one has
$$h_{\Sigma, \omega}(Y|X)=h_{\Sigma, \omega}(Y),$$
and
$$\mdim_{\Sigma, \omega, \rM}(Y, \rho_Y|X)=\mdim_{\Sigma, \omega, \rM}(Y, \rho_Y)$$
for every dynamically generating continuous pseudometric $\rho_Y$ on $Y$, and
$$ \mdim_{\Sigma, \omega, \rM}(Y|X)=\mdim_{\Sigma, \omega, \rM}(Y).$$

\begin{remark} \label{R-entropy}
Similarly one can define the sofic entropy of a probability-measure-preserving action relative to an extension, which we shall explore in a subsequent paper.
\end{remark}

The following lemma is the relative version of \cite[Proposition 2.4]{KL13b}. The proof of \cite[Proposition 2.4]{KL13b} works here.

\begin{lemma} \label{L-entropy indep of metric}
Let $\rho_Y$ and $\rho'_Y$  be dynamically generating continuous pseudometrics on $Y$. Then
$$ h_{\Sigma, \omega, 2}(\rho_Y|X)=h_{\Sigma, \omega, 2}(\rho'_Y|X)=h_{\Sigma, \omega, \infty}(\rho'_Y|X)=h_{\Sigma, \omega, \infty}(\rho_Y|X).$$
\end{lemma}

The following lemma is the relative version of \cite[Lemma 4.4]{Li}. The proof of \cite[Lemma 4.4]{Li} also works here.

\begin{lemma} \label{L-pseudometric to metric}
Let $\rho_Y$ be a dynamically generating continuous pseudometric on $Y$. Enumerate the elements of $\Gamma$ as $s_1, s_2, \dots$. Define $\rho'_Y$ on $Y$ by $\rho'_Y(y, z)=\sum^\infty_{n=1}\frac{1}{2^n}\rho(s_ny, s_nz)$ for all $y, z\in Y$. Then $\rho'_Y$ is a compatible metric on $Y$. Furthermore, if $e_\Gamma=s_m$ then for any $\varepsilon>0$ one has
$$h^{4\varepsilon}_{\Sigma, \omega, \infty}(\rho'_Y|X)\le h^\varepsilon_{\Sigma, \omega, \infty}(\rho_Y|X)\le h^{\varepsilon/2^m}_{\Sigma, \omega, \infty}(\rho'_Y|X).$$
In particular, $\mdim_{\Sigma, \omega, \rM}(Y, \rho_Y|X)=\mdim_{\Sigma, \omega, \rM}(Y, \rho'_Y|X)$.
\end{lemma}

From Lemma~\ref{L-pseudometric to metric} we have the following relative version of \cite[Proposition 4.5]{Li}.

\begin{proposition} \label{P-pseudometric vs metric}
One has
$$ \mdim_{\Sigma, \omega, \rM}(Y|X)=\inf_{\rho_Y}\mdim_{\Sigma, \omega, \rM}(Y, \rho_Y|X)$$
for $\rho_Y$ ranging over all dynamically generating continuous pseudometrics on $Y$.
\end{proposition}

The following proposition is obvious.

\begin{proposition} \label{P-metric compare with absolute}
One has $h_{\Sigma,\omega}(Y|X)\le \min(h_{\Sigma, \omega}(X), h_{\Sigma, \omega}(Y))$ and $\mdim_{\Sigma, \omega, \rM}(Y, \rho_Y|X)\le \mdim_{\Sigma, \omega, \rM}(Y, \rho_Y)$ for every continuous pseudometric $\rho_Y$ on $Y$. In particular, $\mdim_{\Sigma, \omega, \rM}(Y|X)\le \mdim_{\Sigma, \omega, \rM}(Y)$.
\end{proposition}

The following result is the relative version of \cite[Theorem 6.1]{Li}. The proof of \cite[Theorem 6.1]{Li} also works here.

\begin{theorem} \label{T-comparison}
We have $\mdim_{\Sigma, \omega}(Y|X)\le \mdim_{\Sigma, \omega, \rM}(Y|X)$.
\end{theorem}

Next we discuss the behaviour of the relative metric mean dimension under taking inverse limits. The following lemma is the generalization of \cite[Lemma 7.7]{LL} from amenable groups to sofic groups and to the relative case. We follow the argument in the proof of \cite[Lemma 7.7]{LL}.

\begin{lemma} \label{L-metric inverse limit}
Let $\{Y_j\}_{j\in \Nb}$ be a sequence of compact metrizable spaces carrying continuous $\Gamma$-actions and factor maps $p_j: Y_{j+1}\rightarrow Y_j$ for all $j\in \Nb$.
Set $Y=\varprojlim_{j\to \infty}Y_j$.
Let $X$ be a compact metrizable space carrying a continuous $\Gamma$-action and $\pi: X\rightarrow Y$ be a factor map. Denote by $\pi_j$ the factor map $X\rightarrow Y\rightarrow Y_j$. Let $\rho_j$ be a continuous pseudometric of $Y_j$ for each $j\in \Nb$ such that $\rho_j(p_j(y), p_j(z))\le \rho_{j+1}(y, z)$ for all $y, z\in Y_{j+1}$. Then there is a decreasing sequence $\{\lambda_j\}_{j\in \Nb}$ of positive numbers such that the continuous pseudometric of $Y$ defined by
$$ \rho((y_j)_j, (z_j)_j)=\max_{j\in \Nb}\lambda_j\rho_j(y_j, z_j)$$
satisfies
$$ \mdim_{\Sigma, \omega, \rM}(Y, \rho|X)= \lim_{j\to \infty}\mdim_{\Sigma, \omega, \rM}(Y_j, \rho_j|X).$$
\end{lemma}
\begin{proof} We shall require that $\lambda_j\diam(Y_j, \rho_j)<1/2^j$ for all $j\in \Nb$, which implies that $\rho$ is a continuous pseudometric on $Y$. We shall also require $\lambda_{j+1}\le \lambda_j\le 1$ for all $j\in \Nb$.

Let $\rho_X$ be a compatible metric of $X$.

Let $k\in \Nb$. Define a continuous pseudometric $\rho'_k$ of $Y_k$ by
$$\rho'_k(y, z)=\max_{1\le j\le k}\lambda_j\rho_j(p_j\circ p_{j+1}\circ \cdots \circ p_{k-1}(y), p_j\circ p_{j+1}\circ \cdots \circ p_{k-1}(z)).$$
Then $\rho'_k\le \rho_k$.
Thus
$$ \mdim_{\Sigma, \omega, \rM}(Y_k, \rho'_k|X)\le \mdim_{\Sigma, \omega, \rM}(Y_k, \rho_k|X).$$
Then we can find some $0<\varepsilon_k<\frac{1}{k}$ such that
$$ \frac{1}{|\log \varepsilon_k|}h^{\varepsilon_k}_{\Sigma, \omega, \infty}(\rho'_k|X)\le \mdim_{\Sigma, \omega, \rM}(Y_k, \rho'_k|X)+\frac{1}{k}\le \mdim_{\Sigma, \omega, \rM}(Y_k, \rho_k|X)+\frac{1}{k}.$$
Denote by $q_k$ the natural factor map $Y\rightarrow Y_k$.
Note that
$$ \rho(y, z)\le \max(\rho'_k(q_k(y), q_k(z)), \max_{j>k}\lambda_j\diam(Y_j, \rho_j))$$
for all $y, z\in Y$. It follows that for any $d\in \Nb$, one has
$$ \rho_\infty(\varphi, \psi)\le \max((\rho'_k)_\infty(q_k\circ \varphi, q_k\circ \psi), \max_{j>k}\lambda_j\diam(Y_j, \rho_j))$$
for all $\varphi, \psi\in Y^d$. Thus for any $\varepsilon>\max_{j>k}\lambda_j\diam(Y_j, \rho_j)$, if $\Phi\subseteq Y^d$ is $(\rho_\infty, \varepsilon)$-separated, then
$q_k\circ \Phi:=\{q_k\circ \varphi: \varphi\in \Phi\}$ is $((\rho'_k)_\infty, \varepsilon)$-separated. Therefore, if $\max_{j>k}\lambda_j\diam(Y_j, \rho_j)<\varepsilon_k$, then
$$ N_{\varepsilon_k}(\Map(\pi, \rho_X, F, \delta, \sigma), \rho_\infty)\le N_{\varepsilon_k}(\Map(\pi_k, \rho_X, F, \delta, \sigma), (\rho'_k)_\infty)$$
for all $F\in \cF(\Gamma)$, $\delta>0$ and $\sigma: \Gamma\rightarrow \Sym(d)$, and hence
\begin{align*}
\frac{1}{|\log \varepsilon_k|}h^{\varepsilon_k}_{\Sigma, \omega, \infty}(\rho|X)&\le \frac{1}{|\log \varepsilon_k|}h^{\varepsilon_k}_{\Sigma, \omega, \infty}(\rho'_k|X)\le \mdim_{\Sigma, \omega, \rM}(Y_k, \rho_k|X)+\frac{1}{k}.
\end{align*}

Now we require further that $\max_{j>k}\lambda_j\diam(Y_j, \rho_j)<\varepsilon_k$ for all $k\in \Nb$. Since we can choose $\varepsilon_k$ once $\lambda_1, \dots, \lambda_k$ are given, by induction such a sequence $\{\lambda_j\}_{j\in \Nb}$ exists. Then
\begin{align*}
\mdim_{\Sigma, \omega, \rM}(Y, \rho|X)&\le \varliminf_{k\to \infty}\frac{1}{|\log \varepsilon_k|}h^{\varepsilon_k}_{\Sigma, \omega, \infty}(\rho|X)\\
&\le \varliminf_{k\to \infty}(\mdim_{\Sigma, \omega, \rM}(Y_k, \rho_k|X)+\frac{1}{k})\\
&=\varliminf_{k\to \infty}\mdim_{\Sigma, \omega, \rM}(Y_k, \rho_k|X).
\end{align*}

For any $k\in \Nb$, $\varepsilon>0$, $F\in \cF(\Gamma)$, $\delta>0$ and $\sigma: \Gamma\rightarrow \Sym(d)$, it is clear that
$$N_{\lambda_k\varepsilon}(\Map(\pi, \rho_X, F, \delta, \sigma), \rho_\infty)\ge  N_\varepsilon(\Map(\pi_k, \rho_X, F, \delta, \sigma), (\rho_k)_\infty),$$
and hence
$$h^{\lambda_k\varepsilon}_{\Sigma, \omega, \infty}(\rho|X)\ge h^\varepsilon_{\Sigma, \omega, \infty}(\rho_k|X).$$
It follows that $\mdim_{\Sigma, \omega, \rM}(Y, \rho|X)\ge \mdim_{\Sigma, \omega, \rM}(Y_k, \rho_k|X)$. Similarly, one has
$\mdim_{\Sigma, \omega, \rM}(Y_j, \rho_j|X)\ge \mdim_{\Sigma, \omega, \rM}(Y_k, \rho_k|X)$ for all $j\ge k$. Consequently,
$\mdim_{\Sigma, \omega, \rM}(Y, \rho|X)= \lim_{k\to \infty}\mdim_{\Sigma, \omega, \rM}(Y_k, \rho_k|X)$.
\end{proof}

To round this section, we discuss what happens when $\Gamma$ is  amenable.

We recall first the definition of the metric mean dimension of $\Gamma\curvearrowright X$ for amenable groups $\Gamma$ in \cite{LW}. Let $\rho$ be a continuous pseudometric of $X$.
For a finite open cover $\cU$ of $X$, we define the {\it mesh of $\cU$ with respect to $\rho$} as
$$ \mesh(\cU, \rho):=\max_{U\in \cU}\diam(U, \rho).$$
For a nonempty finite subset $F$ of $\Gamma$, we define a continuous pseudometric $\rho_F$ of $X$ by
$$ \rho_F(x_1, x_2):=\max_{s\in F}\rho(sx_1, sx_2).$$
For any $\varepsilon>0$, the function $F\mapsto \log \min_{\mesh(\cU, \rho_F)<\varepsilon}|\cU|$ defined on the set of all nonempty finite subsets of $\Gamma$ satisfies the conditions of the Ornstein-Weiss lemma \cite{OW} \cite[Theorem 6.1]{LW}, whence $\frac{1}{|F|}\log \min_{\mesh(\cU, \rho_F)<\varepsilon}|\cU|$ converges to some real number, denoted by $S(X, \varepsilon, \rho)$, as $F$ becomes more and more left invariant. The {\it metric mean dimension of $\Gamma\curvearrowright X$ with respect to $\rho$} \cite[page 13]{LW} is defined as
$$\mdim_\rM(X, \rho):=\varliminf_{\varepsilon\to 0}\frac{S(X, \varepsilon, \rho)}{|\log \varepsilon|}.$$
The {\it metric mean dimension of $\Gamma\curvearrowright X$} is defined as
$$\mdim_\rM(X):=\inf_\rho \mdim_\rM(X, \rho)$$
for $\rho$ ranging over all compatible metrics of $X$.

The following lemma is the relative version of \cite[Lemma 5.2]{Li}. The proof of \cite[Lemma 5.2]{Li} also works here.

\begin{lemma} \label{L-metric amenable case}
Suppose that $\Gamma$ is amenable. Let $\rho_Y$ be a continuous pseudometric on $Y$. Then $\mdim_{\Sigma, \omega, \rM}(Y, \rho_Y|X)\ge \mdim_\rM(Y, \rho_Y)$.
\end{lemma}

By Proposition~\ref{P-metric compare with absolute} and \cite[Lemma 5.3]{Li} when $\Gamma$ is amenable, for any continuous pseudometric $\rho_Y$ on $Y$ we have
$$\mdim_{\Sigma, \omega, \rM}(Y, \rho_Y|X)\le \mdim_{\Sigma, \omega, \rM}(Y, \rho_Y)\le \mdim_\rM(Y, \rho_Y).$$
Combining these inequalities with Lemma~\ref{L-metric amenable case}, we get

\begin{theorem} \label{T-metric amenable case}
Suppose that $\Gamma$ is amenable. Then
$$\mdim_{\Sigma, \omega, \rM}(Y, \rho_Y|X)=\mdim_{\Sigma, \omega, \rM}(Y, \rho_Y)=\mdim_\rM(Y, \rho_Y)$$
for every continuous pseudometric $\rho_Y$ on $Y$. In particular,
$$\mdim_{\Sigma, \omega, \rM}(Y|X)=\mdim_{\Sigma, \omega, \rM}(Y)=\mdim_\rM(Y).$$
\end{theorem}

\section{Mean dimension and mean rank} \label{S-mdim vs mrk}

In this section we prove Theorem~\ref{T-mdim vs mrk}.
Throughout the rest of this paper,  we take $R=\Zb$, and consider the length function on $\Zb$-modules given by $\rk(\sM)=\dim_{\Qb}(\Qb\otimes_{\Zb}\sM)$.

For any countable $\Zb\Gamma$-module $\cM$, its Pontryagin dual $\widehat{\cM}$ consisting of all group homomorphisms $\cM\rightarrow \Rb/\Zb$ is a compact metrizable abelian group, and $\Gamma$ has an induced action on $\widehat{\cM}$ by continuous automorphisms. Explicitly,
$$ (sx)(a)=x(s^{-1}a)$$
for all $x\in \widehat{\cM}$, $a\in \cM$ and $s\in \Gamma$. In fact, all the actions of $\Gamma$ on compact metrizable abelian groups by continuous automorphisms, called {\it algebraic actions}, arise this way \cite{Schmidt}. A pseudometric $\rho$ on a  group $X$  is called {\it translation-invariant} if $\rho(x_1x_3, x_2x_3)=\rho(x_3x_1, x_3x_2)=\rho(x_1, x_2)$ for all $x_1, x_2, x_3\in X$.

\begin{theorem} \label{T-mdim vs mrk}
Let $\cM_1\subseteq \cM_2$ be countable $\Zb\Gamma$-modules. Then
$$\mdim_{\Sigma, \omega, \rM}(\widehat{\cM_1}|\widehat{\cM_2})=\mdim_{\Sigma, \omega}(\widehat{\cM_1}|\widehat{\cM_2})=\underline{\mdim}_{\Sigma, \omega}(\widehat{\cM_1}|\widehat{\cM_2})=\mrk_{\Sigma, \omega}(\cM_1|\cM_2).$$
Furthermore, there exists a translation-invariant compatible  metric $\rho$ on $\widehat{\cM_1}$ with
\begin{align*}
\mdim_{\Sigma, \omega, \rM}(\widehat{\cM_1}, \rho|\widehat{\cM_2})=\mdim_{\Sigma, \omega}(\widehat{\cM_1}|\widehat{\cM_2}).
\end{align*}
\end{theorem}

Theorem~\ref{T-mdim vs vrk} follows from Theorems~\ref{T-mrk vs vrk} and \ref{T-mdim vs mrk}.

\begin{remark} \label{R-ultrafiler fix}
Since it is not clear whether Proposition~\ref{P-ultrafilter} holds for the sofic metric mean dimension, Theorem~\ref{T-mdim vs mrk} does not imply directly the version using
$\varlimsup_{i\to \infty}$ instead of $\lim_{i\to \omega}$. However, one can define the $\varlimsup_{i\to \infty}$ and $\varliminf_{i\to \infty}$ versions of the sofic mean length in Definition~\ref{D-mean length}, the relative sofic mean topological dimension in Definition~\ref{D-relative mdim}, and the relative sofic metric mean dimension in Definition~
\ref{D-relative metric mdim}. Then the addition formula in Theorem~\ref{T-addition for mean length} becomes several inequalities. In this way one can follow the proofs of Theorem~\ref{T-mrk vs vrk} and \ref{T-mdim vs mrk} to prove their $\varlimsup_{i\to \infty}$ and $\varliminf_{i\to \infty}$ versions.
\end{remark}

We show first that the mean topological dimension is no less than the mean rank.

\begin{lemma} \label{L-mdim vs mrank lower}
For any countable $\Zb\Gamma$-modules $\cM_1\subseteq \cM_2$, one has
\begin{align*}
\underline{\mdim}_{\Sigma, \omega}(\widehat{\cM_1}|\widehat{\cM_2})\ge \mrk_{\Sigma, \omega}(\cM_1|\cM_2).
\end{align*}
\end{lemma}
\begin{proof}Fix a compatible metric $\rho$ on $\widehat{\cM_2}$, and denote by $\pi$ the factor map $\widehat{\cM_2}\rightarrow \widehat{\cM_1}$.

Let $\sA\in \sF(\cM_1)$. Say, $\sA$ is generated by a finite set $A$ as an abelian group. Take a finite open cover $\cU$ of $\widehat{\cM_1}$ such that for any $a\in A$, no item of $\cU$ intersects both $a^{-1}(\Zb)$ and $a^{-1}(1/2+\Zb)$. Then it suffices to show that
$$ \mdim_{\Sigma, \omega}(\pi^{-1}(\cU))\ge \mrk_{\Sigma, \omega}(\sA|\cM_2).$$

Let $F\in \cF(\Gamma)$ and $\delta>0$. Take a finite subset $B'$ of $\cM_2$ such that for any $\xi, \zeta\in \widehat{\cM_2}$ with $\xi(b')=\zeta(b')$ for all $b'\in B'$, one has $\rho(\xi, \zeta)<\delta$. Set $\sB$ to be the subgroup of $\cM_2$ generated by $F^{-1}B'$. Let $\sigma$ be a map from $\Gamma$ to $\Sym(d)$ for some $d\in \Nb$.
Denote by $\varphi$ the quotient map $\cM_2^d\rightarrow \cM_2^d/\sM(\sB, F, \sigma)$.

Take a maximal subset $V$ of $[d]\times A$ subject to the condition that $W:=\{\varphi(\delta_v a): (v, a)\in V\}$ is linearly independent. Then $|V|=\rk(\sM(\sA, \sB, F, \sigma))$.
Consider the abelian group homomorphism $\phi: \cM_2^d/\sM(\sB, F, \sigma)\rightarrow \Qb\otimes_\Zb(\cM_2^d/\sM(\sB, F, \sigma))$ sending $w$ to $1\otimes w$. Then $\phi$ is injective on $W$, and $\phi(W)$ is linearly independent. Extending $\phi(W)$ to a basis of $\Qb\otimes_\Zb(\cM_2^d/\sM(\sB, F, \sigma))$, we find a linear map $\psi: \Qb\otimes_\Zb(\cM_2^d/\sM(\sB, F, \sigma))\rightarrow \rspan_\Qb(\phi(W))$ which is the identity map on $\rspan_\Qb(\phi(W))$. For each $\lambda=(\lambda_{v, a})_{v, a}\in [0, 1/2]^V$, we denote by  $\theta_\lambda$ the linear map $\rspan_\Qb(\phi(W))\rightarrow \Rb$ sending $\phi(\varphi(\delta_v a))$ to $\lambda_{v, a}$ for all $(v, a)\in V$, and by
$\Phi_\lambda$ the abelian group homomorphism $\cM_2^d\rightarrow \Rb/\Zb$ sending $h$ to $\theta_\lambda(\psi(\phi(\varphi(h))))+\Zb$. Then $\Phi_\lambda\in \widehat{\cM_2^d}=\widehat{\cM_2}^d$.
The map $\Phi: [0, 1/2]^V\rightarrow \widehat{\cM_2^d}$ sending $\lambda$ to $\Phi_\lambda$ is clearly continuous.
 If $\lambda$ and $\lambda'$ are different elements of $[0, 1/2]^V$, then
$\lambda_{v, a}\neq \lambda'_{v, a}$ for some $(v, a)\in V$, and hence $\Phi_\lambda(\delta_va)=\lambda_{v, a}+\Zb\neq \lambda'_{v, a}+\Zb=\Phi_{\lambda'}(\delta_va)$. Thus $\Phi$ is an embedding of $[0, 1/2]^V$ into $\widehat{\cM_2}^d$.

We claim that the image of $\Phi$ is contained in $\Map(\rho, F, \delta, \sigma)$.
Let $\lambda \in [0, 1/2]^V$. Let $v\in [d]$ and $s\in F$. For any $b'\in B'$, noting that $b:=s^{-1}b'$ is in $\sB$, we have
\begin{align*}
(s\Phi_{\lambda, v})(b')-\Phi_{\lambda, sv}(b')&=\Phi_{\lambda, v}(s^{-1}b')-\Phi_{\lambda, sv}(b')\\
&=\Phi_\lambda (\delta_v s^{-1}b')-\Phi_\lambda(\delta_{sv}b')\\
&=\Phi_\lambda(\delta_v b-\delta_{sv}sb)=0.
\end{align*}
By our choice of $B'$, we have $\rho(s\Phi_{\lambda, v}, \Phi_{\lambda, sv})<\delta$. Thus $\rho_\infty(s\Phi_\lambda, \Phi_\lambda s)<\delta$. Therefore $\Phi_\lambda\in \Map(\rho, F, \delta, \sigma)$. This proves our claim.

The pull back $\Phi^{-1}(\pi^{-1}(\cU)^d)$ is a finite open cover of $[0, 1/2]^V$. We claim that no atom of $\Phi^{-1}(\pi^{-1}(\cU)^d)$ intersects two opposing faces of $[0, 1/2]^V$. Suppose that some atom $\Phi^{-1}(U)$ of $\Phi^{-1}(\pi^{-1}(\cU)^d)$ contains two points $\lambda$ and $\lambda'$ in opposing faces of $[0, 1/2]^V$, where $U$ is an atom of $\pi^{-1}(\cU)^d$. Say, $U=\prod_{v'\in [d]}\pi^{-1}(U_{v'})$ with $U_{v'}\in \cU$ for each $v'\in [d]$, and $\lambda_{v, a}=0$ and $\lambda'_{v, a}=1/2$ for some $(v, a)\in V$.
Then
$$(\pi(\Phi_{\lambda, v}))(a)=\Phi_{\lambda, v}(a)=\theta_\lambda(\psi(\phi(\varphi(\delta_v a))))+\Zb=\lambda_{v, a}+\Zb=\Zb,$$
and
$$(\pi(\Phi_{\lambda', v}))(a)=\Phi_{\lambda', v}(a)=\theta_{\lambda'}(\psi(\phi(\varphi(\delta_va))))+\Zb=\lambda'_{v, a}+\Zb=1/2+\Zb. $$
Note that $\Phi_\lambda, \Phi_{\lambda'}\in U$, and hence $\pi(\Phi_{\lambda, v}), \pi(\Phi_{\lambda', v})\in U_v$. Thus $U_v$ intersects both $a^{-1}(\Zb)$ and $a^{-1}(1/2+\Zb)$, contradicting our choice of $\cU$. This proves our claim.

By \cite[Lemma 3.2]{LW} for any finite open cover $\cV$ of $[0, 1/2]^V$ with no atom intersecting two opposing faces of the cube $[0, 1/2]^V$ one has $\ord(\cV)\ge |V|$. Therefore
$$\cD(\pi^{-1}(\cU), \rho, F, \delta, \sigma)\ge \cD(\Phi^{-1}(\pi^{-1}(\cU)^d))\ge |V|=\rk(\sM(\sA, \sB, F, \sigma)).$$
It follows that $\mdim_{\Sigma, \omega}(\pi^{-1}(\cU))\ge \mrk_{\Sigma, \omega}(\sA|\cM_2)$.
\end{proof}

The following lemma establishes one direction of the addition formula for the metric mean dimension in the algebraic setting.

\begin{lemma} \label{L-mdim addition lower bound}
Let $X$ be a compact metrizable group, $G$ be a closed subgroup of $X$, and $Y$ be  the homogeneous space $X/G$. Let $\Gamma$ act on $X$ by continuous automorphisms preserving $G$, and consider the induced $\Gamma$-action on $Y$. Let $\rho_X$ be a translation-invariant dynamically generating continuous pseudometric  on $X$.
Let $\rho_G$ be the restriction of $\rho_X$ to $G$. Let $\rho_Y$ be a continuous pseudometric on $Y$ such that for some constant $C>0$ and $K\in \cF(\Gamma)$ one has
$ \rho_Y(\pi(x_1), \pi(x_2))\le C\max_{s\in K}\rho_X(sx_1, sx_2)$ for all $x_1, x_2\in X$, where $\pi: X\rightarrow Y$ is the quotient map.
Then
\begin{align*}
\mdim_{\Sigma, \omega, \rM}(X, \rho_X)\ge \mdim_{\Sigma, \omega, \rM}(Y, \rho_Y|X)+\mdim_{\Sigma, \omega, \rM}(G, \rho_G).
\end{align*}
\end{lemma}
\begin{proof}
Fix an enumeration $s_1, s_2, \dots$ of the elements of $\Gamma$. Define
$$\tilde{\rho}_X(x_1, x_2)=\sum_{j=1}^\infty\frac{1}{2^j}\rho_X(s_jx_1, s_jx_2)$$
for all $x_1, x_2\in X$. Then $\tilde{\rho}_X$ is a translation-invariant compatible metric on $X$, and there exists $C'>0$ such that
$$\rho_Y(\pi(x_1), \pi(x_2))\le C'\tilde{\rho}_X(x_1, x_2)$$
for all $x_1, x_2\in X$.
 Denote by $\tilde{\rho}_G$ the restriction of $\tilde{\rho}_X$ to $G$. By \cite[Lemma 4.4]{Li} we have
\begin{align} \label{E-mdim addition lower bound4}
\mdim_{\Sigma, \omega, \rM}(X, \rho_X)=\mdim_{\Sigma, \omega, \rM}(X, \tilde{\rho}_X), \mbox{ and } \mdim_{\Sigma, \omega, \rM}(G, \rho_G)=\mdim_{\Sigma, \omega, \rM}(G, \tilde{\rho}_G).
\end{align}

Let $\varepsilon>0$. Let $F\in \cF(\Gamma)$ and $\delta>0$.
Let $\sigma$ be a map $\Gamma\rightarrow \Sym(d)$ for some $d\in \Nb$.

Let $\Psi_Y$ be a $(\rho_{Y, \infty}, C'\varepsilon)$-separated subset of $\Map(\pi, \tilde{\rho}_X, F, \delta/2, \sigma)$ with
$$|\Psi_Y|=N_{C'\varepsilon}(\Map(\pi, \tilde{\rho}_X, F, \delta/2, \sigma), \rho_{Y, \infty}).$$
For each $\psi\in \Psi_Y$ take $\psi'\in \Map(\tilde{\rho}_X, F, \delta/2, \sigma)$ with $\psi=\pi\circ \psi'$. Set
$\Psi'_X=\{\psi': \psi\in \Psi_Y\}$.
Also let $\Phi_G$ be a $(\tilde{\rho}_{G, \infty}, \varepsilon)$-separated subset of $\Map(\tilde{\rho}_G, F, \delta/2, \sigma)$ with
$$|\Phi_G|=N_\varepsilon(\Map(\tilde{\rho}_G, F, \delta/2, \sigma), \tilde{\rho}_{G, \infty}).$$

Let $\psi'\in \Psi'_X$ and $\phi\in \Phi_G$. Denote by $\psi'\phi$ the map $[d]\rightarrow X$ sending $v$ to $\psi'(v)\phi(v)$.
For any $s\in F$,
we have
\begin{align*}
\sum_{v\in [d]}\tilde{\rho}_X(s(\psi'\phi(v)), \psi'\phi(sv))^2&=\sum_{v\in [d]}\tilde{\rho}_X((s\psi'(v))(s\phi(v)), \psi'(sv)\phi(sv))^2\\
&\le \sum_{v\in [d]}(\tilde{\rho}_X(s\psi'(v), \psi'(sv))+\tilde{\rho}_X(s\phi(v), \phi(sv)))^2\\
&= \sum_{v\in [d]}(\tilde{\rho}_X(s\psi'(v), \psi'(sv))+\tilde{\rho}_G(s\phi(v), \phi(sv)))^2\\
&\le 2\sum_{v\in [d]}\tilde{\rho}_X(s\psi'(v), \psi'(sv))^2+2\sum_{v\in [d]}\tilde{\rho}_G(s\phi(v), \phi(sv))^2\\
&\le 2d(\delta/2)^2+2d(\delta/2)^2=\delta^2d,
\end{align*}
where in the first inequality we use the fact that $\tilde{\rho}_X$ is translation-invariant.
Thus
$\psi'\phi\in \Map(\tilde{\rho}_X, F, \delta, \sigma)$.

We claim that the set $\{\psi'\phi: \psi'\in \Psi'_X, \phi\in \Phi_G\}$ is $(\tilde{\rho}_{X, \infty}, \varepsilon)$-separated.
Let $\psi'_1, \psi'_2\in \Psi'_X$ and $\phi_1, \phi_2\in \Phi_G$. If $\psi_1\neq \psi_2$, then
\begin{align*}
C'\tilde{\rho}_{X, \infty}(\psi'_1\phi_1, \psi'_2\phi_2)\ge \rho_{Y, \infty}(\pi\circ (\psi'_1\phi_1), \pi\circ (\psi'_2\phi_2))=\rho_{Y, \infty}(\psi_1, \psi_2)>C'\varepsilon.
\end{align*}
If $\psi_1=\psi_2$ and $\phi_1\neq \phi_2$, then $\tilde{\rho}_{G, \infty}(\phi_1, \phi_2)>\varepsilon$, and hence
\begin{align*}
\tilde{\rho}_{X, \infty}(\psi'_1\phi_1, \psi'_2\phi_2)=\tilde{\rho}_{X, \infty}(\phi_1, \phi_2)=\tilde{\rho}_{G, \infty}(\phi_1, \phi_2)>\varepsilon,
\end{align*}
where in the first equality we use the fact that $\tilde{\rho}_X$ is translation-invariant. This proves our claim.

Now we have
$$N_\varepsilon(\Map(\tilde{\rho}_X, F, \delta, \sigma), \tilde{\rho}_{X, \infty})\ge N_{C'\varepsilon}(\Map(\pi, \tilde{\rho}_X, F, \delta/2, \sigma), \rho_{Y, \infty})N_\varepsilon(\Map(\tilde{\rho}_G, F, \delta/2, \sigma), \tilde{\rho}_{G, \infty}).$$
Thus
\begin{align*}
h^\varepsilon_{\Sigma, \omega, \infty}(\tilde{\rho}_X, F, \delta)&\ge h^{C'\varepsilon}_{\Sigma, \omega, \infty}(\rho_Y, F, \delta/2|\tilde{\rho}_X)+h^\varepsilon_{\Sigma, \omega, \infty}( \tilde{\rho}_G, F, \delta/2)\\
&\ge h^{C'\varepsilon}_{\Sigma, \omega, \infty}(\rho_Y|X)+h^\varepsilon_{\Sigma, \omega, \infty}(\tilde{\rho}_G).
\end{align*}
Since $F$ is an arbitrary finite subset of $\Gamma$ and $\delta>0$ is arbitrary, we get
$$h^\varepsilon_{\Sigma, \omega, \infty}(\tilde{\rho}_X)\ge h^{C'\varepsilon}_{\Sigma, \omega, \infty}(\rho_Y|X)+h^\varepsilon_{\Sigma, \omega, \infty}(\tilde{\rho}_G).$$
It follows that
$$\mdim_{\Sigma, \omega, \rM}(X, \tilde{\rho}_X)\ge \mdim_{\Sigma, \omega, \rM}(Y, \rho_Y|X)+\mdim_{\Sigma, \omega, \rM}(G, \tilde{\rho}_G).$$
Now the lemma  follows from \eqref{E-mdim addition lower bound4}.
\end{proof}

Consider the translation-invariant metric $\vartheta$ on $\Rb/\Zb$ defined by
\begin{align*}
\vartheta(x+\Zb, y+\Zb):=\min_{z\in \Zb}|x-y-z|.
\end{align*}
For any $\Zb\Gamma$-module $\cM$ and any finite subset $A$ of $\cM$, we define a translation-invariant continuous pseudometric $\vartheta^A$ on $\widehat{\cM}$ by
\begin{align*}
\vartheta^A(x, y):=\max_{a\in A}\vartheta(x(a), y(a)).
\end{align*}
Note that $A$ generates $\cM$ as a $\Zb\Gamma$-module if and only if $\vartheta^A$ is dynamically generating. Now we specialize Lemma~\ref{L-mdim addition lower bound} to algebraic actions.

\begin{lemma} \label{L-mdim addition lower bound algebraic}
Let $\cM_1\subseteq \cM_2$ be finitely generated $\Zb\Gamma$-modules. Let $A_1$ and $A_2$ be finite generating subsets of $\cM_1$ and $\cM_2$ respectively. Denote by $\overline{A_2}$ the image of $A_2$ under the quotient map $\cM_2\rightarrow \cM_2/\cM_1$.
Then
\begin{align*}
\mdim_{\Sigma, \omega, \rM}(\widehat{\cM_2}, \vartheta^{A_2})\ge \mdim_{\Sigma, \omega, \rM}(\widehat{\cM_1}, \vartheta^{A_1}|\widehat{\cM_2})+\mdim_{\Sigma, \omega, \rM}(\widehat{\cM_2/\cM_1}, \vartheta^{\overline{A_2}}).
\end{align*}
\end{lemma}
\begin{proof} Note that $\vartheta^{\overline{A_2}}$ is the restriction of $\vartheta^{A_2}$ to $\widehat{\cM_2/\cM_1}$.
Denote by $\pi$ the factor map $\widehat{\cM_2}\rightarrow \widehat{\cM_1}$.
Let $a\in A_1$. Then $a=\sum_{b\in A_2}f_bb$ for some $f_b\in \Zb\Gamma$. Write
each $f_b$ as $\sum_{s\in K} f_{b, s}s$ with some $K\in \cF(\Gamma)$ (independent of $a$ and $b$) and $f_{b, s}\in \Zb$.
Set $C_a=\sum_{b\in A_2}\sum_{s\in K}|f_{b, s}|$. For any $x, y\in \widehat{\cM_2}$,
we have
\begin{align*}
\vartheta(x(a), y(a))&=\vartheta(\sum_{b\in A_2}x(f_bb), \sum_{b\in A_2}y(f_bb))\\
&=\vartheta(\sum_{b\in A_2}\sum_{s\in K}f_{b, s}(s^{-1}x)(b), \sum_{b\in A_2}\sum_{s\in K}f_{b, s}(s^{-1}y)(b))\\
&\le \sum_{b\in A_2}\sum_{s\in K}|f_{b, s}|\vartheta((s^{-1}x)(b), (s^{-1}y)(b))\\
&\le \sum_{b\in A_2}\sum_{s\in K}|f_{b, s}|\max_{t\in K^{-1}}\vartheta^{A_2}(tx, ty)\\
&=C_a \max_{t\in K^{-1}}\vartheta^{A_2}(tx, ty).
\end{align*}
Set $C=\max_{a\in A_1}C_a$.
Then
$$\vartheta^{A_1}(\pi(x), \pi(y))\le C\max_{t\in K^{-1}}\vartheta^{A_2}(tx, ty)$$
for all $x, y\in \widehat{\cM_2}$. Therefore the lemma follows from Lemma~\ref{L-mdim addition lower bound}.
\end{proof}

The following lemma is a special case of \cite[Lemma 7.3]{Li}.

\begin{lemma} \label{L-mdim vs mrank free module}
Let $\cM=(\Zb\Gamma)^n$ for some $n\in \Nb$. Take $A\in \cF(\cM)$ to be the standard basis of $\cM$ as a $\Zb\Gamma$-module.
Then
$$\mdim_{\Sigma, \omega, \rM}(\widehat{\cM}, \vartheta^A)=n.$$
\end{lemma}

In the next several lemmas, we prove the equality between the metric mean dimension and  the mean rank for modules in increasing generality.

\begin{lemma} \label{L-mdim vs mrank fp}
Let $\cM$ be a  finitely presented $\Zb\Gamma$-module, and let $A$ be a finite generating subset of $\cM$. Then
$ \mdim_{\Sigma, \omega, \rM}(\widehat{\cM}, \vartheta^A)=\mrk_{\Sigma, \omega}(\cM)$.
\end{lemma}
\begin{proof} We may write $\cM$ as $(\Zb\Gamma)^n/\cM_1$ for some $n\in \Nb$ and some $\Zb\Gamma$-submodule $\cM_1$ of $(\Zb\Gamma)^n$ such that, denoting by $A_2$ the standard basis of $(\Zb\Gamma)^n$ as a $\Zb\Gamma$-module,  $A$ is the image of $A_2$ under the quotient map $(\Zb\Gamma)^n\rightarrow \cM$. Set $\cM_2=(\Zb\Gamma)^n$. Since $\cM$ is finitely presented, $\cM_1$ is finitely generated \cite[Proposition 4.26]{Lam}. Take a finite generating subset $A_1$ of $\cM_1$.

By Lemma~\ref{L-mdim vs mrank free module}
and  Proposition~\ref{P-mean rank basic} we have
\begin{align*}
\mdim_{\Sigma, \omega, \rM}(\widehat{\cM_2}, \vartheta^{A_2})=n=\mrk_{\Sigma, \omega}(\cM_2).
\end{align*}
By Lemma~\ref{L-mdim addition lower bound algebraic}, Theorem~\ref{T-comparison}, Proposition~\ref{P-compare with absolute},  Lemma~\ref{L-mdim vs mrank lower}, and Theorem~\ref{T-addition for mean length}
we have
\begin{align*}
\mdim_{\Sigma, \omega, \rM}(\widehat{\cM_2}, \vartheta^{A_2})
&\ge \mdim_{\Sigma, \omega, \rM}(\widehat{\cM_1}, \vartheta^{A_1}|\widehat{\cM_2})+\mdim_{\Sigma, \omega, \rM}(\widehat{\cM}, \vartheta^{A})\\
&\ge \mdim_{\Sigma, \omega}(\widehat{\cM_1}|\widehat{\cM_2})+\mdim_{\Sigma, \omega}(\widehat{\cM})\\
&\ge \underline{\mdim}_{\Sigma, \omega}(\widehat{\cM_1}|\widehat{\cM_2})+\mdim_{\Sigma, \omega}(\widehat{\cM})\\
&\ge \mrk_{\Sigma, \omega}(\cM_1|\cM_2)+\mrk_{\Sigma, \omega}(\cM)\\
&= \mrk_{\Sigma, \omega}(\cM_2)\\
&=\mdim_{\Sigma, \omega, \rM}(\widehat{\cM_2}, \vartheta^{A_2}).
\end{align*}
Thus $\mdim_{\Sigma, \omega, \rM}(\widehat{\cM}, \vartheta^A)=\mrk_{\Sigma, \omega}(\cM)$.
\end{proof}

\begin{lemma} \label{L-mdim vs mrank fg}
Let $\cM$ be a  finitely generated $\Zb\Gamma$-module, and let $A$ be a finite generating subset of $\cM$. Then
$ \mdim_{\Sigma, \omega, \rM}(\widehat{\cM}, \vartheta^A)=\mrk_{\Sigma, \omega}(\cM)$.
\end{lemma}
\begin{proof} We may write $\cM$ as $\cM'/\cM_\infty$ for some finitely generated free $\Zb\Gamma$-module $\cM'$  and some $\Zb\Gamma$-submodule $\cM_\infty$ of $\cM'$ such that,
denoting by $A'$ the standard basis of $\cM'$ as a $\Zb\Gamma$-module,  $A$ is the image of $A'$ under the quotient map $\cM'\rightarrow \cM$.
Take an increasing sequence $\{\cM_j\}_{j=1}^\infty$ of finitely generated $\Zb\Gamma$-submodules of $\cM_\infty$ with $\bigcup_{j=1}^\infty\cM_j=\cM_\infty$.
Denote by $A_j$ the image of $A'$ under the quotient map $\cM'\rightarrow \cM'/\cM_j$. For each $j\in \Nb$, $\cM'/\cM_j$ is a finitely presented $\Zb\Gamma$-module. Thus by Lemma~\ref{L-mdim vs mrank fp} we have
$$\mdim_{\Sigma, \omega, \rM}(\widehat{\cM'/\cM_j}, \vartheta^{A_j})=\mrk_{\Sigma, \omega}(\cM'/\cM_j).$$
Note that $\vartheta^A$ is the restriction of $\vartheta^{A_j}$ to $\widehat{\cM}$, thus $\mdim_{\Sigma, \omega, \rM}(\widehat{\cM}, \vartheta^A)\le \mdim_{\Sigma, \omega, \rM}(\widehat{\cM'/\cM_j}, \vartheta^{A_j})$.
Therefore by Proposition~\ref{P-continuity} we have
\begin{align*}
\mdim_{\Sigma, \omega, \rM}(\widehat{\cM}, \vartheta^A)&\le \lim_{j\to \infty}\mdim_{\Sigma, \omega, \rM}(\widehat{\cM'/\cM_j}, \vartheta^{A_j})\\
&=\lim_{j\to \infty}\mrk_{\Sigma, \omega}(\cM'/\cM_j)\\
&=\mrk_{\Sigma, \omega}(\cM).
\end{align*}
By Theorem~\ref{T-comparison} and Lemma~\ref{L-mdim vs mrank lower} we also have
$$\mdim_{\Sigma, \omega, \rM}(\widehat{\cM}, \vartheta^A)\ge  \mdim_{\Sigma, \omega}(\widehat{\cM})\ge \mrk_{\Sigma, \omega}(\cM).$$
Thus $\mdim_{\Sigma, \omega, \rM}(\widehat{\cM}, \vartheta^A)=\mrk_{\Sigma, \omega}(\cM)$.
\end{proof}

\begin{lemma} \label{L-mdim vs mrank fg fg}
Let $\cM_1\subseteq \cM_2$ be finitely generated $\Zb\Gamma$-modules, and let $A$ be a finite generating subset of $\cM_1$. Then $\mdim_{\Sigma, \omega, \rM}(\widehat{\cM_1}, \vartheta^A|\widehat{\cM_2})=\mrk_{\Sigma, \omega}(\cM_1|\cM_2)$.
\end{lemma}
\begin{proof} Take a finite generating subset $A_2$ of $\cM_2$. Denote by $\overline{A_2}$ the image of $A_2$ under the quotient map $\cM_2\rightarrow \cM_2/\cM_1$.
By Lemma~\ref{L-mdim addition lower bound algebraic}, Theorem~\ref{T-comparison},  Proposition~\ref{P-compare with absolute}, Lemma~\ref{L-mdim vs mrank lower}, Theorem~\ref{T-addition for mean length}, and Lemma~\ref{L-mdim vs mrank fg}
we have
\begin{align*}
\mdim_{\Sigma, \omega, \rM}(\widehat{\cM_2}, \vartheta^{A_2})
&\ge \mdim_{\Sigma, \omega, \rM}(\widehat{\cM_1}, \vartheta^A|\widehat{\cM_2})+\mdim_{\Sigma, \omega, \rM}(\widehat{\cM_2/\cM_1}, \vartheta^{\overline{A_2}})\\
&\ge \mdim_{\Sigma, \omega}(\widehat{\cM_1}|\widehat{\cM_2})+\mdim_{\Sigma, \omega}(\widehat{\cM_2/\cM_1})\\
&\ge \underline{\mdim}_{\Sigma, \omega}(\widehat{\cM_1}|\widehat{\cM_2})+\mdim_{\Sigma, \omega}(\widehat{\cM_2/\cM_1})\\
&\ge \mrk_{\Sigma, \omega}(\cM_1|\cM_2)+\mrk_{\Sigma, \omega}(\cM_2/\cM_1)\\
&= \mrk_{\Sigma, \omega}(\cM_2)\\
&=\mdim_{\Sigma, \omega, \rM}(\widehat{\cM_2}, \vartheta^{A_2}).
\end{align*}
Thus
$$\mdim_{\Sigma, \omega, \rM}(\widehat{\cM_1}, \vartheta^A|\widehat{\cM_2})
=\mrk_{\Sigma, \omega}(\cM_1|\cM_2).$$
\end{proof}

\begin{lemma} \label{L-mdim vs mrank fg2}
Let $\cM_2$ be a countable $\Zb\Gamma$-module, and let $\cM_1$ be a finitely generated $\Zb\Gamma$-submodule of $\cM_2$ with  a finite generating set $A$. Then
$$\mdim_{\Sigma, \omega, \rM}(\widehat{\cM_1}, \vartheta^A|\widehat{\cM_2})= \mrk_{\Sigma, \omega}(\cM_1|\cM_2).$$
\end{lemma}
\begin{proof}Take an increasing sequence $\{\cM'_j\}_{j=1}^\infty$ of finitely generated $\Zb\Gamma$-submodules of $\cM_2$ containing $\cM_1$ such that $\bigcup_{j=1}^\infty\cM'_j=\cM_2$. For each $j\in \Nb$, we have
$$\mdim_{\Sigma, \omega, \rM}(\widehat{\cM_1}, \vartheta^A|\widehat{\cM_2})\le \mdim_{\Sigma, \omega, \rM}(\widehat{\cM_1}, \vartheta^A|\widehat{\cM'_j}).$$
Thus by Lemma~\ref{L-mdim vs mrank fg fg} and Proposition~\ref{P-continuity} we get
\begin{align*}
\mdim_{\Sigma, \omega, \rM}(\widehat{\cM_1}, \vartheta^A|\widehat{\cM_2})&\le \lim_{j\to \infty}\mdim_{\Sigma, \omega, \rM}(\widehat{\cM_1}, \vartheta^A|\widehat{\cM'_j})\\
&=\lim_{j\to \infty}\mrk_{\Sigma, \omega}(\cM_1|\cM'_j)\\
&=\mrk_{\Sigma, \omega}(\cM_1|\cM_2).
\end{align*}
By Theorem~\ref{T-comparison}, Proposition~\ref{P-compare with absolute}, and Lemma~\ref{L-mdim vs mrank lower}, we also have
$$\mdim_{\Sigma, \omega, \rM}(\widehat{\cM_1}, \vartheta^A|\widehat{\cM_2})\ge \mdim_{\Sigma, \omega}(\widehat{\cM_1}|\widehat{\cM_2})\ge \underline{\mdim}_{\Sigma, \omega}(\widehat{\cM_1}|\widehat{\cM_2})\ge\mrk_{\Sigma, \omega}(\cM_1|\cM_2).$$
Thus $\mdim_{\Sigma, \omega, \rM}(\widehat{\cM_1}, \vartheta^A|\widehat{\cM_2})= \mrk_{\Sigma, \omega}(\cM_1|\cM_2)$.
\end{proof}

We are ready to prove Theorem~\ref{T-mdim vs mrk}.

\begin{proof}[Proof of Theorem~\ref{T-mdim vs mrk}]
Take an increasing sequence $\{\cM'_j\}_{j\in \Nb}$ of finitely generated $\Zb\Gamma$-submodules of $\cM_1$ with $\bigcup_{j=1}^\infty\cM'_j=\cM_1$.
Take a finite generating subset $A_j$ of $\cM'_j$ for each $j\in \Nb$ such that $A_j\subseteq A_{j+1}$ for all $j\in \Nb$. Denote by $p_j$ the factor map
$\widehat{\cM'_{j+1}}\rightarrow \widehat{\cM'_j}$. Then $\vartheta^{A_j}(p_j(x), p_j(y))\le \vartheta^{A_{j+1}}(x, y)$ for all $x, y\in \widehat{\cM'_{j+1}}$.
By Lemma~\ref{L-metric inverse limit} there is a decreasing sequence $\{\lambda_j\}_{j\in \Nb}$ of positive numbers such that
the dynamically generating continuous pseudometric $\rho$ of $\varprojlim_{j\to \infty}\widehat{\cM'_j}=\widehat{\cM_1}$ defined by
$$ \rho(x, y):=\max_{j\in \Nb}\lambda_j\vartheta^{A_j}(x, y)$$
satisfies
$$\mdim_{\Sigma, \omega, \rM}(\widehat{\cM_1}, \rho|\widehat{\cM_2})= \lim_{j\to \infty}\mdim_{\Sigma, \omega, \rM}(\widehat{\cM'_j}, \vartheta^{A_j}|\widehat{\cM_2}).$$
By Lemma~\ref{L-mdim vs mrank fg2} and Proposition~\ref{P-continuity} we have
\begin{align*}
\mdim_{\Sigma, \omega, \rM}(\widehat{\cM_1}, \rho|\widehat{\cM_2})&= \lim_{j\to \infty}\mdim_{\Sigma, \omega, \rM}(\widehat{\cM'_j}, \vartheta^{A_j}|\widehat{\cM_2})\\
&=\lim_{j\to \infty}\mrk_{\Sigma, \omega}(\cM'_j|\cM_2)\\
&=\mrk_{\Sigma, \omega}(\cM_1|\cM_2).
\end{align*}
By Theorem~\ref{T-comparison}, Proposition~\ref{P-compare with absolute}, and Lemma~\ref{L-mdim vs mrank lower}, we also have
$$\mdim_{\Sigma, \omega, \rM}(\widehat{\cM_1}, \rho|\widehat{\cM_2})\ge \mdim_{\Sigma, \omega}(\widehat{\cM_1}|\widehat{\cM_2})\ge \underline{\mdim}_{\Sigma, \omega}(\widehat{\cM_1}|\widehat{\cM_2})\ge \mrk_{\Sigma, \omega}(\cM_1|\cM_2).$$
Thus
\begin{align*}
\mdim_{\Sigma, \omega, \rM}(\widehat{\cM_1}, \rho|\widehat{\cM_2})&=\mdim_{\Sigma, \omega}(\widehat{\cM_1}|\widehat{\cM_2})=\underline{\mdim}_{\Sigma, \omega}(\widehat{\cM_1}|\widehat{\cM_2})= \mrk_{\Sigma, \omega}(\cM_1|\cM_2).
\end{align*}
Fix an enumeration $s_1, s_2, \dots$ of the elements of $\Gamma$. Define
$$\tilde{\rho}(x, y)=\sum_{j=1}^\infty\frac{1}{2^j}\rho(s_jx, s_jx)$$
for all $x, y\in \widehat{\cM_1}$. By Lemma~\ref{L-pseudometric to metric}, $\tilde{\rho}$ is a translation-invariant compatible metric on $\widehat{\cM_1}$, and
\begin{align*}
\mdim_{\Sigma, \omega, \rM}(\widehat{\cM_1}, \tilde{\rho}|\widehat{\cM_2})=\mdim_{\Sigma, \omega, \rM}(\widehat{\cM_1}, \rho|\widehat{\cM_2}).
\end{align*}
\end{proof}

\section{Applications to mean dimension} \label{S-application}

In this section we give three applications of Theorems~\ref{T-mrk vs vrk}, \ref{T-mdim vs mrk} and \ref{T-mdim vs vrk}  to the mean dimension of algebraic actions.

Our first application is the following addition formula for the mean topological dimension of algebraic actions, which follows  from Theorems~\ref{T-addition for mean length} and \ref{T-mdim vs mrk}.

\begin{theorem}\label{T-addition for mdim}
Let $\Gamma$ be a countable sofic group, and let
$$0\rightarrow X_1\rightarrow X_2\rightarrow X_3\rightarrow 0$$
be a short exact sequence of compact metrizable abelian groups equipped with $\Gamma$-actions by continuous automorphisms such that the maps are $\Gamma$-equivariant. Then
$$ \mdim_{\Sigma, \omega}(X_2)=\mdim_{\Sigma, \omega}(X_3|X_2)+\mdim_{\Sigma, \omega}(X_1).$$
\end{theorem}

Our second application concerns the possible values of the mean topological dimension for algebraic actions. It is known that for any countable amenable group with subgroups of arbitrarily large finite index, the mean topological dimension of its continuous actions on compact metrizable spaces can take all values in $\Rb_{\ge 0}\cup \{+\infty\}$ \cite{LW, CK}. We shall show that this is not the case when one restricts to algebraic actions.

For any discrete group $\Gamma$, denote by $\cH(\Gamma)$ the subgroup of $\Qb$ generated by $|H|^{-1}$ for $H$ ranging over all finite subgroups
of $\Gamma$. The class of {\it elementary amenable groups} is the smallest class of groups containing all finite groups and abelian groups and being closed
under taking subgroups, quotient groups, group extensions and directed unions \cite{Chou}.

\begin{notation} \label{N-class C}
We denote by $\fC$ the smallest class of groups containing all free groups and being closed under directed unions and extensions with elementary amenable quotients.
\end{notation}

The following result is due to Linnell (see \cite[Theorem 1.5]{Linnell} and also \cite[Theorem 10.19]{Luck02}).

\begin{theorem} \label{T-Linnell}
If a group $\Gamma$ belongs to $\fC$ and there is an upper bound on the orders of the finite subgroups of $\Gamma$, then the strong Atiyah conjecture holds for $\Gamma$, i.e. for any $m, n\in \Nb$ and any
$f\in M_{m, n}(\Cb\Gamma)$, denoting by $P_f$ the orthogonal projection from $(\ell^2(\Gamma))^{n\times 1}$ to $\ker f$, one has $\tr_{\cL\Gamma}P_f\in \cH(\Gamma)$.
\end{theorem}

The class of sofic groups is closed under directed unions, free products and extensions with amenable quotients \cite[Theorem 1]{ES06}. Thus every group in the class $\fC$ is sofic.

\begin{lemma} \label{L-range of vrk}
Let $\Gamma$ be a group in $\fC$ with an upper bound on the orders of the finite subgroups of $\Gamma$. Let $R$ be a unital subring of $\Cb$. For any $R\Gamma$-modules $\cM_1\subseteq \cM_2$, one has
$\vrk(\cM_1|\cM_2)\in \cH(\Gamma)\cup\{+\infty\}$.
\end{lemma}
\begin{proof} Let $\cM$ be a finitely presented $R\Gamma$-module. Then $\cM$ is of the form $(R\Gamma)^{1\times n}/(R\Gamma)^{1\times m}f$ for some $m,n\in \Nb$ and some $f\in M_{m, n}(R\Gamma)$.
By Lemma~\ref{L-vrk for fp} and
Theorem~\ref{T-Linnell} we have $\vrk(\cM)=\tr_{\cL\Gamma}P_f \in \cH(\Gamma)$.

By Lemma~\ref{L-vrk basic} one can compute $\vrk(\cM_1|\cM_2)$ for all $R\Gamma$-modules $\cM_1\subseteq \cM_2$ using the algorithm in the proof of Lemma~\ref{L-unique}.
Note that $\cH(\Gamma)$ is a discrete subgroup of $\Rb$. It follows that $\vrk(\cM_1|\cM_2)\in \cH(\Gamma)\cup\{+\infty\}$.
\end{proof}

Combining Lemma~\ref{L-range of vrk} and Theorems~\ref{T-mdim vs vrk} and \ref{T-mrk vs vrk}, we get

\begin{corollary} \label{C-range of mdim}
Let $\Gamma$ be a countable group in the class $\fC$ with an upper bound on the orders of the finite subgroups of $\Gamma$. For any countable $\Zb\Gamma$-modules $\cM_1\subseteq \cM_2$, one has
$$\mdim_{\Sigma, \omega}(\widehat{\cM_1}|\widehat{\cM_2}), \mrk_{\Sigma, \omega}(\cM_1|\cM_2)\in \cH(\Gamma)\cup \{+\infty\}.$$
\end{corollary}

Our last application concerns the relation between zero mean topological dimension actions and finite entropy actions.
Lindenstrauss showed that when $\Gamma$ is amenable, inverse limits of finite entropy actions must have zero mean topological dimension \cite[Proposition 6.11]{Lindenstrauss}. This is also true for sofic groups.

\begin{proposition} \label{P-finite entropy to zero mdim}
Let $\Gamma$ be a countable sofic group. Let $\{X_j\}_{j\in \Nb}$ be a sequence of compact metrizable spaces carrying continuous $\Gamma$-actions and factor maps $X_{j+1}\rightarrow X_j$ for all $j\in \Nb$. Set  $X=\varprojlim_{j\to \infty}X_j$. Suppose that $h_{\Sigma, \omega}(X_j)<+\infty$ for every $j\in \Nb$. Then $\mdim_{\Sigma, \omega}(X)=0$ or $-\infty$.
If furthermore $0\le h_{\Sigma, \omega}(X_j)<+\infty$ for every $j\in \Nb$, then $\mdim_{\Sigma, \omega}(X)=0$.
\end{proposition}
\begin{proof} Let $j\in \Nb$.
By Proposition~\ref{P-metric compare with absolute} we have $h_{\Sigma, \omega}(X_j|X)\le h_{\Sigma, \omega}(X_j)<+\infty$. From Definition~\ref{D-relative metric mdim} it is clear that  $\mdim_{\Sigma, \omega, \rM}(X_j, \rho_j|X)\le 0$ for every compatible metric $\rho_j$ of $X_j$. By Theorem~\ref{T-comparison} we get
$\mdim_{\Sigma, \omega}(X_j|X)\le \mdim_{\Sigma, \omega, \rM}(X_j|X)\le 0$.

From Proposition~\ref{P-topology inverse limit} we conclude that $\mdim_{\Sigma, \omega}(X)\le \varlimsup_{j\to \infty}\mdim_{\Sigma, \omega}(X_j|X)\le 0$.

Now assume that $0\le h_{\Sigma, \omega}(X_j)<+\infty$ for every $j\in \Nb$. Take a compatible metric $\rho'_j$ for each $X_j$ with $\diam(X_j, \rho'_j)\le 1$.
Denote by $p_{j, k}$ and $\pi_j$ the factor maps $X_k\rightarrow X_j$ for $j\le k$ and  $X\rightarrow X_j$ respectively. Define a new compatible metric $\rho_k$ on $X_k$ by
$$ \rho_k(x, y):=\max_{1\le j\le k}\rho'_j(p_{j, k}(x), p_{j, k}(y)).$$
Then $\diam(X_k, \rho_k)\le 1$, and $\rho_j(p_{j, k}(x), p_{j, k}(y))\le \rho_k(x, y)$ for all $j\le k$ and $x, y\in X_k$.
Now we define a compatible metric $\rho$ for $X$ by
$$\rho(x, y):=\sum_{j=1}^\infty \frac{1}{2^j}\rho_j(\pi_j(x), \pi_j(y)).$$

Let $F\in \cF(\Gamma)$ and $\delta>0$. We can find some $j\in \Nb$ such that $\rho(x, y)\le \rho_j(\pi_j(x), \pi_j(y))+\delta$ for all $x, y\in X$.
It is easy to check that for any map $\sigma: \Gamma\rightarrow \Sym(d)$, every $\varphi\in X^d$ with $\pi_j\circ \varphi\in \Map(\rho_j, F, \delta, \sigma)$ belongs to $\Map(\rho, F, 2\delta, \sigma)$.
Since $0\le h_{\Sigma, \omega}(X_j)$, the set of $i\in \Nb$ with $\Map(\rho_j, F, \delta, \sigma_i)\neq \emptyset$ is in $\omega$. Thus the set of $i\in \Nb$ with $\Map(\rho, F, 2\delta, \sigma_i)\neq \emptyset$ is in $\omega$. Therefore $\mdim_{\Sigma, \omega}(\cU, \rho, F, 2\delta)\ge 0$ for every finite open cover $\cU$ of $X$. Consequently, $\mdim_{\Sigma, \omega}(X)\ge 0$.

\end{proof}

 Now we turn our attention to algebraic actions.
Note that the algebraic actions of countable sofic groups have fixed points, thus their  topological entropy and  mean topological dimension are always nonnegative.

We consider first finitely presented modules.

\begin{corollary} \label{C-entropy fp}
Let $\Gamma$ be a countable sofic group and $\cM$ be a finitely presented $\Zb\Gamma$-module. Then $\mdim_{\Sigma, \omega}(\widehat{\cM})=0$ if and only if
$h_{\Sigma, \omega}(\widehat{\cM})<+\infty$.
\end{corollary}
\begin{proof} By Proposition~\ref{P-finite entropy to zero mdim} we just need to show the ``only if'' part. We have $\cM=(\Zb\Gamma)^{1\times n}/(\Zb\Gamma)^{1\times m}f$ for some $m,n\in \Nb$ and $f\in M_{m, n}(\Zb\Gamma)$.
By Theorem~\ref{T-mdim vs vrk} if $\mdim_{\Sigma, \omega}(\widehat{\cM})=0$, then $\vrk(\widehat{\cM})=0$, and hence by Lemma~\ref{L-vrk for fp} one has $\tr_{\cL\Gamma}P_f=0$.
Since $\tr_{\cL\Gamma}$ is  faithful, we get $P_f=0$, i.e. $f$ is injective on $(\ell^2(\Gamma))^{n\times 1}$. By \cite[Theorem 1.1]{Hayes15}, one concludes that $h_{\Sigma, \omega}(\widehat{\cM})<+\infty$.
\end{proof}

Next we consider finitely generated modules. The following corollary was proved in \cite[Corollary 9.5]{LL} under the further assumption that $\Gamma$ is elementary amenable. Using
Theorem~\ref{T-mdim vs mrk}, Proposition~\ref{P-continuity}, and Corollaries~\ref{C-range of mdim} and \ref{C-entropy fp} one can carry out the proof of \cite[Corollary 9.5]{LL}  in general case. Recall the class $\fC$ of groups in Notation~\ref{N-class C}.

\begin{corollary} \label{C-entropy fg}
Let $\Gamma$ be a countable sofic group in the class $\fC$ with an upper bound on the orders of the finite subgroups of $\Gamma$, and $\cM$ be a finitely generated $\Zb\Gamma$-module. Then $\mdim_{\Sigma, \omega}(\widehat{\cM})=0$ if and only if
$h_{\Sigma, \omega}(\widehat{\cM})<+\infty$.
\end{corollary}


\end{document}